\documentclass{article}
\pdfoutput=1

\usepackage{etoolbox}

\newtoggle{colt}
\togglefalse{colt}

\newcommand{\colttogfalse}[1]{\iftoggle{colt}{}{#1}}

\usepackage{boxedminipage}
\usepackage{multirow,nicefrac}
\usepackage{makecell,upgreek}
\usepackage{footnote}
\usepackage{longtable}
\usepackage{tablefootnote}
\usepackage[T1]{fontenc}
\usepackage{verbatim} 
\usepackage[utf8]{inputenc}
\usepackage{xcolor}
\usepackage[shortlabels]{enumitem}

\colttogfalse{
\usepackage[breaklinks=true]{hyperref}
\usepackage{amsthm}
\usepackage{fullpage}
}

\usepackage{amsfonts,amssymb,mathrsfs}
\usepackage{mathtools}
\DeclareMathSymbol{\shortminus}{\mathbin}{AMSa}{"39}

\usepackage{natbib}

\usepackage{prettyref}

\usepackage[capitalise,nameinlink]{cleveref}
\Crefname{equation}{Eq.}{Eqs.}
\Crefname{assumption}{Assumption}{Assumptions}
\Crefname{condition}{Condition}{Conditions}

\usepackage{breakcites,bm}
\hypersetup{colorlinks,citecolor=black,linkcolor=blue,linktocpage=true}
\usepackage{latexsym}
\usepackage{relsize}
\usepackage{thm-restate}
\usepackage{appendix}

\usepackage{xcolor}
\usepackage{dsfont}
\usepackage{url}
\makeatletter
\g@addto@macro{\UrlBreaks}{\UrlOrds}
\makeatother

\newcommand{\defeq}{\stackrel{\text{def}}{=}}

\numberwithin{equation}{section}

\newcommand{\dmax}{d_{\max}}

\newcommand{\B}{B}

\DeclareFontFamily{U}{mathx}{\hyphenchar\font45}
\DeclareFontShape{U}{mathx}{m}{n}{
      <5> <6> <7> <8> <9> <10>
      <10.95> <12> <14.4> <17.28> <20.74> <24.88>
      mathx10
      }{}
\DeclareSymbolFont{mathx}{U}{mathx}{m}{n}
\DeclareFontSubstitution{U}{mathx}{m}{n}
\DeclareMathAccent{\widecheck}{0}{mathx}{"71}
\DeclareMathAccent{\wideparen}{0}{mathx}{"75}

\newcommand{\ignore}[1]{}

\newcommand{\Mst}{M_{\star}}

\newcommand{\R}{\mathbb{R}}

\newcommand{\calA}{\mathcal{A}}

\newcommand{\opnorm}[1]{
\left\|#1\right\|_{\op}}
\newcommand{\opnormil}[1]{
\|#1\|_{\op}}

\newcommand{\N}{\mathbb{N}}

\newcommand{\fhat}{\widehat{f}}

\newcommand{\Alg}{\mathsf{alg}}
\newcommand{\calN}{\mathcal{N}}

\newcommand{\matx}{\mathbf{x}}
\newcommand{\matm}{\mathbf{m}}
\newcommand{\matu}{\mathbf{u}}
\newcommand{\matw}{\mathbf{w}}
\newcommand{\maty}{\mathbf{y}}
\newcommand{\matq}{\mathbf{q}}

\newcommand{\iidsim}{\overset{\mathrm{i.i.d}}{\sim}}
\newcommand{\Regret}{\mathrm{Regret}}

\newcommand{\dimx}{d_{x}}
\newcommand{\dimu}{d_{u}}

\newcommand{\Bst}{B_{\star}}
\newcommand{\Ast}{A_{\star}}

\newcommand{\Pst}{P_{\star}}
\newcommand{\Kst}{K_{\star}}

\newcommand{\I}{\mathbb{I}}
\newcommand{\mate}{\mathbf{e}}

\newcommand{\matz}{\mathbf{z}}

\newcommand{\op}{\mathrm{op}}

\newcommand{\calF}{\mathcal{F}}

	\colttogfalse{
	\newtheorem{theorem}{Theorem}[section]

	\newtheorem{lemma}{Lemma}[section]
	
	\newtheorem{corollary}{Corollary}[section]
	\newtheorem{proposition}[lemma]{Proposition}

	\theoremstyle{definition}

	\newtheorem{definition}{Definition}[section]

	\newtheorem{remark}{Remark}[section]

  }

  	\newtheorem{claim}{Claim}[section]
 \newtheorem{assumption}{Assumption}
  \newtheorem{condition}{Condition}[section]

\makeatletter
\newcommand{\neutralize}[1]{\expandafter\let\csname c@#1\endcsname\count@}
\makeatother

\newtheorem*{theorem*}{Theorem}
\newtheorem*{lemma*}{Lemma}
\newtheorem*{corollary*}{Corollary}
\newtheorem*{proposition*}{Proposition}
\newtheorem*{claim*}{Claim}
\newtheorem*{fact*}{Fact}
\newtheorem*{observation*}{Observation}

\newtheorem*{definition*}{Definition}
\newtheorem*{remark*}{Remark}
\newtheorem*{example*}{Example}

\DeclareMathAlphabet{\mathbfsf}{\encodingdefault}{\sfdefault}{bx}{n}

\DeclareMathOperator*{\argmin}{arg\,min}

\let\Pr\relax
\DeclareMathOperator{\Pr}{\mathbb{P}}

\newcommand{\norm}[1]{\|#1\|}
\newcommand{\ceil}[1]{\lceil #1 \rceil}

\newcommand{\floor}[1]{\lfloor #1 \rfloor}

\newcommand{\E}{\mathbb{E}}

\newcommand{\trace}{\mathrm{tr}}
\newcommand{\traceb}[1]{\trace \left[ #1 \right]}

\newcommand{\st}{\star}

\newcommand{\eps}{\varepsilon}

\renewcommand{\leq}{~\le~}
\renewcommand{\geq}{~\ge~}

\let\oldtfrac\tfrac
\renewcommand{\tfrac}[2]{\smash{\oldtfrac{#1}{#2}}}

\let\nablaold\nabla
\renewcommand{\nabla}{\nablaold\mkern-2.5mu}

\newcommand{\Exp}{\mathbb{E}}

\newcommand{\Toep}{\mathsf{Toep}}

\newcommand{\calH}{\mathcal{H}}

\newcommand{\diag}{\mathrm{diag}}

\newcommand{\rmd}{\mathrm{d}}

\newcommand{\calC}{\mathcal{C}}

\newcommand{\brackalign}[1]{\left[#1\right]_{\mathrm{algn}}}

\newcommand{\Kinitial}{K_{\mathrm{init}}}
\newcommand{\explore}{\mathsf{explore}}
\newcommand{\commit}{\mathsf{commit}}

\newcommand{\mainalg}{\mathsf{OnlineCE}}

\newcommand{\HS}{\mathsf{HS}}
\newcommand{\matv}{\mathbf{v}}
\newcommand{\calU}{\mathcal{U}}

\newcommand{\fst}{f^{\star}}

\renewcommand{\defeq}{:=}
\newcommand{\covw}{\Sigma_\matw}
\newcommand{\herm}{\mathsf{H}}
\newcommand{\hilspace}{\mathcal{H}_{\matx}}
\newcommand{\hilbx}{\mathcal{H}_{\matx}}

\newcommand{\calV}{\mathcal{V}}
\newcommand{\nn}{\nonumber}
\newcommand{\psdhil}[1][\hilspace]{\mathbb{S}_{+}^{#1}}

\newcommand{\Ahat}{\widehat{A}}
\newcommand{\Bhat}{\widehat{B}}
\newcommand{\Khat}{\widehat{K}}
\newcommand{\Phat}{\widehat{P}}

\newcommand{\Pinfty}[2]{P_\infty( #1, #2 )}
\newcommand{\Kinfty}[2]{K_\infty( #1, #2)}

\newcommand{\Bt}{B(t)}
\newcommand{\At}{A(t)}

\newcommand{\Acl}{A_{\mathrm{cl}}}
\newcommand{\Aclone}{A_{\mathrm{cl},1}}
\newcommand{\Acltwo}{A_{\mathrm{cl},2}}

\newcommand{\Aclhat}{\widehat{A}_{\mathrm{cl}}}
\newcommand{\Aclst}{A_{\mathrm{cl}_\star}}

\newcommand{\Kinit}{K_0}
\newcommand{\dlyap}[2]{\mathsf{dlyap}\left(#1,\; #2\right)}
\newcommand{\dlyapm}[3]{\mathsf{dlyap}_{(#3)}\left(#1,\; #2\right)}

\newcommand{\Sigmaexp}{\Sigma_{0}}
\newcommand{\Linit}{\Sigma_{0}}

\newcommand{\tLambda}{\overline{\Lambda}}

\newcommand{\innerHx}[2]{\left \langle #1,\; #2 \right\rangle_{\hilspace}}
\newcommand{\innerHS}[2]{\left \langle #1,\; #2 \right\rangle_{\mathsf{HS}}}
\newcommand{\expcovone}{\Sigma_{\matx,0}}
\newcommand{\expcov}{\expcovone^{1/2}}

\newcommand{\expcovonehat}{\widehat{\Sigma}_{\matx, 0}}

\newcommand{\uvar}{\sigma^2_{\matu}}

\newcommand{\normHS}[1]{\left\lVert#1 \right\rVert_{\mathsf{HS}}}
\newcommand{\nucnorm}[1]{\left\lVert#1 \right\rVert_{\mathsf{tr}}}
\newcommand{\normhilbx}[1]{\left\lVert#1 \right\rVert_{\hilbx}}
\newcommand{\normHx}[1]{\normhilbx{#1}}
\newcommand{\normHSil}[1]{\|#1 \|_{\mathsf{HS}}}
\newcommand{\nucnormil}[1]{\lVert#1 \|_{\mathsf{tr}}}

\renewcommand{\epsilon}{\varepsilon}

\newcommand{\maxp}{M}

\newcommand{\inner}[2]{\left\langle #1 , #2\right\rangle}

\newcommand{\epsilonshape}{\epsilon_{\mathrm{cov}}}

\newcommand{\PinftyK}[3]{P_\infty(#1 ; #2 , #3)}

\newcommand{\Deltahs}{\Delta_{\mathsf{HS}}}

\newcommand{\logtracenorm}{W_{\mathrm{tr}}}
\newcommand{\projlambda}{S_\lambda}
\newcommand{\projnlambda}{\overline{S}_\lambda}

\newcommand{\Cstable}{\calC_{\mathsf{stable}}}
\newcommand{\Ctail}{\calC_{\mathsf{tail}, \lambda}}

\newcommand{\Covst}[1]{\Sigma_\star(#1)}

\newcommand{\Abar}{\bar{A}}
\newcommand{\Aclbar}{\bar{A}_{\mathrm{cl}}}
\newcommand{\Bbar}{\bar{B}}
\newcommand{\Pbar}{\bar{P}}
\newcommand{\Mbar}{\bar{M}}
\newcommand{\DelbarAcl}{\bar{\Delta}_{\Acl}}

\newcommand{\epsunifop}{\varepsilon_{\mathsf{op}}}
\newcommand{\epsunifhs}{\varepsilon_{\mathsf{HS}}}

\newcommand{\Tname}{T}
\newcommand{\Tinit}{\Tname_{\mathsf{init}}}
\newcommand{\Texplore}{\Tname_{\mathsf{exp}}}

\newcommand{\Kbase}{K_{\mathsf{init}}}
\newcommand{\lambdast}{\lambda_{\mathsf{safe}}}

\newcommand{\Jstar}{\mathcal{J}_\star}
\newcommand{\initcovone}{\Sigma_{\matx, \mathsf{init}}}

\newcommand{\initcovonehat}{\widehat{\Sigma}_{\matx, \mathsf{init}}}
\newcommand{\Pinit}{P_\mathsf{init}}

\newcommand{\warmstart}{\mathsf{WarmStart}}
\newcommand{\epsopnot}{\epsilon_{\mathrm{op},0}}

\newcommand{\Mpst}{M_{P_{\star}}}
\newcommand{\Mpnot}{M_{P_{0}}}
\newcommand{\kappagam}{\kappa_{P}}
\newcommand{\epsgam}{\varepsilon_{P}}
\newcommand{\epssysnum}[1]{\epsilon_{\mathrm{sys},#1}}

\newcommand{\jcpnote}[1]{\textcolor{purple}{\textbf{?? Juanky}: #1}}

\title{Towards a Dimension-Free Understanding of Adaptive Linear Control}
\author{Juan C. Perdomo \\\texttt{jcperdomo@berkeley.edu}\\ University of California, Berkeley \and Max Simchowitz\\\texttt{msimchow@berkeley.edu}\\ University of California, Berkeley \and Alekh Agarwal  \\\texttt{alekha@microsoft.com} \\ Microsoft Research \and Peter Bartlett \\\texttt{peter@berkeley.edu} \\  University of California, Berkeley}

\begin{document}
\maketitle

\begin{abstract}%
We study the problem of adaptive control of the linear quadratic regulator for systems in very high, or even infinite dimension. We demonstrate that while sublinear regret requires finite dimensional inputs,  the ambient state dimension of the system need not be bounded in order to perform online control. We provide the first regret bounds for LQR which hold for infinite dimensional systems, replacing dependence on ambient dimension with more natural notions of problem complexity. Our guarantees arise from a novel perturbation bound for certainty equivalence which scales with the \emph{prediction error} in estimating the system parameters, without requiring consistent parameter recovery in more stringent measures like the operator norm. When specialized to finite dimensional settings, our bounds recover near optimal dimension and time horizon dependence.

%
%
\end{abstract}


\setcounter{tocdepth}{0}

\newcommand{\Sigmaw}{\Sigma_{\matw}}
\newcommand{\thetahat}{\widehat{\theta}}
\newcommand{\thetast}{{\theta_\star}}
\newcommand{\Jfunc}[1]{\mathcal{J}(#1)}
\newcommand{\rkhslqr}{\textsf{RKHS}\text{-}\textsf{LQR}}
\newcommand{\Jst}{\mathcal{J}_{\star}}
\newcommand{\matf}{\mathbf{f}}
\newcommand{\calO}{\mathcal{O}}
\section{Introduction}
Reinforcement learning (RL) has matured considerably in recent years, setting its sights on increasingly ambitious tasks in ever more complex environments. With this increased complexity, it is neither possible
nor desirable to learn models of the environment that are uniformly accurate across all possible states.
In particular, to scale learning methods to complex sensorimotor state
observations, it is critical to focus model estimation on the parts of
the state space that are most relevant to the cost and which can be
influenced by the available control actions.
This paper investigates the possibility of meeting this challenge in
high-dimensional control tasks.  We study the problem of learning the
optimal \emph{Linear Quadratic Regulator} (LQR) where the states live
in a potentially infinite dimensional \emph{Reproducing Kernel Hilbert
Space} (RKHS), a  linear control problem in which the dynamics, optimal value function, and optimal control policy are infinite dimensional, and therefore cannot be efficiently estimated to uniform precision. We focus on the regret setting, known as \emph{online} LQR, in which
a learner faces an unknown linear dynamical system and must adaptively tune
a control policy to compete with the optimal policy.

Recent work has studied the statistical complexity of finite-dimensional LQRs at length (both in online and batch settings) \citep{dean2019sample,dean2018regret,mania2019certainty,faradonbeh2018optimality,simchowitz2020naive,fazel2018global,tu2019gap}. However, all known results scale explicitly with the ambient dimension of the state space, which is infinite in our setting. In this work, we develop more fine-grained complexity measures to understand the hardness of linear control. While our measures are always crudely bounded by the state dimension, they behave more like an intrinsic dimension that adapts to the problem structure. Hence, they can be well-defined in many infinite dimensional RKHS settings as well,
where they depend most strongly on the decay of the spectrum of the noise covariance. Such complexity measures based on intrinsic dimension are well-understood in supervised learning, and have been extended
recently to the bandit and discrete action RL settings. Extending these ideas to continuous control is more challenging as several aspects of existing theory critically leverage the estimation of system parameters in operator norm, which necessitates an explicit dimension dependence. This motivates the following question:
\begin{quote}
\emph{Is it possible to obtain dimension-free sample complexity and regret
guarantees in continuous control? What is the right measure of complexity?}
\end{quote}

We answer the first question in the affirmative and show that the
\emph{spectral properties of the noise covariance}, such as its trace and eigenvalue decay, provide a much sharper characterization of the problem complexity than the state dimension in many parameter regimes.

\subsection{Problem Setting \& Background}
We consider the problem of adaptive control of the linear quadratic regulator, or \emph{online LQR}. To enable dimension-free results, we assume that the states $\matx_t$ lie in a Hilbert space $\hilbx$. While states $\matx_t$ are potentially infinite-dimensional, as shown in \Cref{thm:main_lb_input}, it is necessary to assume the inputs $\dimu \in \R^{\dimu}$ are finite dimensional in order to guarantee sublinear regret.
Given an arbitrary linear operator  $X: \calH_1 \to \calH_2$ between Hilbert spaces, we let $\|X\|_{\op}$, $\|X\|_{\HS}$, $\nucnorm{X}$ denote its operator, Hilbert-Schmidt (HS), and trace norms. These norms may be infinite in general, and we say $X$ is bounded, Hilbert Schmidt, or trace class if the corresponding norm is finite. We let $X^\herm$  (resp. $\matx^\herm$) denote adjoints of operators (resp. vectors) and $\matx \otimes \matx$ denote outer products. The dynamics evolve according to:
\begin{align}
\label{eq:dynamics_def}
\matx_{t+1} = \Ast \matx_t + \Bst \matu_t + \matw_t, \quad \matw_t \iidsim \calN(0,\Sigmaw),
\end{align}
where $\Ast:\hilbx \to \hilbx$ and $\Bst: \R^{\dimu}\to\hilbx$ are bounded linear operators, and $\Sigmaw$ is \emph{trace class} (i.e. $\traceb{\Sigmaw} = \Exp\|\matw_t\|^2 < \infty$), self-adjoint, and PSD.\footnote{The choice of Gaussian noise $\matw_t$ is made for simplicity, our analysis can be easily extended to work for any stochastic, sub-Gaussian distribution.} In LQR, the goal is to select a policy that minimizes cumulative quadratic losses $\inner{\matx_t}{ Q \matx_t} + \inner{\matu_t}{R\matu_t}$, where $Q,R$ are bounded, positive-definite operators. Given a bounded linear operator $K: \hilbx \to \R^{\dimu}$, the infinite-horizon cost of
the static feedback law  $\matu_t = K\matx_t$ is
\begin{align}
\Jfunc{K} \defeq \lim_{T \to \infty} \frac{1}{T}\Exp\left[\textstyle \sum_{t=1}^T \inner{\matx_t}{ Q \matx_t} + \inner{\matu_t}{R\matu_t}\right], \quad \text{ subject to } \matu_t = K \matx_t. \label{eq:JK}
\end{align}
We assume that $(\Ast,\Bst)$ is \emph{stabilizable}, meaning there
exists a controller $K$ such that $\Jfunc{K}$ is finite, which is true if and
only if
the spectral radius $\rho(\Ast + \Bst K) \defeq \limsup_{i \to \infty} \|(\Ast + \Bst K)^i\|_{\op}^{1/i} < 1$. We define the optimal control policy $\Kst\defeq \inf_{K:\hilbx \to \R^{\dimu}}\Jfunc{K}$. Under general
conditions, $\Kst$ is unique, does not depend on the noise covariance $\Sigmaw$, and its induced feedback law attains the optimal infinite horizon cost over \emph{all} control policies. 
In the online LQR protocol, the system matrices $(\Ast,\Bst)$ are  unknown, and the learner's goal is to adaptively learn to control the system so as attain low regret.
We define the regret of a learning algorithm $\calA$ (which chooses actions $\matu_t$ based on the history of
previous states and actions) as
\begin{align}
\Regret_T(\calA) \defeq \left(\textstyle\sum_{t=1}^T \inner{\matx_t}{ Q \matx_t} + \inner{\matu_t}{R\matu_t}\right) -  T\Jst, \quad \Jst \defeq \Jfunc{\Kst}.
\end{align}
While we are not aware of other work which studies  online control of LQR for infinite dimensional systems, this model has a long and rich  history within the control theory community dating back at least to the 1970s  (see for example \cite{curtain2012introduction, bensoussan2007representation}, and the references therein).

\paragraph{Dimension-Free Problem Parameter and Asymptotic Notation}
A central object in the analysis of LQR is the solution to a
discrete algebraic Ricatti equation (DARE) $\Pst$ (see
\Cref{eq:DARE}), which represents the value function for the optimal
controller (see \Cref{sec:full_preliminaries} for further details). The bounds in our
setting are parameterized by the operator norms of $\Pst$, the system matrices $(\Ast,\Bst)$, and
the noise covariance $\Sigmaw$. These terms are considered dimension-free in prior literature (e.g. \cite{mania2019certainty,simchowitz2020naive}), and indeed do not scale with the dimension when the state dimension is finite. We define the quantity $\Mst$ as a uniform bound on these dimension free system parameters.
\begin{equation}
\label{defn:Mst}
\Mst \defeq \max\{\opnorm{\Ast}^2, \opnorm{\Bst}^2, \opnorm{\Pst}, \opnorm{\Sigmaw},1\}
\end{equation}
We use $a \lesssim b$ to denote that $a \le c \cdot b$, where $c$ is a
universal constant independent of any problem parameters. We let $\log_+(x) \defeq \max\{\log(x), 1 \}$ For a
time horizon $T$, we use $\tilde{\calO}(f(T))$ to denote a term that, for $T$ sufficiently large, is bounded by $f(T)$ times logarithmic factors in relevant problem parameters. We define a weaker asymptotic notation $\mathcal{O}_\star(f(T))$ to denote
a term bounded by $f(T)^{1 + o(1)}$, times logarithmic factors, where $o(1) \to 0$ as $T \to \infty$.

\subsection{The Challenges of Dimension-Free Linear Control}
Though there is now a mature theory of dimension-free learning rates
in prediction and online decision making
\citep{bartlett2002rademacher,zhang2005learning,srinivas2009gaussian,rakhlin2014online},
dimension-free rates in reinforcement learning have remained more
elusive. This is because a learned model of transition dynamics that is accurate under one policy may be highly inaccurate on states visited under another policy.

Addressing policy mismatch in learning requires some handle on the
complexity of the class of state distributions that a learner can encounter
under available policies. Numerous strategies have been proposed, via
both combinatorial quantities like Eluder Dimension
\citep{russo2013eluder}  and linear-algebraic notions such as  Bellman
Rank \citep{jiang2017contextual}.  Targeting dimension-free rates more
specifically, recent work has studied MDPs with linear transitions,
where the parameters lie in an RKHS
\citep{yang2020bridging,yang2020reinforcement,agarwal2020pc}. These
developments assume that the dynamics can be factorized as the inner
product of two vectors that have bounded RKHS norm. The dynamical
matrices that arise in LQR, however, are considerably richer objects:
\emph{bounded operators} on the RKHS, rather than mere elements of it,
and this leads to fundamental differences between the settings. In a similar vein, \citet{kakade2020information} consider a nonlinear dynamical model with kernelized dynamics, finite-dimensional state, and well-conditioned Gaussian noise. Their setting is in general incomparable to ours, yet in the finite-dimensional LQR setting in which we overlap, our techniques yield more refined bounds. See the discussion following \Cref{theorem:regret_with_decay_rates} for further comparison.

It is not obvious that dimension-free learning in LQR is even possible. In fact, we show that the worst-case regret necessarily scales with the \emph{ambient input dimension}:
\begin{theorem}
\label{thm:main_lb_input}
\normalfont{\textbf{(informal)}} Fix any integer $r_{\matx} \ge 1$, and input dimension $\dimu$. Consider identity costs $Q = I_{\hilspace}$, $R = I_{\dimu}$, and noise covariance $\Sigmaw$ with trace $\traceb{\Sigmaw} \le r_{\matx}$.  Then, there exists a family of stabilizable instances $(A,B)$ with Hilbert-Schmidt norm bounded by $2$ such that any algorithm must suffer $\Omega(T)$ regret for $T \le r_{\matx} \dimu^2$, and $\Omega(\sqrt{ T r_{\matx} \dimu^2})$ regret thereafter.
\end{theorem}
The formal statement of the lower bound, its proof, and further
discussion are given in \Cref{app:lower_bound}. Notably, the above
theorem stands in stark contrast to analogous results for linear
MDPs and those presented in~\citet{kakade2020information}, which apply to infinite dimensional inputs.

Addressing high-dimensional states in LQR has  posed a challenge in both theory and practice \citep{sagaut2006large,liu2010interior}. All relevant prior work has incurred a dependence on the ambient state dimension. Model-based methods, which estimate system parameters and propose a policy based on those estimates, have required consistent recovery of those parameters (say, in operator or Frobenius norm), which is far stronger than a prediction error guarantee \citep{dean2018regret,dean2019sample,mania2019certainty,cohen2019learning}. Model-free methods, which eschew learning the system parameters in favor of directly optimizing the policy or value function, have been observed to suffer an even worse dependence on ambient dimension \citep{tu2019gap}. To summarize,  dimension-free statistical learning encounters major obstacles when translated to linear control.
\vspace{-5pt}
\subsection{Summary of Results \label{sec:informal_results}}
\paragraph{Certainty Equivalence}
Our results are based on \emph{certainty equivalence},
\citep{theil1957note, simon1956dynamic}, first analyzed for the online
LQR setting by \cite{mania2019certainty}. Given a stabilizable
$(A,B)$, we let $\Kinfty{A}{B}$ denote the optimal controller $K$, which minimizes the cost functional in \Cref{eq:JK} with $(\Ast,\Bst)$ set to $(A,B)$. It is a well-known fact that $\Kinfty{A}{B}$ has a closed form expression in terms of the system parameters $(A,B)$ and the DARE (see formal preliminaries in \Cref{sec:2}). Given estimates $(\Ahat,\Bhat)$ of $(\Ast,\Bst)$, the \emph{certainty equivalence controller} is $\Khat = \Kinfty{\Ahat}{\Bhat}$; that is, the optimal control policy as if the true system were $(\Ahat,\Bhat)$.  To be well-posed, this controller requires $(\Ahat,\Bhat)$ to be  stabilizable, which occurs when $(\Ahat,\Bhat)$ is sufficiently close to $(\Ast,\Bst)$ in operator norm (see, e.g., \Cref{prop:uniform_perturbation_bound}). 

\paragraph{A Regret Bound} We analyze a simple explore-then-commit
style algorithm,  $\mainalg$, based on certainty equivalence, similar
to that of \cite{mania2019certainty}. For now, we assume access to
initial ``warm-start'' estimates $(A_0,B_0)$ whose distance from
$(\Ast,\Bst)$ in operator norm is a small constant. We show how to get rid of this assumption later. The algorithm proceeds by first
synthesizing an exploratory controller $K_0 = \Kinfty{A_0,B_0}$, and
then collecting $\Texplore$ steps of samples with inputs $\matu_t = K_0
\matx_t + \matv_t$, where $\matv_t$ is i.i.d. Gaussian noise injected
for exploration. In the second phase, the algorithms constructs refined estimates $(\Ahat,\Bhat)$ by
performing ridge regression on the collected data, synthesizes the
certainty equivalence controller $\Khat = \Kinfty{\Ahat}{\Bhat}$, and selects inputs $\matu_t = \Khat \matx_t$ for the remainder of the protocol.

In \Cref{theorem:main_regret_theorem}, we demonstrate that this
relatively simple algorithm enjoys regret that scales polynomially with the eigendecay of the noise covariance $\Sigmaw$, Hilbert-Schmidt norm of $\Ast$, input dimension $\dimu$, and the operator norms of relevant system operators.
When specialized to common rates of eigendecay, we attain the following regret bounds.

\newtheorem*{thm:cor_informal}{\Cref{theorem:regret_with_decay_rates} (informal)}
\begin{thm:cor_informal} Let $\sigma_j$ be the eigenvalues of
$\Sigmaw$. If the initial estimates $(A_0,B_0)$ are sufficiently close to $(\Ast,\Bst)$,
then $\mainalg$ suffers regret at most:
\begin{itemize}
	\item (polynomial decay) $\calO_\star ( \sqrt{ \calC_{P} \dmax^2 T^{1 + 1/\alpha}} )$, if $\sigma_j = j^{-\alpha}$ for $\alpha > 1$,
	\item (exponential decay) $\mathcal{O}_\star ( \sqrt{\calC_P \dmax^2 \dimu T})$, if $\sigma_j = \exp(-\alpha j)$ for $\alpha > 0$.
	\item (finite dimension) $\widetilde{\mathcal{O}}(  \sqrt{\calC_P (\dimu +  \dimx)^3 T})$ if $\hilbx = \R^{\dimx}$.
\end{itemize}
In the above expressions, $\calC_P$ is a polynomial in $\Mst$ and $\dmax \defeq \max\{ \traceb{\covw}, \dimu, \logtracenorm \}$ where $\logtracenorm$ is slightly larger than $\traceb{\covw} + \normHS{\Bst}^2$
\end{thm:cor_informal}
Interestingly, this result shows that achieving dimension-free rates
for online linear control does not require new algorithmic ideas, but rather a refined analysis of classical ones, like certainty equivalence. In
particular, when specialized to finite dimension, our
regret bound has the same dimension dependence as the one of \cite{mania2019certainty}. \cite{simchowitz2020naive} show that this dependence is sharp in the regime where state and input spaces have the same dimension. More generally, however, the spectrum of $\covw$ is a significantly sharper complexity measure as indicated before and evidenced in the following example.
\paragraph{An Illustrative Example}Let $(\mate_i)_{i=1}^{\infty}$ be an orthonormal basis for $\hilbx$, $(\matf_i)_{i = 1}^{\dimu}$ an orthonormal basis for $\R^{\dimu}$, and consider the following problem instance:
\begin{align*}
\Ast = \frac{1}{2}\sum_{i=1}^{\dimu} \mate_i\otimes \mate_i  +  \sum_{i > \dimu} \frac{1}{i^2}\cdot \mate_i \otimes \mate_i, \quad \Bst = \sum_{i=1}^{\dimu} \mate_i \otimes \matf_i, \quad \covw = \sum_{i=1}^\infty \frac{1}{i^2} \cdot
\mate_i\otimes \mate_{i}	,
\end{align*}
where $Q= I_{\hilbx}$ and $R = I_{\dimu}$. In this example, $\Ast$ is
infinite-dimensional, yet only has a finite dimensional controllable
subspace. Therefore, in order to learn the optimal policy, it is not
necessary for the learner to estimate the whole system, but rather the
parts of it that are relevant for control as determined by noise and the controllability properties of $\Ast$ and $\Bst$. While guarantees from previous work are vacuous in this setting since the ambient system dimension is infinite, the example corresponds to $\alpha=2$ in case 1 of Theorem~\ref{theorem:regret_with_decay_rates}, with $\dmax = O(\dimu)$ and $\Mst \leq 1$, yielding an $\calO_\star(\dimu T^{3/4})$ regret bound based on our theory.

\paragraph{Suboptimality Bounds from Prediction Error}
 During the exploration phase, the $\mainalg$ algorithm selects inputs
 according to $\matu_t = K_0 \matx_t + \matv_t$.
Let $\expcovone$ denote the  stationary covariance over states induced by this policy (see \Cref{eq:stationary_cov} for details).
 The ridge regression step in $\mainalg$ recovers $\Bst$ in
 Hilbert-Schmidt norm (since $\Bst$ is finite rank), but recovers
 $\Ast$ only in the \emph{covariance-weighted} Hilbert-Schmidt norm
 $\|(\Ast - \Ahat)\expcov\|_{\HS}$, corresponding to the prediction
 error.  The key technical innovation in this paper that underlies all of our results is a perturbation bound on the suboptimality of the certainty equivalence controller $\Khat$ in terms of this prediction error.

\newtheorem*{thm:end_to_end_informal}{\Cref{theorem:end_to_end_bound} (informal)}
\begin{thm:end_to_end_informal} Let $K_0$ be any state-feedback
controller that stabilizes $(\Ast,\Bst)$, and let $\expcovone$ denote
the induced state covariance with $\uvar = 1$. Then, if $(\Ahat,\Bhat)$ are  within a small but constant operator norm error of $(\Ast,\Bst)$,
\begin{align*}
\Jfunc{\Khat} - \Jst \le \calC_{J} \cdot (\epsilonshape)^{2 - o(1)}, \text{ where } \epsilonshape \defeq \textstyle\max \left\{ \normHSil{(\Ahat - \Ast) \expcov},  \normHSil{\Bhat - \Bst} \right\}.
\end{align*}
Here, $\calC_J$ is a polynomial in  $\Mst$ and $o(1)$ denotes a term
that tends to $0$ as $\epsilonshape \to 0$.
\end{thm:end_to_end_informal}
Importantly, a $\epsilonshape^{2}$ scaling of the perturbation is known to be optimal \citep{mania2019certainty}, and in finite dimensions, the $o(1)$ term in the exponent of our bound can be discarded.
\paragraph{Alignment Condition} As stated, the regret guarantee guarantee for $\mainalg$
requires  access to a warm-start estimates $(A_0,B_0)$. We use this to perform a
projection step ensuring that the certainty equivalent controller
$\Khat$ is stabilizing. While this condition is
stronger than those that have previously appeared in the literature, which typically assume that the learner initially has access to an \emph{arbitrary} stabilizing controller $\Kinitial$, we show that, under a certain alignment condition, it is possible to achieve this warm start condition:
\newtheorem*{infprop:warm_start_estimates}{\Cref{prop:warm_start_estimates} (informal)}
\begin{infprop:warm_start_estimates} Under a certain alignment condition, which holds if all eigenvalues of $\Sigmaw$ are strictly positive (though decaying to zero, see \Cref{assumption:alignment}), after collecting $\mathcal{O}(1)$ many samples, ridge regression returns estimates $(A_0,B_0)$ of $(\Ast,\Bst)$ satisfying the requisite closeness condition for \Cref{theorem:main_regret_theorem}.
\end{infprop:warm_start_estimates}
One can stitch together an initial estimation phase described by the
theorem above with the analysis of $\mainalg$ to provide an algorithm
that dispenses with the access to warm-start estimates, yet this requires the above alignment condition. Crucially, we use this condition for coarse recovery up to a constant tolerance. Hence, the initial estimation phase adds only a constant burn-in to the regret. Importantly, the alignment condition \emph{does not} afford us consistent parameter recovery.

\subsection{Related Work}
In the interest of brevity, we provide an abridged discussion of
related work here, and defer an extended discussion to
\Cref{sec:extended_related_work}. The learning community has seen a
surge in interest in linear control and in system identification \citep{vidyasagar2006learning,hardt2018gradient,pereira2010learning,oymak2019non,simchowitz2018learning,sarkar2019near,dean2019sample}.
We consider the online LQR setting first proposed by
\cite{abbasi2011regret},  and subsequently studied by
\cite{faradonbeh2018optimality,cohen2019learning,dean2018regret,mania2019certainty,abeille2020efficient,simchowitz2020improper}.
Our work is based on the analysis of certainty equivalence for online control, first studied by \cite{mania2019certainty} and refined by \cite{simchowitz2020naive}. Concurrent work has also studied \emph{model-free} approaches for control \citep{fazel2018global,krauth2019finite,abbasi2019model,tu2019gap}. As noted, all analyses incur dependence on ambient system dimension.

Sample complexity guarantees depending on instrinsic measures of complexity (rather than ambient dimension) are well-known in supervised learning~\citep{bartlett2002rademacher, zhang2005learning} and bandit problems~\citep{srinivas2009gaussian}. More recently, these results have been extended to the reinforcement learning literature as well, for a class of problems defined as linear MDPs~\citep{jin2020provably,agarwal2020pc,yang2020provably}. Further discussion comparing the linear MDP regime to LQR is deferred to \Cref{sec:extended_related_work}. 

\newcommand{\dlyapname}{\mathsf{dlyap}}
\newcommand{\sigmau}{\sigma_{\matu}}
\newcommand{\Sigmast}{\Sigma_{\star}}
\newcommand{\Lfactor}{\mathcal{L}}
\newcommand{\Sigmainfty}[3]{(#1;#2,#3)}
\section{From Prediction Error Bounds to Controller Suboptimality}
\label{sec:2}
In this section, we establish the main perturbation bounds regarding the suboptimality of a certainty-equivalent controller $\Khat = \Kinfty{\Ahat}{\Bhat}$ given estimates $(\Ahat, \Bhat)$ that satisfy a prediction error bound under a particular exploratory distribution. In doing so, we highlight a key change of measure lemma that allows us to evaluate the behavior of the system under \emph{any} state-feedback law given only a prediction error bound under a single exploratory policy. We begin by stating some further preliminaries.

%
\paragraph{Formal Preliminaries}
As in finite-dimensional settings, the optimal controller $\Kinfty{A}{B}$ for LQR in infinite dimension can be computed in terms of the PSD operator $\Pinfty{A}{B}:\hilbx \to \hilbx$ which solves the Discrete Algebraic Riccati Equation (DARE),\footnote{See for example \cite{zabczyk1974remarks, zabczyk1975optimal, lee1972control, curtain2012introduction}.}
\begin{align}
\Pinfty{A}{B} &~~\text{ solves } \quad P = A^\herm P A - A^\herm P B(R + B^\herm P B)^{-1} B^\herm P A + Q. \label{eq:DARE}\\
\Kinfty{A}{B} &~~\defeq -(R + B^\herm P B)^{-1}B^\herm P A, \quad\text{ where } P = \Pinfty{A}{B} \label{eq:optimal_control}
\end{align}
We define $\Pst \defeq \Pinfty{\Ast}{\Bst}$ and recall $\Kst \defeq \Kinfty{\Ast}{\Bst}$.

The DARE is intimately related to the discrete Lyapunov operator, $\dlyapname$.  Given a bounded linear operator $A:\hilbx \to \hilbx$
that is stable (i.e. $\rho(A) < 1$), and a symmetric bounded operator $\Lambda: \hilbx \to \hilbx$,  $\dlyap{A}{\Lambda}$ denotes the solution to the equation $	X = A^\herm X A + \Lambda.$
A classic result in Lyapunov theory states that the solution $X$ is unique, and is given by $\dlyap{A}{\Lambda} = \sum_{j=0}^\infty
(A^\herm)^j \Lambda A^j$. For any controller $K$ such that $K$ is stabilizing for $(A,B)$ we define $\PinftyK{K}{A}{B} \defeq \textstyle \dlyap{A+BK}{Q + K^\herm R K}$,
which can be viewed as the value function induced by the controller $K$ (see \Cref{sec:full_preliminaries} for details). Two  consequences of this interpretation are that $\PinftyK{K}{A}{B} = \Pinfty{A}{B}$ for $K = \Kinfty{A}{B}$, and $\Pinfty{A}{B} \preceq \PinftyK{K'}{A}{B}$ for any other stabilizing controller $K'$.

We adopt the following notation to refer to the steady-state covariance operator for the true system $(\Ast,\Bst)$, where $\matu_t$ is chosen by combining a state feedback policy $K$ with isotropic Gaussian noise $\matv_t$:
\begin{align}
\Covst{K,\uvar} \defeq & \lim_{t\to \infty}\Exp[\matx_t \otimes \matx_t ], \text{ s.t. }  \matx_{t+1} = \Ast  \matx_t + \Bst \matu_t + \matw_t, \nonumber\\
&\text{ where } \matu_t = K\matx_t + \matv_t, \quad \matw_t \iidsim \calN(0,\Sigmaw), \matv_t \iidsim \calN(0,\uvar I). \label{eq:stationary_cov}
\end{align}
We let $\Covst{K} = \Covst{K,0}$. A short calculation reveals that,  
\begin{align*}
\Covst{K,\uvar} = \dlyap{(\Ast + \Bst K)^\herm}{\Sigmaw + \uvar \Bst\Bst^\herm}.
\end{align*}
Lastly, for the remainder of the presentation, we make the following assumption on the costs:
\begin{assumption}
\label{assumption:normalization}
	The cost operators $Q,R$ satisfy $Q, R \succ I$.
\end{assumption}
Since scaling both $Q$ and $R$ by a constant does not change the form of the optimal controller, this assumption is without loss of generality if the operators are already positive definite.
\subsection{Performance Difference and Change of Measure \label{subsec:change_of_measure}}
The suboptimality of a controller $K$ for an instance $(\Ast,\Bst)$ admits the following closed form, often referred to as the \emph{performance-difference lemma} \citep{fazel2018global}:
\begin{align}
\Jfunc{K} - \Jfunc{\Kst} &= \traceb{(R + \Bst^\herm \Pst \Bst)^\top \cdot (K - \Kst) \Covst{K} (K - \Kst)^\herm}\nonumber \\
&\le \opnorm{R + \Bst^\herm \Pst \Bst}\|(K - \Kst) \Covst{K}^{\frac{1}{2}} \|_{\HS}^2 \label{eq:performance_difference}.
\end{align}
Hence, the correct geometry in which $K$ should approximate $\Kst$ is in the HS norm, weighted by its steady-state covariances $\Covst{K}$.
Weighting by the $\Covst{K}^{\frac{1}{2}}$ is crucial for dimension-free bounds since 
recovery of $\Kst$ in the unweighted HS norm would incur dependence on ambient dimension.

One could achieve low error in the $\Covst{K}^{\frac{1}{2}}$-weighted norm if one already had access to samples with covariance $\Covst{K}$, but this logic  becomes circular.
Instead, we ensure that $\|(K - \Kst) \expcov\|_{\HS}^2$ is small, where $\expcovone = \Covst{K_0,\uvar}$ is the state covariance under an arbitrary stabilizing controller $K_0$ and some additional Gaussian excitation. Our first technical contribution shows that, up to constant factors, the exploratory covariance $\expcovone$ dominates the target $\Covst{K}$:
\begin{lemma}
\label{theorem:expcov_upper_bound}
Let $K$ be any stabilizing controller for $(\Ast, \Bst)$ and let $\Covst{K}$ be its induced state covariance. Then, for any  stabilizing controller $\Kinit$ and any variance $\uvar \ge 1$,  we have that $\expcovone = \Covst{K_0,\uvar}$ satisfies
\begin{align*}
	\Covst{K} \preceq \calC_{K, \uvar} \cdot \expcovone,
\end{align*}
where $\calC_{K, \uvar}= \max\left\{2,  \frac{128}{\uvar} \opnorm{\Sigma_\matw} \opnorm{K - \Kinit}^2 \opnorm{P_K}^3 \log \left(3 \opnorm{P_K} \right)^2\right\}$ for $P_K = \PinftyK{K}{\Ast}{\Bst}$.
\end{lemma}
Hence, it suffices to replace the performance difference bound in \Cref{eq:performance_difference} with an estimate in the norm induced by the exploratory covariance. More specifically, we have that, 
\begin{align*}
\|(K - \Kst) \Covst{K}^{\frac{1}{2}} \|_{\HS}^2  \leq \calC_{K, \uvar} \|(K - \Kst) \expcov \|_{\HS}^2,
\end{align*}
therefore ensuring good performance under the exploratory distribution is enough to ensure good performance under \emph{any} other induced distribution.

In finite dimensions with full-rank noise $\lambda_{\min}(\Sigmaw) > 0$, the above comparison follows quite directly. The key challenge in our setting is ruling out the possibility that one controller $K$ ``pushes'' large eigenvalues of $\Sigmaw$ into one region of the state space, in which the other controller $K_0$ induces small excitation. The key insight is that the closed loop systems $\Ast + \Bst K$ and $\Ast + \Bst K_0$ differ only along the column space of $\Bst$, and these directions are excited in the $\expcovone$ covariance due to the injection of Gaussian noise. Importantly, the bound holds \emph{without} further controllability assumptions, which may fail in infinite dimensions. 

\subsection{A Dimension-Free Perturbation Bound \label{sec:dimfree_perturbation}}
Building on the above insight, we show that it suffices to have a prediction error bound (i.e estimate $\Ast$ in the $\expcovone$-induced HS norm) in order to synthesize a close to optimal controller. Since $\dimu$ is finite, we  can recover $\Bst$ in the unweighted HS norm. Specifically, fix a stabilizing controller $K_0$, and set $\expcovone \defeq \Covst{K_0,\uvar}$. Now, define the error terms,
\begin{align}
\epsilonshape = \textstyle\max \{ \|(\Ahat - \Ast) \expcov\|_{\HS}, \|\Bhat - \Bst\|_{\HS} \}, ~~ \epsunifop \defeq \textstyle\max\{ \|\Ahat - \Ast\|, \|\Bhat - \Bst\| \}. \label{eq:eps_defs}
\end{align}
Here, $\epsilonshape$ corresponds to the relevant HS norms (weighted for $\Ast$, unweighted for $\Bst$). In addition to $\epsilonshape$, we also consider a uniform (unweighted) operator norm bound error $\epsunifop$. This error needs to be smaller than some problem dependent constant to ensure that $(\Ahat,\Bhat)$ is stabilizable, and that $\Khat = \Kinfty{\Ahat}{\Bhat}$ stabilizes $(\Ast,\Bst)$. To this end, our perturbation bound imposes the following condition:
\begin{condition} \label{cond:unif_close}
 The error $\epsunifop$ defined in \Cref{eq:eps_defs} satisfies $\epsunifop \le \Cstable \defeq 1/229 \opnorm{\Pst}^3 $.
\end{condition}

We now state the main perturbation bound. We assume $1 \le \uvar \lesssim 1$, as in our online algorithm.
\begin{theorem}
\label{theorem:end_to_end_bound}
Let $\uvar \ge 1$, $\Kinit = \Kinfty{A_0}{B_0}$, and $\expcovone \defeq \Sigmast(K_0,\uvar)$. If $(A_0, B_0)$, $(\Ahat, \Bhat)$ both satisfy \Cref{cond:unif_close}, then, for $\Khat \defeq \Kinfty{\Ahat}{\Bhat}$, 
\begin{align*}
	\Jfunc{\Khat} - \Jfunc{\Kst} \lesssim \sigmau^4\Mst^{36}  \cdot\Lfactor  \exp(\tfrac{1}{50}\sqrt{\mathcal{L}}) \cdot \epsilonshape^2 ,
	\quad \text{ where } \Lfactor \defeq \log\left(e + \tfrac{ 2e \|\Ahat - \Ast\|_{\op}^2\traceb{\expcovone}}{\epsilonshape^2} \right),
\end{align*}
and $\Mst$ is defined as in \Cref{defn:Mst}. Moreover, in finite dimensions with $\expcovone \succ 0$, $\Lfactor$ can be replaced by $\log(1 + \mathrm{cond}(\expcovone))$, where $\mathrm{cond}(\cdot)$ denotes condition number.
\end{theorem}
The main novelty of the above perturbation bound, relative to previous analysis of certainty equivalent control is that, assuming that $\epsunifop$ is smaller than a constant\footnote{This condition can be relaxed for $(A_0,B_0)$; it suffices that $K_0$ is an arbitrary, stabilizing controller for $(\Ast,\Bst)$.}, the suboptimality gap $\Jfunc{\Khat} - \Jfunc{\Kst}$ depends only on the weighted HS norm or prediction error $\epsilonshape^2$. In particular, we observe that $\Lfactor  \exp(\tfrac{1}{50}\sqrt{\mathcal{L}})$ grows more slowly as $\epsilonshape \to 0$ than any power $(\epsilonshape)^{\alpha}$; hence, for any $\alpha > 0$, there is some $c_{\alpha}$ such that $\Jfunc{\Khat} - \Jfunc{\Kst} \le c_{\alpha}\Mst^{36}\epsilonshape^{2-\alpha}$. 

Furthermore, in the finite dimensional setting with full rank noise covariance, the condition number of $\expcovone$ is bounded; thus, $\Lfactor = \mathcal{O}(1)$ in this regime, and the scaling is exactly $\epsilonshape^2$, which is known to be optimal \citep{mania2019certainty,simchowitz2020improper}. Moreover,  \Cref{theorem:end_to_end_bound} depends only on the operator norm of natural control theoretic quantities, and hides no dimension like terms. Hence, up to $\Mst$ dependence, it matches the best-possible perturbations bounds attainable in the finite dimensional setting. Lastly, little effort was made in sharpening the dependence on $\Mst$, which we believe can be refined considerably.

\begin{proof}[Proof Sketch of \Cref{theorem:end_to_end_bound}] The proof uses arguments in \Cref{subsec:change_of_measure} to reduce to bounding $\|(K - \Kst) \expcov\|_{\HS}$. A direct computation (\Cref{prop:controller_perturbation}) shows that this term can be bounded in terms of $\epsilonshape$, and the weighted error $\|\expcov (\Phat - \Pst)\expcov\|_{\HS}$, where $\Phat \defeq \Pinfty{\Ahat}{\Bhat}$ is the certainty equivalent value function. In \Cref{theorem:P_perturbation}, we bound the latter by considering a linear interpolating curve $ (A(t),B(t)), ~ t \in [0,1]$  between $(\Ast,\Bst)$ and $(\Ahat,\Bhat)$, and  the value function $P(t) = \Pinfty{A(t)}{B(t)}$ along that curve. $P(t)$ can be shown to be continuously differentiable for all $t \in [0,1]$ under \Cref{cond:unif_close}, and hence it suffices to bound the supremum of  $\|\expcov P'(t)\expcov\|_{\HS}$ over $t \in [0,1]$. Bounding this term requires the majority of our technical effort, and  relies on both the change-of-measure bounds similar to \Cref{theorem:expcov_upper_bound}, and a careful application of the self-bounding ODE method introduced by \cite{simchowitz2020naive}.
The full proof of the end-to-end bound, and all its constituent results, is outlined in \Cref{app:ricatti_perturb}.
\end{proof}

\section{Algorithms \& Regret Bounds}
Having concluded our discussion of certainty equivalence, in this section
we now leverage our earlier results to prove that a simple explore-then-commit style algorithm, $\mainalg$, achieves sublinear regret for LQR in infinite dimension. 

In order to more clearly communicate the salient features of our analysis, we divide this section into two parts. First, we prove a regret
bound assuming the learner has access to initial system estimates $(A_0, B_0)$ such that these operators lie within some problem dependent constant of the true system operators $(\Ast, \Bst)$. The following subsection shows how under an
appropriate alignment condition, one can incorporate an initial phase to the algorithm that achieves these warm start estimates $(A_0, B_0)$ while only adding a constant term (in $T$) to the overall regret.

\subsection{Regret with Warm-Start}
We now describe the $\mainalg$ algorithm assuming access to \emph{warm-start} estimates $(A_0,B_0)$ of the true system parameters $(\Ast, \Bst)$. In particular, we assume 
\begin{condition}[Warm Start]
\label{condition:warm_start}
The pair $(A_0, B_0)$ satisfy $\max\{\opnorm{A_0 - \Ast}, \opnorm{B_0 - \Bst}\} \le 1/2 \cdot \Cstable.$
\end{condition} 

Under a minor technical extension, $\Cstable = 1 / (229 \opnorm{\Pst}^3)$ can be replaced by a data dependent quantity $\lesssim  \opnorm{\Pinfty{A_0}{B_0}}^{-3}$, which can be can be verified using only a confidence interval around $(\Ast,\Bst)$. For simplicity, we present the main algorithm without this modification, which we defer  to \Cref{app:data_dependent}. Given these estimates, our algorithm, $\mainalg$, consists of the following explore-then-commit strategy:
\begin{enumerate}
	\item Synthesize $\Kinit \defeq \Kinfty{A_0}{B_0}$ and choose inputs $\matu_t = \Kinit \matx_t + \matv_t$ where $\matv_t \sim \calN(0,  I)$ ($\uvar = 1$) for $\Texplore$ many iterations, collecting observations $\{ (\matx_t, \matv_t) \}_{t=1}^{\Texplore}$.
	\item Compute system estimates $(\Ahat, \Bhat)$ via ridge regression on data  $\{ (\matx_t, \matv_t) \}_{t=1}^{\Texplore}$ followed by a projection onto a safe set around  warm start estimates $(A_0, B_0)$ (see \Cref{app:algorithm_descriptions} for full pseudocode).
	\item Synthesize $\Khat = \Kinfty{\Ahat}{\Bhat}$ and choose inputs $\matu_t = \Khat\matx_t$ for the remaining time steps.
\end{enumerate}
Throughout, we let $\sigma_j(\Lambda)$ denote the $j$-th largest eigenvalue of a PSD operator $\Lambda$. Furthermore, we define
$\logtracenorm \defeq \normHS{\Bst}^2 + \sum_{j=1}^\infty \sigma_j(\covw) \log(j)$, which captures the magnitude of noise in the system under an exploratory policy. Note that for $\logtracenorm$ to be finite, $\covw$ needs to be slightly stronger than trace class. We recall our earlier asymptotic notation: $\widetilde{\calO}(f(T))$ suppresses logarithms, and $\mathcal{O}_\star(f(T)) \defeq \widetilde{\mathcal{O}}(f(T)^{1 + o(1)})$ where $o(1) \to 0$ as $T$ goes to infinity. 

We now state our main regret bound for $\mainalg$. For simplicity, we assume that the initial state  is drawn from the steady state covariance, $\matx_1 \sim \calN(0,\expcovone)$ (see \Cref{sec:further_remarks_regret} for further discussion). 

\begin{theorem}
\label{theorem:main_regret_theorem}
Let $(\sigma_j)_{j=1}^\infty = (\sigma_j(\expcovone))_{j=1}^\infty $ be the eigenvalues of $\expcovone$ and define,
\begin{align*}
\textstyle d_\lambda \defeq |\{\sigma_j : \sigma_j \geq \lambda \}|, \quad \Ctail \defeq \frac{1}{\lambda}\sum_{ j >  d_\lambda} \sigma_j.
\end{align*}
If the learner has access to warm start estimates satisfying \Cref{condition:warm_start}, with probability $1-\delta$,  $\mainalg$ satisfies
\begin{align*}
	\Regret_T(\mainalg) \leq \mathcal{O}_\star \left( \sqrt{ \Mst^{42} \dmax^2(d_\lambda + \Ctail) T} \right)
\end{align*}
where $M_\star$ is as in \Cref{defn:Mst} and $\dmax \defeq \max\{ \traceb{\covw}, \dimu, \logtracenorm \}$.
\end{theorem}

We state this first theorem in terms of the eigenvalues of $\expcovone$, but the bounds can also be stated in terms of $\covw$ under particular eigenvalue decay assumptions. We carry out this translation in \Cref{subsec:lyapunov_theory} through novel eigenvalue comparison inequalities for Lyapunov operators which may be of independent interest. In particular, the following result formalizes our earlier statement about the spectrum of $\covw$ encoding the right problem complexity as we remarked in Section~\ref{sec:informal_results}.

\begin{theorem}\label{theorem:regret_with_decay_rates}
Let $(\sigma_j(\Sigmaw))_{j=1}^{\infty}$ be the (descending) eigenvalues of $\Sigmaw$. In the setting of \Cref{theorem:main_regret_theorem}, the $\mainalg$ algorithm suffers regret at most, 
\begin{itemize}
	\item (polynomial decay) $\mathcal{O}_\star \left( \sqrt{M_\star^{46} \dmax^2 T^{1 + 1/\alpha}} \right)$, if $\sigma_j(\Sigmaw) = j^{-\alpha}$ for $\alpha > 1$, 
	\item (exponential decay) $\mathcal{O}_\star \left( \sqrt{M_\star^{45} \dmax^2 \dimu T} \right)$, if $\sigma_j(\Sigmaw) = \exp(-\alpha j)$ for $\alpha > 0$.
	\item (finite dimension) $\widetilde{\mathcal{O}}\left(  \sqrt{ M_\star^{42} (\dimu +  \dimx)^3 T}\right)$, if $\hilbx = \R^{\dimx}$ and $\covw = I$.
\end{itemize}
\end{theorem}

For finite dimensional systems, the extra terms depending on $\Lfactor$ become $\mathcal{O}(1)$, and we achieve the optimal $\sqrt{\dmax^3 T}$ regret, which is optimal in the  regime where $\dimx \asymp \dimu$ \citep{simchowitz2020naive}. The results for polynomial and exponential decay illustrate how the ambient state dimension $\dimx$ is only a very coarse measure of complexity for linear control. Our analysis of certainty equivalence shows that the magnitude of the system noise, $\traceb{\covw}$, is a more accurate measure of problem hardness. While this measure is $\Omega(\dimx)$ in the case of a full rank noise covariance, it can be considerably smaller as per our example in the introduction.

\paragraph{Comparison to \cite{kakade2020information}} The authors consider a setting where the dynamics are kernelized into a feature map, and the dynamical parameters are linear operators in the kernel space. Their model requires a finite dimensional state with well-conditioned Gaussian noise. Furthermore, their bounds scale with the ambient state dimension (via the optimal control cost $J^{\star}$), even if their associated intrinsic kernel dimension is $\mathcal{O}(1)$. It is unclear how to extend their techniques to high-dimensional noise with spectral decay, due to the subtleties of their change-of-measure argument (Lemma 3.9).

Moreover, the techniques in this paper (\Cref{theorem:expcov_upper_bound} and \Cref{lemma:comparing_spectrum}) can also be used to attain refined bounds on their  algorithm-dependent intrinsic dimension quantity, $\gamma_T(\lambda)$, which improves on worst-case analyses based on kernel structure.

\subsection{Finding Warm Start Estimates}
\label{subsec:warm_start}
We now move on to discussing how to achieve initial warm start estimates that satisfy \Cref{condition:warm_start}. To do so, we require two additional assumptions. The first is standard within the online control literature (see for example \cite{dean2018regret,mania2019certainty,cohen2019learning}):
\begin{assumption}[Initial Controller]
\label{assumption:initial_controller}
The learner has initial access to a controller $\Kbase$ that stabilizes $(\Ast, \Bst)$.
\end{assumption}
The second is an alignment condition specific to our setting. Before stating it, we define $\Aclst \defeq (\Ast + \Bst \Kbase)$ and let $U(\Lambda_r + \Lambda_{/r})V$ be its SVD where $\Lambda_r = \diag(s_1, \dots, s_r)$ is a diagonal operator containing the first $r$ singular values and $\Lambda_{/r}$ is another diagonal operator whose first $r$ entries are 0 and the rest contain the tail singular values from  $s_{r+1}$ on. Furthermore, we define $\initcovone \defeq \Covst{\Kbase, \uvar}$ 
%

\begin{assumption}[Alignment]
\label{assumption:alignment}
There exists $r \textstyle< \infty$ and $\rho \textstyle> 0$ such that, 
\begin{align*}
V^\herm \Lambda_r^2 V \preceq \rho \initcovone \text{ and }s_{r+1} < \frac{1}{16}\Cstable.
\end{align*}
\end{assumption}
Since $\initcovone \succeq \covw$, \Cref{assumption:alignment}  holds for some $r < \infty$ as long as $\covw$ is positive definite in the sense that all its eigenvalues are strictly larger than, though decaying to, 0. Using this alignment condition, we show that the $\warmstart$ algorithm returns estimates $(A_0, B_0)$ satisfying \Cref{condition:warm_start}. Given an initial state $\matx_1 = 0$, $\warmstart$ chooses actions according to $\matu_t = \Kbase \matx_t + \matv_t$ for $\Tinit$ many iterations, where $\Tinit = \mathcal{O}(1)$ is a constant independent of the  horizon $T$. Having collected a constant number of samples, the algorithm returns ridge regression estimates. Ssee \Cref{app:algorithm_descriptions} for formal description of $\warmstart$.

\begin{proposition}
\label{prop:warm_start_estimates}	\normalfont{\textbf{(informal)}}
If \Cref{assumption:initial_controller,assumption:alignment} hold, there exists a constant $\Tinit$, independent of the time horizon $T$,  such that after collecting $\Tinit$ samples $(\matx_t, \matv_t)$ under the exploration policy, $\matu_t = \Kbase \matx_t + \matv_t$, for $\matv_t \sim \calN(0, \uvar I)$, with probability $1-\delta$, ridge regression returns estimates $A_0, B_0$ satisfying \Cref{condition:warm_start}. 

\end{proposition}
Using this estimation result, we can prove the following corollary:
\begin{corollary}
\label{corollary:regret_warm_start}
If \Cref{assumption:alignment,assumption:initial_controller} hold, then with probability $1-\delta$, the regret incurred by $\warmstart$ satisfies
\begin{align*}
\Regret_T(\warmstart) \lesssim  \log(1/\delta) \left(   \uvar \traceb{R} + \opnorm{Q + \Kbase^\herm R \Kbase} \traceb{\Covst{\Kbase,\uvar}}\right) \Tinit.
\end{align*}
\end{corollary}
In \Cref{sec:further_remarks_regret}, we describe how running $\warmstart$ followed by $\mainalg$ satisfies an end-to-end regret guarantee whose asymptotics exactly match those of the $\mainalg$ algorithm.

\section{Conclusion}
In summary, this paper presents the first dimension-free regret guarantees for online LQR. We show that with a warm start, a simple
approach based on certainty equivalence achieves sublinear regret for any $\Sigmaw$ whose eigendecay is ever so slightly faster than
trace-class and transition operator $\Ast$ that has finite Hilbert-Schmidt norm.  While our bounds are nearly optimal when specialized to finite dimension, they provide a step towards a much sharper understanding of problem complexity in broader settings. 

We believe that there are a number of promising directions for future work in this area. For example, it would be interesting to understand
whether the alignment condition we introduce is necessary for certainty equivalence to succeed, or whether having a small prediction error bound
(i.e $\normHSil{(\Ahat - \Ast)\expcov} \leq \epsilon$) is in fact sufficient to guarantee that the certainty equivalent controller
stabilizes the true system. Perhaps other algorithmic ideas, such as system level synthesis \citep{wang2019system}, could be used to
circumvent operator-norm closeness. Furthermore, instead of depending on the ambient state dimension, our bounds instead depend on spectral decay of the noise covariance $\covw$. It is an open question whether this measure of complexity can be made even
sharper in some settings. 

Lastly, given that our lower bounds rule out the possibility of infinite dimensional inputs, it is worth exploring what additional structural assumptions are necessary in order to incorporate richer control structures. More broadly, it would be exciting to understand whether dimension-free rates are possible in more general continuous control settings, or to extend the techniques from this paper to settings with partial observability like LQG.

\section*{Acknowledgements}
We gratefully acknowledge the support of Microsoft through the BAIR Open Research Commons and of the NSF through grant DMS-2023505. JCP is supported by an NSF Graduate Research Fellowship. MS is generously supported by an Open Philanthropy Fellowship grant. We thank Yasin Abbasi-Yadkori for pointing us to his self-normalized inequality for vectors in Hilbert space, \Cref{lemma:self_normalized}.

\newpage

\bibliographystyle{plainnat}
\bibliography{refs}

\begin{thebibliography}{67}
\providecommand{\natexlab}[1]{#1}
\providecommand{\url}[1]{\texttt{#1}}
\expandafter\ifx\csname urlstyle\endcsname\relax
  \providecommand{\doi}[1]{doi: #1}\else
  \providecommand{\doi}{doi: \begingroup \urlstyle{rm}\Url}\fi

\bibitem[Abbasi-Yadkori(2012)]{yasinthesis}
Yasin Abbasi-Yadkori.
\newblock \emph{Online learning for linearly parametrized control problems}.
\newblock PhD thesis, University of Alberta, 2012.

\bibitem[Abbasi-Yadkori and Szepesv{\'a}ri(2011)]{abbasi2011regret}
Yasin Abbasi-Yadkori and Csaba Szepesv{\'a}ri.
\newblock Regret bounds for the adaptive control of linear quadratic systems.
\newblock In \emph{Proceedings of the 24th Annual Conference on Learning
  Theory}, pages 1--26, 2011.

\bibitem[Abbasi-Yadkori et~al.(2014)Abbasi-Yadkori, Bartlett, and
  Kanade]{abbasi2014tracking}
Yasin Abbasi-Yadkori, Peter Bartlett, and Varun Kanade.
\newblock Tracking adversarial targets.
\newblock In \emph{International Conference on Machine Learning}, pages
  369--377, 2014.

\bibitem[Abbasi-Yadkori et~al.(2019)Abbasi-Yadkori, Lazic, and
  Szepesv{\'a}ri]{abbasi2019model}
Yasin Abbasi-Yadkori, Nevena Lazic, and Csaba Szepesv{\'a}ri.
\newblock Model-free linear quadratic control via reduction to expert
  prediction.
\newblock In \emph{The 22nd International Conference on Artificial Intelligence
  and Statistics}, pages 3108--3117. PMLR, 2019.

\bibitem[Abeille and Lazaric(2017)]{abeille2017thompson}
Marc Abeille and Alessandro Lazaric.
\newblock {Thompson sampling for linear-quadratic control problems}.
\newblock In Aarti Singh and Jerry Zhu, editors, \emph{Proceedings of the 20th
  International Conference on Artificial Intelligence and Statistics},
  volume~54 of \emph{Proceedings of Machine Learning Research}, pages
  1246--1254. PMLR, 20--22 Apr 2017.

\bibitem[Abeille and Lazaric(2018)]{abeille2018improved}
Marc Abeille and Alessandro Lazaric.
\newblock Improved regret bounds for thompson sampling in linear quadratic
  control problems.
\newblock \emph{Proceedings of Machine Learning Research}, 80, 2018.

\bibitem[Abeille and Lazaric(2020)]{abeille2020efficient}
Marc Abeille and Alessandro Lazaric.
\newblock Efficient optimistic exploration in linear-quadratic regulators via
  lagrangian relaxation.
\newblock In \emph{International Conference on Machine Learning}, pages 23--31.
  PMLR, 2020.

\bibitem[Agarwal et~al.(2020)Agarwal, Henaff, Kakade, and Sun]{agarwal2020pc}
Alekh Agarwal, Mikael Henaff, Sham~M. Kakade, and Wen Sun.
\newblock Pc-pg: Policy cover directed exploration for provable policy gradient
  learning.
\newblock In \emph{Advances in Neural Information Processing Systems}, 2020.

\bibitem[Agarwal et~al.(2019)Agarwal, Bullins, Hazan, Kakade, and
  Singh]{agarwal2019online}
Naman Agarwal, Brian Bullins, Elad Hazan, Sham Kakade, and Karan Singh.
\newblock Online control with adversarial disturbances.
\newblock In \emph{Proceedings of the 36th International Conference on Machine
  Learning}, volume~97 of \emph{Proceedings of Machine Learning Research},
  pages 111--119. PMLR, 2019.

\bibitem[Bartlett and Mendelson(2002)]{bartlett2002rademacher}
Peter~L Bartlett and Shahar Mendelson.
\newblock Rademacher and gaussian complexities: Risk bounds and structural
  results.
\newblock \emph{Journal of Machine Learning Research}, 3\penalty0
  (Nov):\penalty0 463--482, 2002.

\bibitem[Bensoussan et~al.(2007)Bensoussan, Da~Prato, Delfour, and
  Mitter]{bensoussan2007representation}
Alain Bensoussan, Giuseppe Da~Prato, Michel~C Delfour, and Sanjoy~K Mitter.
\newblock \emph{Representation and control of infinite dimensional systems}.
\newblock Springer Science \& Business Media, 2007.

\bibitem[Bertsekas()]{bertsekas1995dynamic}
Dimitri~P Bertsekas.
\newblock \emph{Dynamic programming and optimal control}, volume~1.

\bibitem[Cassel et~al.(2020)Cassel, Cohen, and Koren]{cassel2020logarithmic}
Asaf Cassel, Alon Cohen, and Tomer Koren.
\newblock Logarithmic regret for learning linear quadratic regulators
  efficiently.
\newblock In \emph{International Conference on Machine Learning}, pages
  1328--1337. PMLR, 2020.

\bibitem[Chen and Gu(2000)]{chen2000control}
Jie Chen and Guoxiang Gu.
\newblock \emph{Control-oriented system identification: an H-infinity
  approach}, volume~19.
\newblock Wiley-Interscience, 2000.

\bibitem[Chen and Yang(2021)]{hansonwright}
Xiaohui Chen and Yun Yang.
\newblock Hanson–wright inequality in hilbert spaces with application to
  $k$-means clustering for non-euclidean data.
\newblock \emph{Bernoulli}, 27\penalty0 (1):\penalty0 586--614, 02 2021.

\bibitem[Cohen et~al.(2018)Cohen, Hasidim, Koren, Lazic, Mansour, and
  Talwar]{cohen2018online}
Alon Cohen, Avinatan Hasidim, Tomer Koren, Nevena Lazic, Yishay Mansour, and
  Kunal Talwar.
\newblock Online linear quadratic control.
\newblock In \emph{International Conference on Machine Learning}, pages
  1029--1038. PMLR, 2018.

\bibitem[Cohen et~al.(2019)Cohen, Koren, and Mansour]{cohen2019learning}
Alon Cohen, Tomer Koren, and Yishay Mansour.
\newblock Learning linear-quadratic regulators efficiently with only sqrt(t)
  regret.
\newblock In \emph{International Conference on Machine Learning}, pages
  1300--1309. PMLR, 2019.

\bibitem[Curtain and Zwart(2012)]{curtain2012introduction}
Ruth~F Curtain and Hans Zwart.
\newblock \emph{An introduction to infinite-dimensional linear systems theory},
  volume~21.
\newblock Springer Science \& Business Media, 2012.

\bibitem[Dean et~al.(2018)Dean, Mania, Matni, Recht, and Tu]{dean2018regret}
Sarah Dean, Horia Mania, Nikolai Matni, Benjamin Recht, and Stephen Tu.
\newblock Regret bounds for robust adaptive control of the linear quadratic
  regulator.
\newblock In \emph{Advances in Neural Information Processing Systems}, pages
  4188--4197, 2018.

\bibitem[Dean et~al.(2019)Dean, Mania, Matni, Recht, and Tu]{dean2019sample}
Sarah Dean, Horia Mania, Nikolai Matni, Benjamin Recht, and Stephen Tu.
\newblock On the sample complexity of the linear quadratic regulator.
\newblock \emph{Foundations of Computational Mathematics}, pages 1--47, 2019.

\bibitem[Dragomir(2014)]{dragomir2014some}
Silvestru~Sever Dragomir.
\newblock Some trace inequalities for operators in hilbert spaces.
\newblock \emph{arXiv preprint arXiv:1409.6295}, 2014.

\bibitem[Faradonbeh et~al.(2018)Faradonbeh, Tewari, and
  Michailidis]{faradonbeh2018optimality}
Mohamad Kazem~Shirani Faradonbeh, Ambuj Tewari, and George Michailidis.
\newblock On optimality of adaptive linear-quadratic regulators.
\newblock \emph{arXiv preprint arXiv:1806.10749}, 2018.

\bibitem[Farahmand et~al.(2017)Farahmand, Barreto, and
  Nikovski]{farahmand2017value}
Amir-massoud Farahmand, Andre Barreto, and Daniel Nikovski.
\newblock Value-aware loss function for model-based reinforcement learning.
\newblock In \emph{Artificial Intelligence and Statistics}, pages 1486--1494,
  2017.

\bibitem[Fazel et~al.(2018)Fazel, Ge, Kakade, and Mesbahi]{fazel2018global}
Maryam Fazel, Rong Ge, Sham Kakade, and Mehran Mesbahi.
\newblock Global convergence of policy gradient methods for the linear
  quadratic regulator.
\newblock In \emph{International Conference on Machine Learning}, pages
  1467--1476. PMLR, 2018.

\bibitem[Goldenshluger(1998)]{goldenshluger1998nonparametric}
Alexander Goldenshluger.
\newblock Nonparametric estimation of transfer functions: rates of convergence
  and adaptation.
\newblock \emph{IEEE Transactions on Information Theory}, 44\penalty0
  (2):\penalty0 644--658, 1998.

\bibitem[Hardt et~al.(2018)Hardt, Ma, and Recht]{hardt2018gradient}
Moritz Hardt, Tengyu Ma, and Benjamin Recht.
\newblock Gradient descent learns linear dynamical systems.
\newblock \emph{The Journal of Machine Learning Research}, 19\penalty0
  (1):\penalty0 1025--1068, 2018.

\bibitem[Hazan et~al.(2020)Hazan, Kakade, and Singh]{hazan2020nonstochastic}
Elad Hazan, Sham Kakade, and Karan Singh.
\newblock The nonstochastic control problem.
\newblock In \emph{Algorithmic Learning Theory}, pages 408--421. PMLR, 2020.

\bibitem[Helmicki et~al.(1991)Helmicki, Jacobson, and
  Nett]{helmicki1991control}
Arthur~J Helmicki, Clas~A Jacobson, and Carl~N Nett.
\newblock Control oriented system identification: a worst-case/deterministic
  approach in h/sub infinity.
\newblock \emph{IEEE Transactions on Automatic control}, 36\penalty0
  (10):\penalty0 1163--1176, 1991.

\bibitem[{Hewer}(1971)]{hewer1971}
G.~{Hewer}.
\newblock An iterative technique for the computation of the steady state gains
  for the discrete optimal regulator.
\newblock \emph{IEEE Transactions on Automatic Control}, 16\penalty0
  (4):\penalty0 382--384, 1971.
\newblock \doi{10.1109/TAC.1971.1099755}.

\bibitem[Ioannou and Sun(2012)]{ioannou2012robust}
Petros~A Ioannou and Jing Sun.
\newblock \emph{Robust adaptive control}.
\newblock Courier Corporation, 2012.

\bibitem[Jiang et~al.(2017)Jiang, Krishnamurthy, Agarwal, Langford, and
  Schapire]{jiang2017contextual}
Nan Jiang, Akshay Krishnamurthy, Alekh Agarwal, John Langford, and Robert~E
  Schapire.
\newblock Contextual decision processes with low bellman rank are
  pac-learnable.
\newblock In \emph{International Conference on Machine Learning}, pages
  1704--1713. PMLR, 2017.

\bibitem[Jin et~al.(2020)Jin, Yang, Wang, and Jordan]{jin2020provably}
Chi Jin, Zhuoran Yang, Zhaoran Wang, and Michael~I Jordan.
\newblock Provably efficient reinforcement learning with linear function
  approximation.
\newblock In \emph{Conference on Learning Theory}, pages 2137--2143. PMLR,
  2020.

\bibitem[Kakade et~al.(2020)Kakade, Krishnamurthy, Lowrey, Ohnishi, and
  Sun]{kakade2020information}
Sham Kakade, Akshay Krishnamurthy, Kendall Lowrey, Motoya Ohnishi, and Wen Sun.
\newblock Information theoretic regret bounds for online nonlinear control.
\newblock In \emph{Advances in Neural Information Processing Systems},
  volume~33, pages 15312--15325, 2020.

\bibitem[Krauth et~al.(2019)Krauth, Tu, and Recht]{krauth2019finite}
Karl Krauth, Stephen Tu, and Benjamin Recht.
\newblock Finite-time analysis of approximate policy iteration for the linear
  quadratic regulator.
\newblock In \emph{Advances in Neural Information Processing Systems}, pages
  8514--8524, 2019.

\bibitem[Krstic et~al.(1995)Krstic, Kokotovic, and
  Kanellakopoulos]{krstic1995nonlinear}
Miroslav Krstic, Petar~V Kokotovic, and Ioannis Kanellakopoulos.
\newblock \emph{Nonlinear and adaptive control design}.
\newblock John Wiley \& Sons, Inc., 1995.

\bibitem[Lee et~al.(1972)Lee, Chow, and Barr]{lee1972control}
Kwang~Yun Lee, Shui-nee Chow, and Robert~O Barr.
\newblock On the control of discrete-time distributed parameter systems.
\newblock \emph{SIAM Journal on Control}, 10\penalty0 (2):\penalty0 361--376,
  1972.

\bibitem[Liu and Vandenberghe(2010)]{liu2010interior}
Zhang Liu and Lieven Vandenberghe.
\newblock Interior-point method for nuclear norm approximation with application
  to system identification.
\newblock \emph{SIAM Journal on Matrix Analysis and Applications}, 31\penalty0
  (3):\penalty0 1235--1256, 2010.

\bibitem[Ljung(1999)]{ljung1999system}
Lennart Ljung.
\newblock System identification.
\newblock \emph{Wiley encyclopedia of electrical and electronics engineering},
  pages 1--19, 1999.

\bibitem[Mania et~al.(2019)Mania, Tu, and Recht]{mania2019certainty}
Horia Mania, Stephen Tu, and Benjamin Recht.
\newblock Certainty equivalence is efficient for linear quadratic control.
\newblock In \emph{Advances in Neural Information Processing Systems}, pages
  10154--10164, 2019.

\bibitem[Mania et~al.(2020)Mania, Jordan, and Recht]{mania2020active}
Horia Mania, Michael~I Jordan, and Benjamin Recht.
\newblock Active learning for nonlinear system identification with guarantees.
\newblock \emph{arXiv preprint arXiv:2006.10277}, 2020.

\bibitem[Ouyang et~al.(2017)Ouyang, Gagrani, and Jain]{ouyang2017control}
Yi~Ouyang, Mukul Gagrani, and Rahul Jain.
\newblock Control of unknown linear systems with thompson sampling.
\newblock In \emph{2017 55th Annual Allerton Conference on Communication,
  Control, and Computing (Allerton)}, pages 1198--1205. IEEE, 2017.

\bibitem[Oymak and Ozay(2019)]{oymak2019non}
Samet Oymak and Necmiye Ozay.
\newblock Non-asymptotic identification of lti systems from a single
  trajectory.
\newblock In \emph{2019 American Control Conference (ACC)}, pages 5655--5661.
  IEEE, 2019.

\bibitem[Pereira et~al.(2010)Pereira, Ibrahimi, and
  Montanari]{pereira2010learning}
Jos{\'e} Pereira, Morteza Ibrahimi, and Andrea Montanari.
\newblock Learning networks of stochastic differential equations.
\newblock In \emph{Advances in Neural Information Processing Systems}, pages
  172--180, 2010.

\bibitem[Rakhlin and Sridharan(2014)]{rakhlin2014online}
Alexander Rakhlin and Karthik Sridharan.
\newblock Online non-parametric regression.
\newblock In \emph{Conference on Learning Theory}, pages 1232--1264. PMLR,
  2014.

\bibitem[Russo and Van~Roy(2013)]{russo2013eluder}
Daniel Russo and Benjamin Van~Roy.
\newblock Eluder dimension and the sample complexity of optimistic exploration.
\newblock In \emph{Advances in Neural Information Processing Systems},
  volume~26, 2013.

\bibitem[Sagaut(2006)]{sagaut2006large}
Pierre Sagaut.
\newblock \emph{Large eddy simulation for incompressible flows: an
  introduction}.
\newblock Springer Science \& Business Media, 2006.

\bibitem[Sarkar and Rakhlin(2019)]{sarkar2019near}
Tuhin Sarkar and Alexander Rakhlin.
\newblock Near optimal finite time identification of arbitrary linear dynamical
  systems.
\newblock In \emph{International Conference on Machine Learning}, pages
  5610--5618. PMLR, 2019.

\bibitem[Simchowitz and Foster(2020)]{simchowitz2020naive}
Max Simchowitz and Dylan Foster.
\newblock Naive exploration is optimal for online lqr.
\newblock In \emph{International Conference on Machine Learning}, pages
  8937--8948. PMLR, 2020.

\bibitem[Simchowitz et~al.(2018)Simchowitz, Mania, Tu, Jordan, and
  Recht]{simchowitz2018learning}
Max Simchowitz, Horia Mania, Stephen Tu, Michael~I Jordan, and Benjamin Recht.
\newblock Learning without mixing: Towards a sharp analysis of linear system
  identification.
\newblock In \emph{Conference on Learning Theory}, pages 439--473. PMLR, 2018.

\bibitem[Simchowitz et~al.(2019)Simchowitz, Boczar, and
  Recht]{simchowitz2019learning}
Max Simchowitz, Ross Boczar, and Benjamin Recht.
\newblock Learning linear dynamical systems with semi-parametric least squares.
\newblock In \emph{Conference on Learning Theory}, pages 2714--2802. PMLR,
  2019.

\bibitem[Simchowitz et~al.(2020)Simchowitz, Singh, and
  Hazan]{simchowitz2020improper}
Max Simchowitz, Karan Singh, and Elad Hazan.
\newblock Improper learning for non-stochastic control.
\newblock In \emph{Conference on Learning Theory}, pages 3320--3436. PMLR,
  2020.

\bibitem[Simon(1956)]{simon1956dynamic}
Herbert~A Simon.
\newblock Dynamic programming under uncertainty with a quadratic criterion
  function.
\newblock \emph{Econometrica, Journal of the Econometric Society}, pages
  74--81, 1956.

\bibitem[Srinivas et~al.(2010)Srinivas, Krause, Kakade, and
  Seeger]{srinivas2009gaussian}
Niranjan Srinivas, Andreas Krause, Sham Kakade, and Matthias Seeger.
\newblock Gaussian process optimization in the bandit setting: No regret and
  experimental design.
\newblock In \emph{Proceedings of the 27th International Conference on
  International Conference on Machine Learning}, ICML'10, page 1015–1022,
  2010.

\bibitem[Sun et~al.(2019)Sun, Jiang, Krishnamurthy, Agarwal, and
  Langford]{sun2019model}
Wen Sun, Nan Jiang, Akshay Krishnamurthy, Alekh Agarwal, and John Langford.
\newblock Model-based rl in contextual decision processes: Pac bounds and
  exponential improvements over model-free approaches.
\newblock In \emph{Conference on Learning Theory}, pages 2898--2933. PMLR,
  2019.

\bibitem[Theil(1957)]{theil1957note}
Henri Theil.
\newblock A note on certainty equivalence in dynamic planning.
\newblock \emph{Econometrica: Journal of the Econometric Society}, pages
  346--349, 1957.

\bibitem[Tsiamis and Pappas(2019)]{tsiamis2019finite}
Anastasios Tsiamis and George~J Pappas.
\newblock Finite sample analysis of stochastic system identification.
\newblock In \emph{2019 IEEE 58th Conference on Decision and Control (CDC)},
  pages 3648--3654. IEEE, 2019.

\bibitem[Tu and Recht(2019)]{tu2019gap}
Stephen Tu and Benjamin Recht.
\newblock The gap between model-based and model-free methods on the linear
  quadratic regulator: An asymptotic viewpoint.
\newblock In \emph{Conference on Learning Theory}, pages 3036--3083. PMLR,
  2019.

\bibitem[Tu et~al.(2017)Tu, Boczar, Packard, and Recht]{tu2017non}
Stephen Tu, Ross Boczar, Andrew Packard, and Benjamin Recht.
\newblock Non-asymptotic analysis of robust control from coarse-grained
  identification.
\newblock \emph{arXiv preprint arXiv:1707.04791}, 2017.

\bibitem[Vidyasagar and Karandikar(2006)]{vidyasagar2006learning}
Mathukumalli Vidyasagar and Rajeeva~L Karandikar.
\newblock A learning theory approach to system identification and stochastic
  adaptive control.
\newblock In \emph{Probabilistic and randomized methods for design under
  uncertainty}, pages 265--302. Springer, 2006.

\bibitem[Wang et~al.(2019)Wang, Matni, and Doyle]{wang2019system}
Yuh-Shyang Wang, Nikolai Matni, and John~C Doyle.
\newblock A system-level approach to controller synthesis.
\newblock \emph{IEEE Transactions on Automatic Control}, 64\penalty0
  (10):\penalty0 4079--4093, 2019.

\bibitem[Whittlesey(1965)]{whittlesey1965analytic}
Emmet~F Whittlesey.
\newblock Analytic functions in banach spaces.
\newblock \emph{Proceedings of the American Mathematical Society}, 16\penalty0
  (5):\penalty0 1077--1083, 1965.

\bibitem[Yang and Wang(2020)]{yang2020reinforcement}
Lin Yang and Mengdi Wang.
\newblock Reinforcement learning in feature space: Matrix bandit, kernels, and
  regret bound.
\newblock In \emph{International Conference on Machine Learning}, pages
  10746--10756. PMLR, 2020.

\bibitem[Yang et~al.(2020{\natexlab{a}})Yang, Jin, Wang, Wang, and
  Jordan]{yang2020provably}
Zhuoran Yang, Chi Jin, Zhaoran Wang, Mengdi Wang, and Michael Jordan.
\newblock Provably efficient reinforcement learning with kernel and neural
  function approximations.
\newblock \emph{Advances in Neural Information Processing Systems}, 33,
  2020{\natexlab{a}}.

\bibitem[Yang et~al.(2020{\natexlab{b}})Yang, Jin, Wang, Wang, and
  Jordan]{yang2020bridging}
Zhuoran Yang, Chi Jin, Zhaoran Wang, Mengdi Wang, and Michael~I Jordan.
\newblock Bridging exploration and general function approximation in
  reinforcement learning: Provably efficient kernel and neural value
  iterations.
\newblock \emph{arXiv preprint arXiv:2011.04622}, 2020{\natexlab{b}}.

\bibitem[Zabczyk(1974)]{zabczyk1974remarks}
Jerzy Zabczyk.
\newblock Remarks on the control of discrete-time distributed parameter
  systems.
\newblock \emph{SIAM Journal on Control}, 12\penalty0 (4):\penalty0 721--735,
  1974.

\bibitem[Zabczyk(1975)]{zabczyk1975optimal}
Jerzy Zabczyk.
\newblock On optimal stochastic control of discrete-time systems in hilbert
  space.
\newblock \emph{SIAM Journal on Control}, 13\penalty0 (6):\penalty0 1217--1234,
  1975.

\bibitem[Zhang(2005)]{zhang2005learning}
Tong Zhang.
\newblock Learning bounds for kernel regression using effective data
  dimensionality.
\newblock \emph{Neural Computation}, 17\penalty0 (9):\penalty0 2077--2098,
  2005.

\end{thebibliography}
\newpage
\appendix
\renewcommand{\contentsname}{Table of Contents: Appendix}
\tableofcontents
\addtocontents{toc}{\protect\setcounter{tocdepth}{2}}


\section{Extended Preliminaries, Organization, and Related Work \label{app:A}}

\subsection{Organization} The appendix is organized as follows:
\begin{itemize}
	\item \Cref{app:A} includes extended related work, a review of notation, and  further preliminaries on linear quadratic control. 
	\item \Cref{part:technical} contains all technical control theoretic contributions. \Cref{sec:change_of_covariance} describes the various change of covariance theorems used to swap between controllers. \Cref{app:ricatti_perturb} establishes the perturbation bounds for certainty equivalence. \Cref{sec:technical_lemmas} states and proves various technical lemmas used throughout. 
	\item \Cref{part:learning} addresses estimation of system parameters from a single trajectory. 
	\item In \Cref{part:regret}, \Cref{app:algorithm_descriptions} contains a formal statement of the algorithms,  \Cref{app:regret} proves an upper bound on the regret of $\mainalg$, and \Cref{app:lower_bound} describes the lower bound requiring finite dimensions. 
\end{itemize}

\subsection{Extended Related Work \label{sec:extended_related_work}}
The problem of learning the parameters of a linear system is historically referred to as \emph{system identification}, and has been studied at length for systems of finite dimension. Classical asymptotic results are detailed in \citet{ljung1999system}; early non-asymptotic results \citep{vidyasagar2006learning,hardt2018gradient,pereira2010learning} suffered from opaque and possible exponential dependencies on system parameters. \citet{dean2019sample} presented the first finite-sample guarantees for control synthesis from statistical data, and estimation techniques were later refined in subsequent works \citep{simchowitz2018learning,sarkar2019near}, and extended to systems with partial observation \citep{oymak2019non,simchowitz2019learning,tsiamis2019finite}. Parallel work has studied estimation in the frequency domain \citep{tu2017non,helmicki1991control,goldenshluger1998nonparametric,chen2000control}.

Building on system-identification, \emph{adaptive control} considers the problem of refining estimates of system parameters during a control task, so as to converge to a near-optimal policy. Classical results are detailed in \citep{krstic1995nonlinear,ioannou2012robust}. The online LQR setting considered in this work is a special case. The study of online LQR was initiated by \cite{abbasi2011regret}, who gave a computationally intractable algorithm based on optimism in the face of uncertainty (OFU). Their algorithm obtained $\sqrt{T}$ regret, albeit with a potentially exponential dependence on dimension.  \citet{dean2018regret} gave an efficient algorithm based on System Level Synthesis which obtained a $T^{2/3}$ regret bound with polynomial dependence on the relevant problem parameters. \cite{mania2019certainty}, \cite{faradonbeh2018optimality}, and \cite{cohen2019learning} simultaneously presented efficient algorithms enjoying  $\sqrt{T}$ regret (as well as  polynomial  dependence in other problem parameters). \citet{cassel2020logarithmic} and \citet{simchowitz2020naive} demonstrated that the $\sqrt{T}$ rate was indeed optimal. The latter provided matching upper and lower bounds of $\tilde{\Theta}(\sqrt{\dimx \dimu^2 T})$ in terms of time horizon and problem dimensions. Other approaches have studied Thompson sampling \citep{abeille2017thompson,ouyang2017control,abeille2018improved}, though  regret guarantees which depend transparently on both time horizon and dimension remain elusive. Finally, \cite{abeille2020efficient} provided an efficient implementation of the OFU algorithm introduced by \cite{abbasi2011regret}, which attained $\sqrt{T}$ regret, and sacrifices suboptimal dependence in problem dimension for improved dependence in other problem parameters. Similar regret guarantees were subsequently attained by \cite{kakade2020information} in a non-linear control setting similar to the one studied in \cite{mania2020active}. There has also been work on a number of related online control settings \citep{abbasi2014tracking,cohen2018online}, notably the nonstochastic control setting proposed by \citet{agarwal2019online} and expanded upon in \citet{hazan2020nonstochastic} and \citet{simchowitz2020improper}.

The majority of the above approaches to online and adaptive control are \emph{model-based}: that is, they learn a representation of the model of the dynamics, and update their control policy accordingly. This present work builds on past study of \emph{certainty equivalence} \citep{mania2019certainty,simchowitz2020improper}; other approaches include robust control synthesis via SLS \citep{dean2018regret,dean2019sample},  OFU \citep{abbasi2011regret,abeille2020efficient}, SDP relaxations \citep{cohen2018online}, and other convex methods \citep{agarwal2019online,simchowitz2020improper}. Concurrent work has also studied \emph{model-free} approaches in the batch \citep{fazel2018global,krauth2019finite} and online \citep{abbasi2019model} settings, though recent work suggest these approaches suffer from high variance, worse dimension dependence, and overall inferior sample complexity \citep{tu2019gap}.

Outside the controls literature, sample complexity and regret guarantees that do not explicitly depend on the ambient dimension, but on more intrinsic measures for learning in RKHSs are well-known in the supervised learning setting (see e.g.~\citep{bartlett2002rademacher, zhang2005learning}) as well as in bandit problems~\citep{srinivas2009gaussian}. More recently, these results have been extended to the reinforcement learning literature as well, for a class of problems defined as linear MDPs~\citep{jin2020provably,agarwal2020pc,yang2020provably}. While linear MDPs also make linearity assumptions on the system dynamics, the precise assumption is quite different from those present in LQR. In a linear MDP, the conditional distribution of the next state, given the current state and action is assumed to be linear under a known featurization of the state, action pair. In contrast, LQRs only require the conditional mean to be linear and do not guarantee certain nice properties of a linear MDP, such as the conditional expectation of any function of the next state being linear in the features of the current state, action pair. Consequently, the results from linear MDPs are incomparable to our work. 

Somewhat related to our development here, the observation of using an accuracy measure for model estimation that is informed by the value function parameterization has been recently leveraged in the reinforcement learning literature~\citep{farahmand2017value, sun2019model}, though the analysis techniques are quite different and the algorithms are not applicable to the continuous control setting.

\subsection{Notation Review}

\emph{Setting.} Recall the state $\matx_t \in \hilbx$, input $\matu_t \in \R^{\dimx}$, noise $\matw_t \in\hilbx$, true dynamics operators $(\Ast,\Bst)$, noise covariance $\Sigmaw$, and cost function $\Jfunc{K}$, defined for any state-feedback controller $K: \hilbx \rightarrow \R^{\dimu}$ which is stabilizing for $(\Ast, \Bst)$. The optimal cost of the LQR problem is $\Jst$, which is achieved the controller $\Kst$. 

Given operators $(A,B)$, $\Pinfty{A}{B}$ and $\Kinfty{A}{B}$ denote the value function and optimal controller (solving \Cref{eq:DARE} and \Cref{eq:optimal_control}). The $\dlyapname$ operator, and its related quantities, are described in the extended preliminaries below (\Cref{sec:full_preliminaries}). Given a controller $K$ which stabilizes $(A,B)$, we set
\begin{align}
\PinftyK{K}{A}{B} \defeq \dlyap{A+BK}{Q + K^\herm R K}, \label{eq:PK}
\end{align} 
where $Q$ and $R$ the costs operators. As detailed in \Cref{assumption:normalization}, we assumed throughout the entirety of our analysis that $\sigma_{\min}(Q) > 1$ and  $\sigma_{\min}(R) > 1$.

\emph{Exploration.} We use $K_0$ to denote exploratory controllers, and $\expcovone$ to denote the induced exploratory covariance with inputs $\matu_t = K_0 \matx_t + \matv_t$,  where $\matv_t \sim \calN(0,\uvar I )$ captures additional Gaussian excitation. We let $\Linit \defeq \uvar \Bst\Bst^\herm + \Sigmaw = \E \lbrack (\Bst\matv_t + \matw_t) \otimes (\Bst\matv_t + \matw_t) \rbrack$, so that $\expcovone = \dlyap{(\Ast + \Bst K_0)^\herm}{\Linit}$.  We recall the dimension-free parameter  from \Cref{defn:Mst}, 
\begin{align}
\Mst \defeq \max\{\opnormil{\Ast}^2,\opnormil{\Bst}^2, \opnorm{\Pst}^2,\opnorm{\Sigmaw},1\}.
\end{align}

\emph{Certainty Equivalence} We let $\Ahat, \Bhat$ denote estimates of the true system operators $\Ast, \Bst$. Likewise, we use $\Phat \defeq \Pinfty{\Ahat}{\Bhat}$ and $\Khat \defeq \Kinfty{\Ahat}{\Bhat}$ to describe the value function (solution to the  the DARE over $\Ahat, \Bhat$) and optimal controller for the estimated system. Similarly, we define $\Pst \defeq \Pinfty{\Ast}{\Bst}$.

\emph{Linear Algebra.} We let $\hilspace$ denote a Hilbert space containing the states. Inputs $\matu_t$ lie in $\R^{\dimu}$ where $\dimu < \infty$. We use upper case $X$ for linear operators, and lower case bound $\matx$ for vectors. $X^\herm$ and $\matx^\herm$ denote adjoints. We let $\langle \cdot, \cdot \rangle$ and $\|\cdot\|$ denote norms and inner products in the relevant Hilbert space. $\opnormil{\cdot}$, $\normHSil{\cdot}$, and $\nucnormil{\cdot}$ denote operator, Hilbert-Schmidt (abbreviated as HS), and trace norms, respectively. Occasionally, we use $\brackalign{XY} := \max_{W:\|W\|_{\op} = 1}\traceb{XWY}$.

Lastly, we adopt the shorthand $\log_+(x) \defeq \max\{\log(x), 1 \}$ and use $a \lesssim b$ to denote that $a \leq C \cdot b$ where $C$ is a universal constant.

\subsection{Extended Preliminaries \label{sec:full_preliminaries}}

\paragraph{The discrete Lyapunov operator and higher order operators}
Recall our earlier definition of the Lyapunov operator, 
\begin{definition}[Lyapunov Operator]
Let $A:\hilbx \to \hilbx$ be a  bounded stable linear operator and let $\Lambda: \hilbx \to \hilbx$ be self-adjoint, $\dlyap{A}{\Lambda}$ is a symmetric bounded operator that solves the matrix equation,
\begin{align} 
	X = A^\herm X A + \Lambda.
\end{align}
Furthermore, it has a closed-form expression given by: 
\begin{align} \dlyap{A}{\Lambda} = \sum_{j=0}^\infty (A^\herm)^j \Lambda A^j.  \label{eq:lyapunov_equation}
\end{align} 
\end{definition}
Throughout our analysis, we will make repeated use of the higher order Lyapunov operator.
\begin{definition}[Higher Order Lyapunov Operator]\label{def:higher_order_dlyap}\begin{align}
\dlyapm{A}{\Sigma}{m} = \sum_{j=0}^\infty (A^\herm)^j \Lambda A^j (j+1)^m
\end{align}
As is shown in \Cref{lemma:repeated_dlyap}, we have $\dlyapm{A^\herm}{\Sigma}{1}=\dlyap{A}{\dlyap{A}{\Sigma}}$.
\end{definition}
The Lyapunov operator satisfies a number of important properties which feature prominently in our technical discussion. We state the following lemma describing some of the most important properties and point the reader to \Cref{subsec:lyapunov_theory} for further results on Lyapunov theory. 

\begin{lemma}[Lemma B.5 in \citet{simchowitz2020naive}]
\label{lemma:basic_lyapunov_facts}
The following relationships hold for $\mathsf{dlyap}$:
\begin{enumerate}
	\item If $\Acl$ is stable and $Y \preceq Z$, then $\dlyap{\Acl}{Y} \preceq \dlyap{\Acl}{Z}$.
	\item If $Q \succeq I$ and $A + BK$ is stable, then
	\begin{align*}
		\pm \dlyap{A + BK}{Y} \preceq \dlyap{A +BK}{I} \cdot \opnorm{Y} \preceq \PinftyK{K}{A}{B} \cdot \opnorm{Y}
	\end{align*}
	\item If $Q \succ I$, then $\Pinfty{A}{B} \succ I$.
	\item If $\Acl$ is stable, then $\opnorm{\dlyap{\Acl}{I}} = \opnorm{\dlyap{\Acl^\herm}{I}}$.
\end{enumerate}
\end{lemma}

\paragraph{Stationary state covariances}
For controllers $K$ such that $\Ast + \Bst K$ is stable, we define the covariance operator:
\begin{align*}
\Covst{K,\uvar} &\defeq \lim_{t\to \infty}\Exp[\matx_t \otimes \matx_t ], \text{ s.t. }  \matx_{t+1} = \Ast  \matx_t + \Bst \matu_t + \matw_t,\\
&\text{ where } \matu_t = K\matx_t + \matv_t, \quad \matw_t \iidsim \calN(0,\Sigmaw), \matv_t \iidsim \calN(0,\uvar I).
\end{align*}
We specialize $\Covst{K} = \Covst{K,0}$. A short calculation shows that:
\begin{align}
\Covst{K,\uvar} = \dlyap{(\Ast + \Bst K)^\herm}{\Sigmaw + \uvar \Bst\Bst^\herm}.
\end{align}

\paragraph{Interpretation of $P$-matrices as value functions}
The $P$-matrices $\Pinfty{A}{B}$ and $\PinftyK{K}{A}{B}$ can be interpreted in terms of value functions. Specifically, consider an LQR problem with cost operators $Q,R$ and initial state $\matx_0$, but \emph{no noise}:
\begin{align}
\matx_{t+1} = \Ast \matx_t + \Bst \matu_t , \quad t \ge 0. \label{eq:nonoise_dynamics}
\end{align}
We can then verify that the operator $\PinftyK{K}{A}{B}$  defined in \Cref{eq:PK} satisfies
\begin{align*}
\inner{\matx}{\PinftyK{K}{A}{B} \matx}  = \sum_{t\ge 0} \inner{\matx_t}{Q \matx_t} + \inner{K \matu_t}{R K \matu_t} \text{ subject to \Cref{eq:nonoise_dynamics}} , \quad \text{ with } \matx_0 = \matx.
\end{align*}
The controller $\Kinfty{A}{B}$ can be shown to be the minimizer of the above cost over all policies, regardless of starting state (see e.g. \cite{bertsekas1995dynamic}). The following calculation provides another perspective of P capturing the costs of the LQR problem. From \Cref{eq:JK},  
\begin{align*}
	\Jfunc{K} &= \lim_{T \to \infty} \frac{1}{T}\Exp\left[\textstyle \sum_{t=1}^T \inner{\matx_t}{ Q \matx_t} + \inner{\matu_t}{R\matu_t}\right], \quad \text{ subject to } \matu_t = K \matx_t \\ 
	& = \lim_{T\rightarrow\infty} \; 
\traceb{(Q + K^\herm R K) \E \lbrack\matx_t \otimes \matx_t \rbrack} \\ 
& = \traceb{(Q + K^\herm R K) \Covst{K} } \\ 
& = \traceb{(Q + K^\herm R K) \dlyap{(\Ast + \Bst K)^\herm}{\Sigmaw}} \\ 
& = \traceb{\dlyap{\Ast + \Bst K}{Q + K^\herm R K} \covw} \\ 
& = \traceb{\PinftyK{K}{\Ast}{\Bst} \covw}.
\end{align*}
The cost of a controller $K$ is therefore captured by the trace inner product of $\PinftyK{K}{\Ast}{\Bst}$ and $\covw$.


\part{Perturbation Bounds and Technical Lemmas \label{part:technical}}


\section{The Change of Covariance Theorems \label{sec:change_of_covariance}}

In this section, we state and prove the various change-of-covariance theorems required in the paper, including \Cref{theorem:expcov_upper_bound}, and its generalization \Cref{theorem:expcov_upper_bound_general}. We begin by stating the more general result and then illustrate how \Cref{theorem:expcov_upper_bound} follows from this statement. We conclude by proving another comparison inequality between the covariance operators induced by different stabilizing controllers.

To state the general theorem, recall the \emph{higher order} Lyapunov operator defined  in \Cref{def:higher_order_dlyap}:
\begin{align}
\dlyapm{A}{\Sigma}{m} = \sum_{j=0}^\infty (A^\herm)^j \Lambda A^j (j+1)^m.
\end{align}
The result is as follows:
\begin{theorem}\label{theorem:expcov_upper_bound_general} Let $K_1 : \hilspace \to \R^{\dimu}$ be such that $A + B K_1$ is stable and let $\Lambda \in \psdhil$ be a trace class, positive semi-definite operator. Then, for any $K_2$ such that $A + BK_2$ is also stable, we have that
\begin{align*}
\dlyap{(A +BK_1)^\herm}{\Lambda} \preceq \dlyap{(A+ BK_2)^\herm}{\tLambda}
\end{align*}
where $\tLambda$ is a bounded operator  defined as
\begin{align*}
\tLambda \defeq 2\Lambda + 4 B (K_1 - K_2) \;\dlyapm{(A+BK_1)^\herm}{\Lambda}{2} \;(K_1 - K_2)^{\herm}B^{\herm}
\end{align*}
\end{theorem}

\begin{proof}[Proof of \Cref{theorem:expcov_upper_bound_general}] We argue by constructing a hypothetical dynamical system and analyzing its behavior in two ways. In particular, define
\begin{align*}	
\Aclone \defeq A + B K_1 \text{ and } \Acltwo \defeq A + B K_2
\end{align*}
and consider the system,
\begin{align}
\matx_{t+1} = \Aclone \matx_t + \matw_t, \text{ where }\matx_{0} = 0 \text { and } \forall \; t \ge 0, ~ \matw_t \iidsim \calN(0, \Lambda). \label{eq:first_system}
\end{align}
Then, for $\Sigma_1 \defeq \dlyap{\Aclone^\herm}{\Lambda}$ we have that
\begin{align*}
\Sigma_{1} =  \lim_{T \to \infty} \Sigma_{1;T}, \text{where } \Sigma_{1;T}  \defeq \Exp[\matx_T \otimes \matx_T],
\end{align*}
and where the above limit exists due to monotonicity of $\Sigma_{1;T}$. To prove our claim, let us express the evolution of \Cref{eq:first_system} as
\begin{align*}
\matx_{t+1} = \Acltwo \matx_t + (\Aclone - \Acltwo) \matx_t +  \matw_t = \Acltwo \matx_t + B \underbrace{(K_1 - K_2)\matx_t}_{:=\matu_t} + \matw_t.
\end{align*}
\newcommand{\matZ}{\mathbf{Z}}
Define the deterministic operator
$G_{T} := \begin{bmatrix} I  \mid \Acltwo \mid \dots  \mid \Acltwo^{T-1}\end{bmatrix}: \hilspace^{T} \to \hilspace$, and the random vector
\begin{align*}
\matz_{[T]} := \begin{bmatrix} B \matu_{T-1} + \matw_{T-1} \\ \dots \\ B \matu_{1} + \matw_{1} \\
B \matu_{0} + \matw_{0}
\end{bmatrix} \in \hilspace^{T}
\end{align*}
Now, we can rewrite $\matx_T$ as  $\matx_T = G_T \,\matz_{[T]}$, so that
\begin{align}
\Sigma_{1;T} := \Exp[\matx_{T}\otimes \matx_T] = G_{T}  \Exp[\matz_{[T]}\otimes \matz_{[T]} ] G_T^{\herm} \label{eq:ExpX_eq}
\end{align}
Given a general linear operator $X: \calU \to \calV$, define $\diag_T(X): \calU^T \to \calV^T$ as the block diagonal operator with $T$ copies of $X$ on its diagonal. Observe that if it holds that, for some psd $\tLambda: \hilspace \to \hilspace$, $\Exp[\matz_{[T]}\otimes \matz_{[T]}]  \preceq \diag_T(\tLambda)$, then by \Cref{eq:ExpX_eq},
\begin{align*}
\Sigma_{1} &= \lim_{T \to \infty}  G_{T}  \Exp[\matz_{[T]}\otimes \matz_{[T]} ] G_T^{\herm} \\
&\preceq \lim_{T \to \infty}  G_{T}\diag_T(\tLambda) G_{T} \\
&= \lim_{T \to \infty} \sum_{t=0}^{T-1} \Acltwo^{t} \tLambda \left(\Acltwo^\herm\right)^t = \dlyap{(A + B K_2)^\herm}{\tLambda}.
\end{align*}

Hence, it remains to bound $\Exp[\matz_{[T]} \otimes \matz_{[T]}] $.  Let us introduce the shorthand $\matu_{[T]} \defeq (\matu_{T-1}, \matu_{T-2},\dots,\matu_{0})$, and similarly for $\matw_{[T]}$ and $\matx_{[T]}$. Then, using the definition of $\matz_{[T]}$ and the fact that $\matu_{t} = (K_1 - K_2)\matx_t$,
\begin{align}
\Exp[\matz_{[T]} \otimes \matz_{[T]}] & = \Exp[(\diag_T(B) \matu_{[T]} + \matw_{[T]})^{\otimes 2}  ] \nn\\
&= \Exp[(\diag_T(B (K_1 - K_2)) \matx_{[T]} + \matw_{[T]})^{\otimes 2}] \nn\\
&\preceq 2\diag_T(B (K_1 - K_2))\cdot \Exp[ \matx_{[T]} \otimes \matx_{[T]}] \cdot \diag_T(B (K_1 - K_2))^{\herm}  + 2\Exp[\matw_{[T]} \otimes \matw_{[T]}]\nn\\
&= 2\diag_T(B (K_1 - K_2))\cdot \Exp[ \matx_{[T]} \otimes \matx_{[T]}] \cdot \diag_T(B (K_1 - K_2))^{\herm}  + 2\diag_T(\Lambda).\label{eq:zTbound_lastline}
\end{align}
The inequality on the third line follows from the fact that, 
\begin{align*}
	(a + b) \otimes (a + b) = 2 \left( a \otimes a + b \otimes b \right)- (a-b) \otimes (a-b).
\end{align*}
In the last line, we have used $\Exp[\matw_{[T]} \otimes \matw_{[T]} ] = \diag_T(\Lambda)$ since $\matw_{[T]} = (\matw_{T-1},\dots,\matw_0)$, and $\matw_t \iidsim \calN(0,\Lambda)$ for $t \ge 0$. To bound $\Exp[ \matx_{[T]} \otimes \matx_{[T]} ]$, we introduce the block Toeplitz operator:
\begin{align*}
\Toep_T: \hilspace^T \to \hilspace^T, \text{ with blocks } \Toep[i,j] = \Aclone^{j-i-1} \I_{j\ge i+1}.
\end{align*}
Then, we have the identity $\matx_{[T]} = \Toep_T \matw_{[T]},$
which implies that
\begin{align}
\Exp \matx_{[T]} \otimes \matx_{[T]} = \Toep_T \Exp \left[\matw_{[T]} \otimes \matw_{[T]}\right] \Toep_T^{\herm} = \Toep_T \cdot \diag_T(\Lambda) \cdot \Toep_T^{\herm}, \label{eq:matx_t_bound}
\end{align}
where $\matw_{t} \iidsim \calN(0,\Lambda)$ for $t \ge 0$. To conclude, let us decompose $\Toep_T$ into single-band operators via $Y_n: \hilspace^T \to \hilspace^T$ for $n \in \{1,\dots,T-1\}$, via:
\begin{align} \label{eq:Toep_decomp}
\Toep_T = \sum_{n=1}^{T-1}Y_n, \text{ where } Y_n: \hilspace^T \to \hilspace^T \text{ has blocks } Y_n[i,j] = \I_{ j = i + n} \cdot \Aclone^{n - 1}.
\end{align}
Hence, continuing from \Cref{eq:matx_t_bound}, and applying the substitution in \Cref{eq:Toep_decomp}, by \Cref{lemma:outer_product_bound} we have that
\begin{align*}
\Exp \lbrack\matx_{[T]} \otimes \matx_{[T]} \rbrack &\preceq  \Toep_T \cdot \diag_T(\Lambda) \cdot \Toep_T^{\herm} \tag{\Cref{eq:matx_t_bound}}\\
&=  \left(\sum_{n=1}^{T-1} Y_n\right)\cdot \diag_T(\Lambda) \cdot \left(\sum_{n=1}^{T-1} Y_n\right)^{\herm} \tag{\Cref{eq:Toep_decomp}}\\
&\preceq 2 \sum_{n=1}^{T-1} n^2 \cdot  Y_n \diag_T(\Lambda)Y_n^{\herm} \tag{\Cref{lemma:outer_product_bound}}
\end{align*}
A simple computation reveals that $X_n \defeq Y_n \diag_T(\Lambda)Y_n^{\herm}$ is block diagonal with $i$-th block given by,
\begin{align*}
X_n[i,i] = \I_{n \le i} \Aclone^{n-1} \Lambda \left( \Aclone^{\herm} \right)^{n-1}.
\end{align*}
Thus, $\left(\sum_{n=1}^T n^2 \cdot  Y_n \diag_T(\Lambda)Y_n^{\herm}\right)$ is block diagonal with $i$-th block given by
\begin{align*}
\sum_{n = 1}^{i} n^2 \cdot \Aclone^{n-1} \Lambda \Aclone^{(n-1)\herm} &\preceq \sum_{n \ge 0} (n+1)^2 \cdot  \Aclone^{n}\Lambda \Aclone^{n \herm} \\
&= \sum_{n \ge 0} (n+1)^2 \cdot (A + BK_1)^{n}\Lambda \left((A + B K_1)^{n}\right)^\herm = \dlyapm{\Aclone}{\Lambda}{2}
\end{align*}
Hence, $\Exp \matx_{[T]} \otimes \matx_{[T]} \preceq 2\diag_T\left(\dlyapm{\Aclone}{\Lambda}{2}
\right)$. Combining with \Cref{eq:zTbound_lastline} gives
\begin{align*}
\Exp[\matz_{[T]} \otimes \matz_{[T]}] &\preceq 2\diag_T( B (K_1 - K_2))\cdot \Exp[ \matx_{[T]} \otimes \matx_{[T]} ] \cdot \diag_T( B (K_1 - K_2))^{\herm}  + 2\diag_T(\Lambda), \\
&\preceq 4\diag_T(B (K_1 - K_2))\cdot \diag_T\left(\dlyapm{\Aclone}{\Lambda}{2}
\right) \cdot \diag_T(B (K_1 - K_2))^{\herm}  + 2\diag_T(\Lambda)\\
&=4\diag_T\left(B (K_1 - K_2) \dlyapm{\Aclone}{\Lambda}{2} (K_1 - K_2)^{\herm} B^{\herm}\right)  + 2\diag_T(\Lambda)\\
&=\diag_T\left(4 B (K_1 - K_2) \dlyapm{\Aclone}{\Lambda}{2}(K_1 - K_2)^{\herm} B^{\herm} + 2\Lambda\right).
\end{align*}
This concludes the proof.

\end{proof}
\subsection{Proof of \Cref{theorem:expcov_upper_bound} \label{sec:proof:expcov_upper_bound_one}}
\begin{proof}
By \Cref{theorem:expcov_upper_bound_general},
\begin{equation}
\label{eq:performance_diff_dlyap}
\dlyap{(\Ast + \Bst K)^\herm}{\Sigma_\matw} \preceq \dlyap{(\Ast + \Bst \Kinit)^\herm}{Z}
\end{equation}
where $Z = 2\Sigma_\matw + 4 \Bst \Bst^\herm \opnorm{K - \Kinit}^2 \opnorm{\dlyapm{(\Ast + \Bst K)^\herm}{\Sigma_\matw}{2}}$. \\
The remainder of the argument consists of simply bounding $Z$ in terms of a constant times $\Bst \Bst^\herm + \covw$.
By \Cref{lemma:dlyapm_bound}, we have that for $P_K = \PinftyK{K}{\Ast}{\Bst}$,
\begin{align*}
	\opnorm{\dlyapm{(\Ast + \Bst K)^\herm}{\Sigma_\matw}{2}} & \leq n^2 \opnorm{P_K} \opnorm{\Sigma_\matw} + (n^2 + 2n + 2) \opnorm{\Sigma_\matw} \opnorm{P_K}^4 \exp\left( - \opnorm{P_K}^{-1}n \right).
\end{align*}
Since $\opnorm{P_K} > 1$ (\Cref{lemma:basic_lyapunov_facts}), if we set $n = 4 \opnorm{P_K} \log\left(3 \opnorm{P_K} \right) \geq \left\lceil \opnorm{P_K} \log\left(5 \opnorm{P_K}^3 \right) \right\rceil$. This implies that, 
\begin{align*}
 n^2 \opnorm{P_K} \opnorm{\Sigma_\matw} \geq (n^2 + 2n + 2) \opnorm{\Sigma_\matw} \opnorm{P_K}^4 \exp\left( - \opnorm{P_K}^{-1}n \right)
\end{align*}
and hence, 
\begin{align*}
\opnorm{\dlyapm{(\Ast + \Bst K)^\herm}{\Sigma_\matw}{2}} \leq 32 \opnorm{\Sigma_\matw}\opnorm{P_K}^3 \log \left(3 \opnorm{P_K} \right)^2.
\end{align*}
Therefore, 
\begin{align*}
Z &\preceq 2 \covw + \frac{128}{\uvar} \opnorm{\Sigma_\matw} \opnorm{K - \Kinit}^2 \opnorm{P_K}^3 \log \left(3 \opnorm{P_K} \right)^2 \Bst \Bst^\herm \uvar \\
& \preceq \max\left\{2,  \frac{128}{\uvar} \opnorm{\Sigma_\matw} \opnorm{K - \Kinit}^2 \opnorm{P_K}^3 \log \left(3 \opnorm{P_K} \right)^2\right\} (\covw + \uvar \Bst \Bst^\herm).
\end{align*}
Going back to \eqref{eq:performance_diff_dlyap},
\begin{align*}
\dlyap{(\Ast + \Bst K)^\herm}{\Sigma_\matw} &\preceq \calC_{K, \uvar} \cdot  \expcovone 
\end{align*}
for $\calC_{K, \uvar}= \max\left\{2,  \frac{128}{\uvar} \opnorm{\Sigma_\matw} \opnorm{K - \Kinit}^2 \opnorm{P_K}^3 \log \left(3 \opnorm{P_K} \right)^2\right\}$.
\end{proof}

\subsection{A Change of Controller Lemma} \label{sec:change_controller}
 \begin{lemma}
\label{corollary:change_of_controller}
Let $K_1$ be a stabilizing controller for the instance $\Ast, \Bst$. Then, for any stabilizing $K_2$,
\begin{align*}
\dlyap{(\Ast + \Bst K_1)^\herm}{\expcovone} \preceq \calC_K \cdot \dlyap{(\Ast + \Bst K_2)^\herm}{\expcovone}
\end{align*}
for
\begin{align*}
\calC_K \defeq 2\left(1 + \frac{64 \opnorm{K_2 - K_1}^2}{\sigma^2_\matu} \opnorm{\expcovone} \opnorm{P_1}^3 \log(2 \opnorm{P_1})^2 \right)
\end{align*}
where $P_1  \defeq \PinftyK{K_1}{\Ast}{\Bst}$
\end{lemma}

\begin{proof}
We apply \Cref{theorem:expcov_upper_bound_general} to get that,
\begin{align*}
\dlyap{(\Ast + \Bst K_1)^\herm}{\expcovone} \preceq \dlyap{(\Ast + \Bst K_2)^\herm}{\tLambda}
\end{align*}
for $$\tLambda = 2 \expcovone + 4 \Bst (K_1 - K_2)\dlyapm{(\Ast + \Bst K_1)^\herm}{\expcovone}{2}(K_1 - K_2)^{\herm} \Bst^{\herm}
.$$
Since $\expcovone = \dlyap{(\Ast + \Bst \Kinit)^\herm}{\uvar \Bst \Bst^\herm  + \covw} \succeq \sigma^2_\matu \Bst\Bst^\herm$, then letting $\Delta_{\mathsf{op}} = \opnorm{K_2 - K_1}$ we have that
\begin{equation}
\label{eq:tlambda_upper_bound}
\tLambda \preceq 2\left(1 + \frac{2 \Delta_{\mathsf{op}}^2}{\sigma^2_\matu} \opnorm{\dlyapm{(\Ast + \Bst K_1)^\herm}{\expcovone}{2}}\right)  \expcovone.
\end{equation}
Using \Cref{lemma:dlyapm_bound}, for any $n \geq 0$, we can upper bound $\opnorm{\dlyapm{(\Ast + \Bst K_1)^\herm}{\expcovone}{2}}$ by
\begin{align}
\label{eq:change_of_controller_n}
n^2 \opnorm{\dlyap{(\Ast + \Bst K_1)^\herm}{\expcovone}} + (n^2 + 2n + 2) \opnorm{\expcovone} \opnorm{P_1}^4 \exp\left( -n \opnorm{P_1}^{-1}\right).
\end{align}
And, by properties of $\mathsf{dlyap}$ (\Cref{lemma:basic_lyapunov_facts}), 
\begin{align*}
	\opnorm{\dlyap{(\Ast + \Bst K_1)^\herm}{\expcovone}} &\leq \opnorm{\dlyap{(\Ast + \Bst K_1)^\herm}{I}} \opnorm{\expcovone} \\ 
	&= \opnorm{\dlyap{(\Ast + \Bst K_1)}{I}} \opnorm{\expcovone} \\ 
	& \leq \opnorm{P_1} \opnorm{\expcovone}.
\end{align*}
We can therefore upper bound \Cref{eq:change_of_controller_n} by:
\begin{align*}
n^2 \opnorm{\expcovone} \opnorm{P_1} + (n^2 + 2n + 2) \opnorm{\expcovone} \opnorm{P_1}^4 \exp\left( -n \opnorm{P_1}^{-1}\right).
\end{align*}
Setting $n = n_0  = \left\lceil\opnorm{P_1} \log \left( 5 \opnorm{P_1}^3\right) \right\rceil \leq 4 \opnorm{P_1} \log(2 \opnorm{P_1})$ (since $\opnorm{P_1} > 1$), we get that
\begin{align*}
(n_0^2 + 2n_0 + 2) \opnorm{\expcovone} \opnorm{P_1}^4 \exp\left( -n_0 \opnorm{P_1}^{-1}\right) \leq n_0^2 \opnorm{\expcovone} \opnorm{P_1}.
\end{align*}
Therefore,
\begin{align*}
\opnorm{\dlyapm{(\Ast + \Bst K_1)^\herm}{\expcovone}{2}} \leq 32 \opnorm{\expcovone} \opnorm{P_1}^3 \log(2 \opnorm{P_1})^2 \
\end{align*}
Plugging this upper bound into \cref{eq:tlambda_upper_bound} finishes the proof.
\end{proof}


\renewcommand{\Sigmaexp}{\Lambda_0}
\section{Proof of Certainty Equivalence Perturbation Bounds \label{app:ricatti_perturb}}

In this section, we provide the full proof of our end-to-end perturbation bound, \Cref{theorem:end_to_end_bound}, and the constituent results that comprise the argument. These supporting results are outlined in \Cref{sec:main_constit_perturbat} and \Cref{theorem:end_to_end_bound} is proven in \Cref{sec:end_to_end_proof}. The constituent results given in \Cref{sec:main_constit_perturbat} are then proven in the subsequent sections. As indicated in the main body of the paper, we assume throughout our presentation that $Q \succ I $ and $R \succ I$.

\subsection{Main Constituent Results \label{sec:main_constit_perturbat} }

\subsubsection{ Performance Difference}
We begin by stating our own variant of the now-ubiquitous performance difference lemma for LQR, which we use to bound the suboptimality of a controller $K$ in terms of its Hilbert-Schmidt difference from $\Kst$, weighted by the initial exploration covariance $\expcovone$. This result follows by combining the standard LQR performance difference lemma with our change of measure result, \Cref{theorem:expcov_upper_bound}.  
\begin{lemma}
\label{lemma:performance_difference}
Let $K$ be a stabilizing controller for $\Ast, \Bst$, then
\begin{align*}	
	\Jfunc K - \Jfunc \Kst \leq \calC_{K, \uvar} \opnorm{R + \Bst^\herm \Pst \Bst} \normHS{(K - \Kst)\expcov}^2
\end{align*}
for $\calC_{K, \uvar}= \max\left\{2,  \frac{128}{\uvar} \opnorm{\Sigma_\matw} \opnorm{K - \Kinit}^2 \opnorm{P_K}^3 \log \left(3 \opnorm{P_K} \right)^2\right\}$ for $P_K = \PinftyK{K}{\Ast}{\Bst}$ defined as in \Cref{theorem:expcov_upper_bound}.
\end{lemma}

\begin{proof}
The proof follows by applying \Cref{theorem:expcov_upper_bound} on the standard performance difference lemma for LQR. By Lemma 12 in \cite{fazel2018global} (as presented in Lemma 4 from \citet{mania2019certainty}),
\begin{align*}
J_\star(K) - J_\star(\Kst) = \traceb{\Covst{K}(K - \Kst)^\herm(R + \Bst^\herm \Pst \Bst)(K - \Kst)}.
\end{align*}
Now, by \Cref{theorem:expcov_upper_bound}, $\Covst{K} \preceq \calC_{K, \uvar}\cdot \expcovone.$, and therefore
\begin{align*}
\traceb{\Covst{K}(K - \Kst)^\herm(R + \Bst^\herm \Pst \Bst)(K - \Kst)} 
&\quad\leq \calC_{K, \uvar} \traceb{\expcovone(K - \Kst)^\herm(R + \Bst^\herm \Pst \Bst)(K - \Kst)} \\ 
& \quad\leq \calC_{K, \uvar} \opnorm{R + \Bst^\herm \Pst \Bst} \normHS{(K - \Kst)\expcov}^2.
\end{align*}
\end{proof}

\subsubsection{Intermediate $K$ Perturbation }
Next, we give a bound controlling the error $\normHSil{(\Kst - \Khat) \expcov}$ in terms of the maximum of the errors, 
\begin{align*}
\max\left\{ \normHSil{\Bhat - \Bst}, \normHSil{(\Ahat - \Ast) \expcov}, \normHSil{\expcov(\Phat - \Pst) \expcov} \right\}.
\end{align*} 
The following proposition is proven in \Cref{sec:K_perturb}.
\newcommand{\Mk}{M_K}
\begin{proposition}
\label{prop:controller_perturbation}
Recall our earlier definitions, $\Phat = \Pinfty{\Ahat}{\Bhat}, \Khat = \Kinfty{\Ahat}{\Bhat}$. Assume that the following inequality holds for some  $\uvar \geq 1$, and $\eps \leq 1$,
\begin{align*}\max\left\{ \normHSil{\Bhat - \Bst}, \normHSil{(\Ahat - \Ast) \expcov}, \normHSil{\expcov(\Phat - \Pst) \expcov} \right\} \leq \eps.
\end{align*} 
Furthermore, suppose $(\Ahat,\Bhat)$ is also stabilizable,  let $P_0 = \PinftyK{K_0}{\Ast}{\Bst}$, and set
\begin{align*}
\Mk= \max\left\{\opnorm{\Ast}^2, \opnorm{\Bst}^2, \opnorm{\Pst},  \opnormil{\Phat}, \opnormil{P_0}, \opnormil{\Sigmaw}, 1 \right\},
\end{align*}
 be a uniform bound on the operator norm on relevant operators. Then,
 \begin{align*}\normHS{(\Kst - \Khat) \expcov} \leq 9 \sigmau \eps \Mk^4.\end{align*}
\end{proposition}
The two challenges in applying \Cref{prop:controller_perturbation} are (a) verifying that the nominal system $(\Ahat,\Bhat)$ is stabilizable, so that the uniform bound $M$ is finite, and (b) bounding the weighted error $\expcov(\Phat - \Pst) \expcov$. We address these two parts of the argument in what follows.

\subsubsection{Operator-Norm $P$-Perturbation}

\begin{proposition}\label{prop:uniform_perturbation_bound} Fix two instances $(A_1,B_1)$, $(A_2,B_2)$, and define $\epsilon_{\op} = \max\{\|B_1 - B_2\|_{\op},\|A_1 - A_2\|_{\op} \}$. Suppose $P_1 = \Pinfty{A_1}{B_1}$, $K_1 \defeq \Kinfty{A_1}{B_1}$, and fix a tolerance parameter $\eta \in (0,1]$. Then, 
\begin{itemize}
	\item If $ \epsilon_{\op} \le \eta/(16\opnorm{P_1}^{3})$,  then $P_2 = \Pinfty{A_2}{B_2}$ and $\PinftyK{K_1}{A_2}{B_2}$ are bounded operators, and
	\begin{align}
	P_2 \preceq  \PinftyK{K_1}{A_2}{B_2} \preceq P_1 + \eta \|P_1\|_{\op} I \preceq (1 +\|P_1\|)\eta P_1
	\end{align}
	\item If the stronger condition $\epsilon_{\op} \le \eta/(16 (1+\eta)^4 \opnorm{P_1}^3)$ holds, then,$\|P_{2} - P_1\|_{\op} \le \eta \|P_1\|_{\op}$.
\end{itemize}
\end{proposition}
The proof is deferred to \Cref{sec:uniform_perturbation_proof}. One important consequence of the above bound is that, by setting $\eta = 1/11$, \Cref{prop:uniform_perturbation_bound} shows that the closeness condition, \Cref{cond:unif_close}, implies that $\|\Phat\|_{\op} \le 1.2 \|\Pst\|_{\op}$ (see \Cref{lemma:uniform_bound_p}). This is an essential ingredient in the subsequent covariance-weighted bound.

\subsubsection{Covariance-Weighted Perturbation of $\Pst$}
To state the covariance-weighted perturbation of $\Pst$, we recall the uniform closeness condition: \Cref{cond:unif_close}. 
\begin{align}
\epsunifop := \max\left\{ \opnorm{\Ahat - \Ast}, \opnorm{\Bhat - \Bst} \right\} \leq \frac{1}{229 \opnorm{\Pst}^3} \label{eq:maximum}
\end{align}

For the following proposition, we assume access to a stabilizing controller $K_0$, and define $P_0 := \opnorm{\PinftyK{K_0}{\Ast}{\Bst}}$. We set $\Mpst := \|\Pst\|_{\op}$ and $\Mpnot = \|P_0\|_{\op}$
\begin{proposition}\label{theorem:P_perturbation} Suppose $\Ahat, \Bhat$ satisfy \Cref{cond:unif_close} and recall our earlier definition $\Phat = \Pinfty{\Ahat}{\Bhat}.$ Then,  
\begin{align*}
\normHS{\expcov (\Phat -  \Pst) \expcov} &\lesssim \calC_{P} \cdot  \epsgam\cdot \sqrt{\log_+\left(\kappagam\right)} \cdot \phi(\kappagam)^{\alpha_{\op}},
\end{align*}
where $\phi(u) \defeq e^{\sqrt{\log(u)}}$ and $\epsgam, \calC_{P}, \kappagam, \alpha_{\op}$ are defined as: 
\begin{align*}
\epsgam^2 &\defeq  2\|(\Ast - \Ahat) \expcov \|_{\HS}^2  +  2.4\Mpst \normHSil{\Bst - \Bhat}^2\opnorm{\expcovone}\\
\calC_{P}  &\defeq \Mpnot^4 \Mpst^{3/2}  \sqrt{\left(1 + \tfrac{\opnorm{\expcovone} }{\sigma^2_\matu}  \right)\|\expcovone\|_{\op}}  \\ 
\kappagam &\defeq 1 + 2\tfrac{\|\Ahat - \Ast\|_{\op}^2\traceb{\expcovone}}{ \epsgam^2}\\
\alpha_{\op} &\defeq 2 \Mpst^{5/2} \epsilon_{\op} \le 1  / 100. 
\end{align*}
Furthermore, the above bound holds for any $\epsgam' \geq \epsgam$. For the case of finite-dimensional systems where $\expcovone \succ 0$, $\kappagam $ can be replaced by $1 + \mathrm{cond}(\expcovone)$, where $\mathrm{cond}(\cdot)$ denotes the condition number.
\end{proposition}
\Cref{theorem:P_perturbation} is the most involved and technically innovative in our analysis. Its proof is presented in  \Cref{sec:cov_pert}. 
\subsection{Proof of the End-to-End Perturbation Bound: \Cref{theorem:end_to_end_bound} \label{sec:end_to_end_proof}}
We recall:
\begin{align}
\Mst \defeq \max\{\opnorm{\Ast}^2, \opnorm{\Bst}^2, \opnorm{\Pst}, \opnorm{\Sigmaw},1\}.
\end{align}

\begin{proof}
	The proof follows by combining the performance difference lemma (\Cref{lemma:performance_difference}), our controller perturbation bound (\Cref{prop:controller_perturbation}), and the riccati perturbation bound (\Cref{theorem:P_perturbation}). 

	\paragraph{Applying performance difference} Starting from the performance difference lemma, we have that 
	\begin{align*}
	J(\Khat) - J(\Kst) \leq \calC_{\Khat, \uvar} \opnorm{R + \Bst^\herm \Pst \Bst} \normHS{(\Khat - \Kst)\expcov}^2
	\end{align*}
	for $\calC_{\Khat, \uvar}= \max\left\{2,  \frac{128}{\uvar} \opnorm{\Sigma_\matw} \opnorm{\Khat - \Kinit}^2 \opnorm{P_{\Khat}}^3 \log \left(2 \opnorm{P_{\Khat}} \right)^2\right\}$ where $P_{\Khat} = \PinftyK{\Khat}{\Ast}{\Bst}$. 

	We now simplify these terms. Since $\Kinit = \Kinfty{A_0}{B_0}$ and $\Khat = \Kinfty{\Ahat}{\Bhat}$ where both pairs of systems $(A_0, B_0)$, $(\Ahat, \Bhat)$ satisfy the uniform closeness conditions, we can apply \Cref{lemma:uniform_bound_p} to conclude that
	\begin{align*}
		\PinftyK{\Khat}{\Ast}{\Bst} \preceq 1.2 \Pst, \quad \opnorm{\Pinfty{\Ahat}{\Bhat}} \leq 1.1 \opnorm{\Pst},\quad \text{ and } \opnorm{\Pinfty{A_0}{B_0}} \leq 1.1 \opnorm{\Pst}.
	\end{align*}
	Furthermore, by \Cref{lem:Pk_bound}, $\opnormil{\Khat}^2 \leq \opnormil{\Pinfty{\Ahat}{\Bhat}}$ and similarly, $\opnorm{\Kinit}^2 \leq \opnorm{\Pinfty{A_0}{B_0}}$. Therefore, 
	\begin{align*}
	\opnormil{\Khat - \Kinit}^2  \leq 2(\opnormil{\Khat}^2 + \opnorm{\Kinit}^2) \leq 4.4 \opnorm{\Pst}.
	\end{align*}
	Recalling that $\Mst = \opnorm{\Pst} \geq 1$ and $\uvar \geq 1$,  
	\[
	\frac{128}{\uvar} \opnorm{\Sigma_\matw} \opnorm{\Khat - \Kinit}^2 \opnorm{P_{\Khat}}^3 \log \left(2 \opnorm{P_{\Khat}} \right)^2 \lesssim  M_\star^5 \log(M_\star)^2. 
	\]
	Therefore, under the uniform closeness assumptions, we have that,
	\begin{align}
	\label{eq:simplified_pd}
	J(\Khat) - J(\Kst) \lesssim  M_\star^7 \log(M_\star)^2\normHS{(\Khat - \Kst)\expcov}^2,
	\end{align}
	where we have used the calculation $\opnorm{R + \Bst^\herm \Pst \Bst} \leq \opnorm{R} + \opnorm{\Bst^\herm \Pst \Bst} \lesssim M_\star^2$.

	\paragraph{Controller perturbation} Now, we include our controller perturbation bound (\Cref{prop:controller_perturbation}) which states that for $\Mk$ defined as,    
	\begin{align*}
\Mk= \max\left\{\opnorm{\Ast}^2, \opnorm{\Bst}^2, \opnorm{\Pst},  \opnormil{\Phat}, \opnormil{P_0}, \opnormil{\Sigmaw}, 1 \right\},
\end{align*}
 we have that
 \begin{align*}\normHS{(\Kst - \Khat) \expcov} \leq 9 \sigmau \eps \Mk^4.\end{align*}
 for 
 
 \[
 \epsilon = \max\left\{\normHS{\Bhat - \Bst}^2, \normHS{(\Ahat - \Ast)\expcov}^2, \normHS{\expcov(\Phat -\Pst)\expcov}^2 \right\}.
 \]
	Under the uniform operator norm closeness, we again have $\|\Phat\|_{\mathrm{op}} \lesssim  \opnorm{\Pst} \lesssim \Mst$, and similarly for $P_0$. 
	Therefore, $\Mk$ as defined in \Cref{prop:controller_perturbation} satisfies $\Mk \lesssim \Mst$ . Using the controller perturbation, we then conclude that,
	\begin{align}
	\label{eq:simplified_controller_perturbation_bound}
	\normHS{(\Khat - \Kst)\expcov}^2 \lesssim \sigmau M_\star^9 \log(M_\star)^2 \epsilon.
	\end{align}

	\paragraph{Riccati perturbation} The last step in the proof is to apply our Riccati perturbation bound from \Cref{theorem:P_perturbation}. The bound states that, 
	\begin{align*}
	\normHS{\expcov (\Phat -  \Pst) \expcov}^2 &\leq \calC_{P}^2 \cdot  \epsgam^2 \cdot \log_+ \left( \kappagam\right) \cdot \phi(\kappagam)^{1/50}
	\end{align*}
	where $\phi(u) \defeq e^{\sqrt{\log(u)}}$, and the remaining quantities are defined as, 
	\begin{align*}
	 \epsgam^2 &\defeq 2 \|(\Ast - \Ahat) \expcov \|_{\HS}^2  + 2.4 \opnorm{\Pst} \normHS{\Bst - \Bhat}^2\opnorm{\expcovone} \\
	 \calC_{P}  &\lesssim  \opnorm{\Pinfty{A_0}{B_0}}^4\opnorm{\Pst}^{3/2}  \sqrt{\left(1 + \tfrac{\opnorm{\expcovone} }{\sigma^2_\matu}  \right)\|\expcovone\|_{\op}}  \\
	\kappagam & \defeq 1 + 2\tfrac{\|\Ahat - \Ast\|_{\op}^2\traceb{\expcovone}}{ \epsgam^2} 
	\end{align*}
	We note that by \Cref{cond:unif_close}, $\opnorm{\Pinfty{A_0}{B_0}} \lesssim  \opnorm{\Pst}$. Furthermore, by \Cref{lemma:basic_lyapunov_facts} and the fact $\uvar \ge 1$,
	\begin{align*}
		\opnorm{\expcovone}  &= \opnorm{\dlyap{(\Ast + \Bst \Kinit)^\herm}{\Bst\Bst^\herm \uvar + \covw}} \\
		& \leq \opnorm{\dlyap{(\Ast + \Bst \Kinit)^\herm}{I}} \opnorm{\Bst\Bst^\herm \uvar + \covw} \\ 
		& = \opnorm{\dlyap{\Ast + \Bst \Kinit}{I}}\opnorm{\Bst\Bst^\herm \uvar + \covw} \\ 
		& \leq  \Mst (\uvar \Mst),
	\end{align*}  
	where we have used the fact that $I \preceq Q$ and hence $$\opnorm{\dlyap{\Ast + \Bst \Kinit}{I}} \leq \opnorm{\dlyap{\Ast + \Bst \Kinit}{Q +\Kst^\herm R \Kst }} = \opnorm{\Pst}.$$  
Using this calculation, we can then bound the relevant quantities as:
	\begin{align*}
		\calC_P^2 &\lesssim \Mst^{11}(1 + \opnorm{\expcovone}/\uvar) \opnorm{\expcovone} \lesssim \uvar \Mst^{15}\\
		\epsgam^2 &  \lesssim  M_\star^3 \epsilon^2\\ 
		\log_+(\kappagam) & \leq\log\left(e + \frac{2e \|\Ahat - \Ast\|_{\op}^2\traceb{\expcovone}}{\epsilon^2} \right)  \defeq \Lfactor,
	\end{align*}
	where we also note that \Cref{theorem:P_perturbation} allows us to replace $\Lfactor$ by $\log(1 + \mathrm{cond}(\expcovone))$ in finite dimension with $\expcovone \succ 0$.  Using these simplifications, we get that, 
	\begin{align}
	\normHS{\expcov (\Phat -  \Pst) \expcov}^2 \lesssim  \uvar M_\star^{18} \sqrt{\Lfactor} \exp(\frac{1}{50}\sqrt{\Lfactor}) \epsilon^2.
	\end{align}
	\paragraph{Wrapping up} Combining this last equation with the bounds from \Cref{eq:simplified_pd} and \Cref{eq:simplified_controller_perturbation_bound}, the total power of $\Mst$ is $\Mst^{(7 + 9 + 18)} \log(\Mst)^2\leq  \Mst^{36}$, yielding:
	\begin{align*}
		J(\Khat) - J(\Kst) \lesssim \sigmau^4 \Mst^{36}  \Lfactor^{1/2} \exp(\frac{1}{50}\sqrt{\Lfactor} ) \epsilon^2.
	\end{align*}
	\end{proof}

\subsection{Proof of Intermediate $K$ Perturbation: \Cref{prop:controller_perturbation} \label{sec:K_perturb}}
Recall the definition
\begin{align*}
\Mk= \max\left\{\opnorm{\Ast}^2, \opnorm{\Bst}^2, \opnorm{\Pst},  \opnormil{\Phat}, \opnormil{P_0}, \opnormil{\Sigmaw}, 1, \right\},
\end{align*}
where $P_0 = \PinftyK{\Kinit}{\Ast}{\Bst}$. To simplify notation, we use $M = \Mk$. Furthermore, we assume $\uvar \geq 1$, as is chosen in our algorithm later on. Our proof appeals to the following lemma from \citet{mania2019certainty}:
\begin{lemma}[\cite{mania2019certainty}]
\label{lemma:strong_convexity_closeness}
Let $f_1, f_2$ be $\gamma$-strongly convex functions. Let $\matx_i = \argmin f_i(\matx)$ for $i = 1,2$. If $\norm{\nabla f_1(\matx_2)}\leq \eps$ then, $\norm{\matx_1 - \matx_2} \leq \eps / \gamma$
\end{lemma}

\begin{proof}[Proof of \Cref{prop:controller_perturbation}]
	The proof is inspired by that of Lemma 2 in  \citet{mania2019certainty}. Consider the functions $\fst$ and $\fhat$ defined as,
	\begin{align*}
	\fst(X)  &\defeq \frac{1}{2} \normHS{R^{1/2}X}^2 + \frac{1}{2} \normHS{\Pst^{1/2}\Bst X}^2 + \innerHS{\Bst^\herm\Pst \Ast \expcov }{X}  \\
	\fhat(X) &\defeq \frac{1}{2} \normHS{R^{1/2}X}^2 + \frac{1}{2} \normHS{\Phat^{1/2}\Bhat X}^2 + \innerHS{\Bhat^\herm\Phat \Ahat \expcov}{X},
	\end{align*}
where $\innerHS{A}{B} = \traceb{A^\herm B}$ denotes the Hilbert-Schmidt inner product.
Both functions are strongly convex with strong convexity parameter lower bounded by $\sigma_{\min}(R) \geq 1$. We observe that,
	\begin{align*}
	\nabla \fst(X)  &= (\Bst^\herm \Pst \Bst + R)X + \Bst^\herm\Pst \Ast \expcov . \\
	\nabla \fhat(X)	& = (\Bhat^\herm \Phat \Bhat + R)X + \Bhat^\herm\Phat \Ahat \expcov .
	\end{align*}
and hence,
\begin{align*}
	X_\star &= \argmin_X \fst(X) = -(\Bst^\herm \Pst \Bst + R)^{-1}\Bst^\herm\Pst \Ast \expcov  = \Kst \expcov  \\
	\widehat{X} &= \argmin_X \fhat(X) = -(\Bhat^\herm \Phat \Bhat + R)^{-1}\Bhat^\herm\Phat \Ahat \expcov  = \Khat \expcov. 
\end{align*}
Next, we show that the norm of the difference between both gradients is small. We will repeatedly use the fact that $\normHS{AB} \leq  \opnorm{A} \normHS{B}$ (and similarly $\normHS{AB} \leq  \opnorm{B} \normHS{A}$) throughout the remainder of the proof.
\begin{align}
\label{eq:norm_grad_diff}
	\normHS{\nabla \fst(X) - \nabla \fhat(X)} \leq \normHS{\Bst^\herm \Pst \Bst - \Bhat^\herm \Phat \Bhat} \opnorm{X} + \normHS{(\Bst^\herm\Pst \Ast -\Bhat^\herm\Phat \Ahat) \expcov}.
\end{align}
Bounding the first term in the above decomposition,
\begin{align*}
		\normHS{\Bhat^\herm \Phat \Bhat - \Bst^\herm \Pst \Bst} &\overset{(i)}{\le} \normHS{\Bhat^\herm \Phat \Bhat - \Bst^\herm \Phat \Bst} + \normHS{\Bst^\herm (\Phat - \Pst) \Bst} \\
		&\leq\normHS{\Bhat^\herm \Phat \Bhat \pm \Bhat^\herm \Phat \Bst- \Bst^\herm \Phat \Bst} +  \eps \\
		&\leq \normHS{\Bhat^\herm \Phat(\Bhat - \Bst)} + \normHS{(\Bhat - \Bst)\Phat\Bst} + \eps\\
		& \leq  4M^{3/2}\epsilon,
\end{align*}
where in the last line, we used $\|\Bhat\|_{\op} \le \epsilon + \|\Bst\|_{\op} \le 2\sqrt{M}$, since $\epsilon \le 1$ and $M \ge \max\{1,\|\Bst\|_{\op}^2\}$. In inequality $(i)$, we used the following calculation:
\begin{align*}
	\normHS{\Bst^\herm (\Phat - \Pst) \Bst}^2 
	& = \traceb{\Bst^\herm (\Phat - \Pst) \Bst \Bst^\herm (\Phat - \Pst) \Bst} \\
	& \leq \traceb{\Bst^\herm (\Phat - \Pst) \expcovone (\Phat - \Pst) \Bst} \tag{$\Bst \Bst^\herm \preceq \expcovone$}\\
	& =  \traceb{\expcov (\Phat - \Pst) \Bst \Bst^\herm (\Phat - \Pst) \expcov} \\
	& \leq \traceb{\expcov (\Phat - \Pst) \expcovone (\Phat - \Pst) \expcov}.
\end{align*}
The last line is equal to $\normHS{\expcov (\Phat - \Pst) \expcov}^2$ which is less than $\epsilon^2$ by assumption. Next, we bound the second term in \Cref{eq:norm_grad_diff},
\[
\normHS{(\Bhat^\herm \Phat \Ahat - \Bst^\herm \Pst \Ast)\expcov} \leq \underbrace{\normHS{(\Bhat^\herm \Phat \Ahat - \Bst^\herm\Phat \Ast) \expcov}}_{T_1} + \underbrace{\normHS{\Bst^\herm(\Phat -  \Pst) \Ast\expcov}}_{T_2}.
\]
We bound $T_1$ as follows:
	\begin{align*}
		\normHS{(\Bhat^\herm \Phat \Ahat \pm \Bhat^\herm \Phat \Ast - \Bst^\herm\Phat \Ast) \expcov}
		&\leq \normHS{\Bhat^\herm \Phat (\Ahat - \Ast)\expcov} + \normHS{(\Bhat - \Bst)^\herm\Phat \Ast \expcov}  \\
		&\leq \eps \opnorm{\Bhat \Phat} + \eps\opnorm{\Phat \Ast \expcov}  \\
		&\leq \left(2M^{3/2} +  M^{3/2}\opnorm{\expcov}\right) \eps,
	\end{align*}
	where again we used $\|\Bhat\| \le \sqrt{2M}$. Before bounding $T_2$, we observe that:
\begin{align}
\Ast \expcovone \Ast^\herm & \preceq 2 (\Ast + \Bst \Kinit) \expcovone (\Ast+ \Bst \Kinit)^\herm + 2(\Bst \Kinit) \expcovone (\Bst \Kinit)^\herm \nn\\
& \preceq 2 \expcovone + 2 \opnorm{\expcovone} \opnorm{\Kinit}^2 \Bst \Bst^\herm \nn \\
& \preceq 2(1 + \opnorm{\expcovone} \opnorm{\Kinit}^2) \expcovone.
\label{eq:gamma_upper_bound}
\end{align}
To go from the first to the second line, we have used the fact that since 
\begin{align*}
\expcovone = \dlyap{(\Ast + \Bst \Kinit)^\herm}{\Bst \Bst^\herm \uvar + \covw},
\end{align*} by definition of $\mathsf{dlyap}$,
\begin{align*}
(\Ast + \Bst \Kinit) \expcovone (\Ast+ \Bst \Kinit)^\herm = \expcovone - \Bst \Bst^\herm \uvar - \covw \preceq \expcovone.
\end{align*}
Applying \eqref{eq:gamma_upper_bound},  we can now bound $T_2$ using a similar calculation as before,
\begin{align*}
\normHS{\Bst^\herm(\Phat -  \Pst) \Ast\expcov}^2 & = \traceb{\Bst^\herm (\Phat -  \Pst) \Ast \expcovone \Ast^\herm (\Phat -  \Pst)\Bst} \\
& \leq  2(1+ \opnorm{\Kinit}^2 \opnorm{\expcovone}) \traceb{\Bst^\herm (\Phat -  \Pst)  \expcovone  (\Phat -  \Pst)\Bst} \\
& =  2(1+ \opnorm{\Kinit}^2 \opnorm{\expcovone})  \normHS{\expcov(\Phat - \Pst) \expcov}^2 \\ 
& \leq   2(1+ \opnorm{\Kinit}^2 \opnorm{\expcovone})  \epsilon^2.
\end{align*} 
Returning to \Cref{eq:norm_grad_diff},
\begin{align*}
\norm{\nabla \fst(X) - \nabla_\matu \fhat(X)} &\leq 5M^{3/2} \eps \opnorm{X} \\
&\qquad+ \left(2M^{3/2} +  M^{3/2}\opnorm{\expcov} + \sqrt{2(1+ \opnorm{\Kinit}^2 \opnorm{\expcovone})}   \right) \eps,
\end{align*}
and for $X = X_\star = \Kst \expcov$, we have that,
\begin{align*}
\norm{\nabla \fhat(X_\star)} &\leq 4M^{3/2} \epsilon \opnorm{\Kst \expcov } \\
&+ \left(2M^{3/2} +  M^{3/2}\opnorm{\expcov} + \sqrt{2(1+ \opnorm{\Kinit}^2 \opnorm{\expcovone})}   \right) \eps.
\end{align*}
We now simplify the above. By \Cref{lem:Pk_bound}, we can take $\opnorm{\Kst} \le \opnorm{\Pst}^{1/2} \le M^{1/2}$, and $\opnorm{\Kinit} \le \opnorm{P_{K_0}}^{1/2} \le M^{1/2}$, where $P_{K_0} = \PinftyK{K_0}{\Ast}{\Bst}$. Moreover, \Cref{lemma:lyapunov_series_bound} yields $\opnorm{\expcovone} \le \opnorm{P_{K_0}}^2(\opnorm{\Sigmaw} + \sigmau^2 \opnorm{\Bst}^2) \le \uvar M^3$ (since $\uvar \ge 1$). Hence, 
\begin{align*} 
\norm{\nabla \fhat(X_\star)} \le 4M^{7/2} \epsilon + (2M^{3/2} +\sigmau M^{5/2} + \sqrt{2(1 + \uvar M^4)})\epsilon \le 9\sigmau M^4 \epsilon. 
\end{align*}
Lastly, by \Cref{lemma:strong_convexity_closeness},
\[
\normHS{X_\star - \widehat{X}} = \normHS{(\Kst - \Khat) \expcov} \leq 9\sigmau M^4 \eps  \frac{1}{\sigma_{\min}(R)}.
\]
The precise statement follows by applying our assumption that $\sigma_{\min}(R) > 1$.
\end{proof}

\subsection{Proof of Operator Norm $P$-Perturbation:  \Cref{prop:uniform_perturbation_bound} \label{sec:uniform_perturbation_proof}}

	We begin the proof with the following lemma, which ensures that the Lyapunov function of a stable matrix $A_1$ is also a Lyapunov function for a sufficiently nearby matrix $A_2$:
	\begin{lemma}\label{lem:lyapunov_series_pert_ub} Let $A_1,A_2$ be two matrices with $A_1$ stable. Set $P_1 = \dlyap{A_1}{\Sigma}$, where $\Sigma \succeq I$. Fix an $\alpha \in (0,1)$, and suppose that $\|A_1 - A_2\|_{\op}^2 \le \frac{\alpha^2}{16 \opnorm{P_1}^3}$. Then, $A_2^\herm P_1 A_2 \preceq P_1 (1 - \frac{1 - \alpha}{\opnorm{P_1}})$, and iterating,
	\begin{align*}
	(A_2^\herm)^j P_1 A_2^j \preceq P_1 (1 - \frac{1 - \alpha}{\opnorm{P_1}})^j, \quad \forall j \ge 0.
	\end{align*}
	\end{lemma}
		\begin{proof} [Proof of \Cref{lem:lyapunov_series_pert_ub}]
		From \Cref{lemma:lyapunov_series_bound}, $A_1^\herm P_1 A_1 \preceq P_1 ( 1- \opnorm{P_1}^{-1})$. Set $\Delta = A_2 - A_1$. For any $\tau > 0$, invoking \Cref{lemma:psd_am_gm},
		\begin{align*}
		A_2^\herm P_1 A_2 &\preceq (A_1 + \Delta)^\herm P_1 (A_1 + \Delta) \\
		&=  A_1^\herm P_1 A_1 + \Delta^\herm P_1 \Delta + \Delta^\herm P_1 A_1 + A_1^\herm P_1 \Delta\\
		&\preceq  (1 +  \tau) A_1^\herm P_1 A_1 + \left(1 +\frac{1}{\tau}\right)\Delta^\herm P_1 \Delta\\
		&\preceq  (1 + \tau) P_1 (1 - \opnorm{P_1}^{-1}) + \left(1 +\frac{1}{\tau}\right)\|\Delta\|_{\op}^2 \opnorm{P_1}\\
		&\preceq   P_1 \left\{(1 - \opnorm{P_1}^{-1})(1 + \tau) + \left(1 +\frac{1}{\tau}\right)\|\Delta\|_{\op}^2 \opnorm{P_1}\right\},\\
		\end{align*}
		where the last line follows from the fact that $P_1 \succeq I$ since $Q \succeq I$. Using our assumption that $\|\Delta\|_{\op}^2 \le \frac{\alpha^2}{16 \opnorm{P_1}^3}$ and optimizing over $\tau$, a short calculation shows that the expression between brackets above is bounded by $(1 - (1 - \alpha)/\opnorm{P_1})$.		
\end{proof}
	
	The next proposition provides a perturbation bound for the $\mathsf{dlyap}$ operator: 
	\begin{proposition}\label{prop:dlyap_pert}  Let $A_1$ be a linear operator, $\Sigma \succeq I$, and $P_1=  \dlyap{A_1}{\Sigma}$. Define $\|\cdot\|_{\circ}$ to be any spectral norm \text{\normalfont (i.e. $\circ \in \{\mathrm{op},\mathsf{HS},\mathsf{tr}\}$)}. Then, for any $\alpha \in [0,1)$, linear operator $A_2$ with $\|A_1 - A_2\|_{\op}^2 \le \frac{\alpha^2}{16\opnorm{P_1}^3}$, and any symmetric $\Sigma_0$, $\dlyap{A_2}{\Sigma_0}$ is a bounded operator, and  
	\begin{align*}
	\|\dlyap{A_2}{\Sigma_0} - \dlyap{A_1}{\Sigma_0}\|_{\circ} \le 2\calC_{\circ}\opnorm{A_1 - A_2}\opnorm{P_1}^{7/2} (1- \alpha)^{-2},
	\end{align*}
	where $\calC_{\circ} = \|P_1^{-1/2}\Sigma_0 P_1^{-1/2}\|_{\circ}$.
	\end{proposition}

		\begin{proof}[Proof of \Cref{prop:dlyap_pert}]
		Set $\Delta = A_1 - A_2$, and define the terms $E_n \defeq P_1^{1/2}(A_1^n - A_2^n)$, $\gamma \defeq 1 - \frac{1 - \alpha}{\opnorm{P_1}}$. 
		Since $A_1, A_2$ satisfy the necessary closeness conditions, by \Cref{lemma:lyapunov_series_bound,lem:lyapunov_series_pert_ub}, we  have that, 
		\begin{align}
		\max\{\|P_1^{1/2}A_1^n\|_{\op}, \|P_2^{1/2}A_1^n\|_{\op}\}\le \sqrt{\|P_1\|_{\op}\gamma^n} \label{eq:gamma_n_bound}.
		\end{align}
Next, by closed form expression for $\mathsf{dlyap}$, 
		\begin{align*}
		\dlyap{A_2}{\Sigma_0} - \dlyap{A_1}{\Sigma_0} &= \sum_{n \ge 0} (A_2^n)^\herm \Sigma_0 A_2^n -  (A_1^n)^\herm \Sigma_0 A_1^n \\
		&= \sum_{n \ge 1} (A_2^n - A_1^n)^\herm\Sigma_0 A_2^n +  A_1^n \Sigma_0 (A_2^n - A_1^n)\\
		&= -\sum_{n \ge 1} E_n^\herm P_1^{-1/2}\Sigma_0 A_2^n +  (A_1^n)^\herm \Sigma_0  P_1^{-1/2}E_n, 
		\end{align*}
		and thus,
		\begin{align*}
		\|\dlyap{A_2}{\Sigma_0} - \dlyap{A_1}{\Sigma_0}\|_{\circ} &\le \sum_{n \ge 1} \|E_n\|_{\op} \left(\|P_1^{-1/2}\Sigma_0 A_2^n\|_{\circ} +  \|(A_1^n)^\herm \Sigma_0  P_1^{-1/2}\|_{\circ}\right)\\
		&\le \sum_{n \ge 1} \|E_n\|_{\op} \|P_1^{-1/2}\Sigma_0 P_1^{-1/2}\|_{\circ} \left(\|P_1^{1/2}A_2^n\|_{\op} +\|P_1^{1/2}A_1^n\|_{\op} \right)\\
		&\le \underbrace{\|P_1^{-1/2}\Sigma_0 P_1^{-1/2}\|_{\circ}}_{\defeq \calC_{\circ}} \sum_{n \ge 1} \|E_n\|_{\op} \left(\|P_1^{1/2}A_2^n\| +\|P_1^{1/2}A_1^n\| \right) \\
		&\le 2 \calC_{\circ}\sum_{n \ge 1} \|E_n\|_{\op} \sqrt{\opnorm{P_1} \gamma^n} . \tag{\Cref{eq:gamma_n_bound}}
		\end{align*}
		Next, we bound $\|E_n\|_{\op}$. Using the identity, $A_1^n - A_2^n = \sum_{i=0}^n A_1^i \Delta A_2^{n-i-1}$,	 
		\begin{align*}
		\|E_n\|_{\op} &\le \sum_{i=0}^n \|P_1^{1/2} A_1^i\|_{\op}\|\Delta\|_{\op}\| A_2^{n-i-1}\|_{\op}\\
		&\le \sum_{i=0}^n \|P_1^{1/2} A_1^i\|_{\op}\|\Delta\|_{\op}\|P_{1}^{1/2} A_2^{n-i-1}\|_{\op}, \tag{$P_1 \succeq I$}\\
		&\le \|\Delta\|_{\op}\|P_1\|_{\op}\sum_{i=0}^n \sqrt{ \gamma^{n-1} }   \tag{\Cref{eq:gamma_n_bound}}.
		\end{align*}
		Hence, combining the above,
		\begin{align*}
		\|\dlyap{A_2}{\Sigma_0} - \dlyap{A_1}{\Sigma_0}\|_{\circ} &\le 2\calC_{\circ}\opnorm{\Delta}\opnorm{P_1}^{3/2} \sum_{n \ge 1} n\gamma^{n - 1/2} \\
		&\le 2\calC_{\circ}\opnorm{\Delta}\opnorm{P_1}^{3/2} \sum_{n \ge 1} n\gamma^{n - 1} \tag*{($\gamma \le 1$)} \\
		&= 2\calC_{\circ}\opnorm{\Delta}\opnorm{P_1}^{3/2} (1-\gamma)^{-2}.\\
		\end{align*}
		Substituting in $\gamma = (1 - \frac{1 - \alpha}{\opnorm{P_1}})$, the above becomes $2\calC_{\circ}\opnorm{\Delta}\opnorm{P_1}^{7/2} (1- \alpha)^{-2}$, concluding the proof.
		\end{proof}
We may now conclude the proof of \Cref{prop:uniform_perturbation_bound}
		\begin{proof}[Proof of \Cref{prop:uniform_perturbation_bound}]
	We prove the each part of the proposition individually.
	\paragraph{Part 1}

	 Define $\Sigma_1 := K_1^\top R K_1 + Q$. Then,
	\begin{align}
	P_{2} &\preceq \PinftyK{K_1}{A_2}{B_2} = \dlyap{A_2 + B_2 K_1}{\Sigma_1} \nonumber \\
	&\preceq \dlyap{A_1 + B_1 K_1}{\Sigma_1} + I \cdot \|\dlyap{A_1 + B_1 K_1}{\Sigma_1} - \dlyap{A_2 + B_2 K_1}{\Sigma_1}\|_{\op} \nonumber\\
	&= P_1 + I \cdot \|\dlyap{A_1 + B_1 K_1}{\Sigma_1} - \dlyap{A_2 + B_2 K_1}{\Sigma_1}\|_{\op} \label{eq:P2st_bound}
	\end{align}
	If $\epsilon_0^2 := \|A_1 + B_1 K_1 - ( A_2 + B_2 K_1)\|_{\op}^2 \le \frac{1}{16\opnorm{P_1}^3}$, then, invoking \Cref{prop:dlyap_pert} with $\Sigma_0 \gets \Sigma_1$, $\|\cdot\|_{\circ} \gets \opnorm{\cdot}$, $\alpha \gets 1/2$, and noting how $\calC_{\circ} \le 1$, we can conclude that,
	\begin{align*}
	\|\dlyap{A_1 + B_1 K_1}{\Sigma_1} - \dlyap{A_2 + B_2 K_1}{\Sigma_1}\|_{\op} \le 8\epsilon_0\opnorm{P_1}^{7/2}.
	\end{align*}
	Observe that if $\epsilon_0 \le \eta/(8\opnorm{P_1}^{5/2})$ for some $\eta \in (0,1]$, then (since $\opnorm{P_1} \ge 1$ by \Cref{lem:Pk_bound}) it holds that $\epsilon_0^2 \leq \frac{1}{64\opnorm{P_1}^3}$. Plugging into the inequality above, we get our desired result,
	\[
	P_2\preceq P_1 + \eta \opnorm{P_1} \cdot I.
	\]
	Therefore, to finish the proof of Part 1, we only need to verify that, under the assumptions of the proposition, $\epsilon_0 \le \eta/(8\opnorm{P_1}^{5/2})$. By the definition of $\epsilon_0$ and \Cref{lem:Pk_bound}, 
	\begin{align*}
	\epsilon_0 \le \|A_1 - A_2\|_{\op} +   \opnorm{K_1}\|B_1 - B_2\|_{\op} \le (1+ \opnorm{K_1})\epsilon_{\op}\le 2\opnorm{P_1}^{1/2}\epsilon_{\op}.
	\end{align*}
Since, $\epsilon_{\op} \le \eta/(16\opnorm{P_1}^{3})$, this calculation above proves that $\epsilon_0 \le \eta/(8\opnorm{P_1}^{5/2})$.

	\paragraph{Part 2} 

	Fix a parameter $\eta_0 \in (0,1)$ to be chosen, and let $\eta$ be as in the theorem statement. Switching the roles of the indices $i = 1$ and $i = 2$ in the first part of the proposition, we find that if  $ \epsilon_{\op} \le \eta_0/(16\opnorm{P_2}^{3})$, we can establish the following PSD inequalities: 
		\begin{align}
		P_1 \preceq  \PinftyK{K_2}{A_1}{B_1} \preceq P_2 + \eta_0 \|P_2\|_{\op} I.
		\end{align}
	If in addition $ \epsilon_{\op} \le \eta/(16\opnorm{P_2}^{3})$, then we have $\|P_2\|_{\op} \le (1+\eta)\|P_1\|_{\op}$. Therefore, if $\epsilon_{\op} \le \eta_0/(16(1+\eta)^3\opnorm{P_1}^{3})$ (replacing $P_2$ in the denominator by $P_1$), then $P_1 \preceq P_2 + \eta_0(1+\eta)\|P_1\|_{\op} I$. 

	Selecting $\eta_0 \gets \eta/(1+\eta)$, we have that if $\epsilon_{\op} \le \eta/(16(1+\eta)^4\opnorm{P_2}^{3})$, $P_1 \preceq P_2 + \eta\|P_1\|_{\op} I$. Hence, combining with Part $1$, we have that $\|P_1 - P_2\|_{\op} \le \eta \|P_1\|_{\op}$.
	\end{proof}

\subsection{Proof of Weighted $P$-Perturbation:  \Cref{theorem:P_perturbation} \label{sec:cov_pert}}

Our argument is based on the self-bounding ODE method introduced by \cite{simchowitz2020naive}, where the perturbation bound is derived from considering an interpolating curve $A(t),B(t)$ between the ground-truth instances $(\Ast,\Bst)$ and estimated instance $(\Ahat,\Bhat)$. 

	\begin{definition}[Interpolating Curves] Given estimates $(\Ahat,\Bhat)$ of $(\Ast,\Bst)$, for $t \in [0,1]$ we define, 
	\begin{align*}
		\At  \defeq \Ast + t(\Ahat - \Ast), \quad \Bt  \defeq \Bst + t(\Bhat - \Bst).
	\end{align*}
	In addition, we define the optimal controller $K(t)$, closed-loop matrix $\Acl(t)$, and value function $P(t)$ as
	\begin{align*}
		K(t) \defeq \Kinfty{A(t)}{B(t)},\quad
		\Acl(t) \defeq A(t) + B(t)K(t),\quad
		P(t)  \defeq \Pinfty{\At}{\Bt}.
	\end{align*}
	Finally, define the following error term in the closed loop matrix
	\begin{align*}
		\Delta_{\Acl}(t) &\defeq \Ahat - \Ast + (\Bhat - \Bst)K(t).
	\end{align*}
	\end{definition}

	At the core of the technique is verifying that the instances along the curve $(A(t),B(t))$ remain stabilizable, so that the operators $K(t), P(t)$ are well defined. The following guarantee, established in \Cref{sec:lem:unif_p_bound}, ensures that this is the case.
	\begin{lemma}
	\label{lemma:uniform_bound_p}Assume that $(\Ahat,  \Bhat)$ satisfy \Cref{cond:unif_close}. Then:
	\begin{enumerate}
		\item The instances $(A(t),B(t))$ are stabilizable along $t \in [0,1]$, and $\sup_{t \in [0,1]}\opnorm{P(t)} \leq 1.1 \opnorm{\Pst}$
		\item $\sup_{u, t \in [0,1]}\opnorm{\PinftyK{K(t)}{A(u \cdot t)}{B(u \cdot t)}} \leq 1.2 \opnorm{\Pst}$ 
	\end{enumerate}
	\end{lemma}

	Since instances $(A(t),B(t))$ are stabilizable along $t \in [0,1]$, \Cref{lemma:p_derivative} ensures that  $P'(t)$ is \emph{Frechet differentiable} in the operator norm, and its derivative is
	\begin{align}
	\frac{\rmd}{\rmd t}P(t) &=\dlyap{\Acl(t)}{E(t)}, \label{eq:P_t_der_defn}\\
	&\text{where } E(t)  \defeq \Acl(t)^\herm P(t) \Delta_{\Acl}(t) + \Delta_{\Acl}(t)^\herm P(t) \Acl(t). \label{eq:E_t_defn}
	\end{align}
	The formal definition of the Frechet derivative is deferred to \Cref{defn:frechet_differentiable}; in this section, we shall only use it as a condition to call relevant lemmas. The heart of proposition is based off the following observation which follows from part (c) of  \Cref{lem:derivative_and_trace}:
	\begin{align*}
	\normHS{\expcov (\Pst - \Phat) \expcov}  = \normHS{\expcov (P(0) - P(1)) \expcov} \le \sup_{t \in [0,1]}\normHS{\expcov P'(t) \expcov}.
	\end{align*}
	To prove the main result, it remains to show the following bound on the derivative:
	\begin{align}
	\forall t \in [0,1], \quad \normHS{\expcov P'(t) \expcov} \le \calC_{P} \cdot  \epsgam\cdot \sqrt{\log_+\left(\kappagam\right)} \cdot \phi(\kappagam)^{\alpha_{\op}} \label{eq:nucnorm_bound_desired}.
	\end{align}
	We verify that \Cref{eq:nucnorm_bound_desired} holds in \Cref{sec:nucnorm_bound_desired}. 
	
	\newcommand{\psichange}{\psi_{\mathrm{diff}}}
	A key step along the way is the following ``change of system'' lemma, which allows us to bound the norm of the error in a covariance induced by one closed-loop system $A_1 + B_1 K$ by that induced by another, $A_2 + B_2 K$. This proposition is, in some sense, the mirror image of the change-of-covariance theorems in \Cref{sec:change_of_covariance}. In those results, we keep the system fixed but change the controller. Here, we keep the controller fixed, but change the system. 
	\begin{proposition}
	\label{prop:change_of_system} Let $(A_1,B_1)$ and $(A_2,B_2)$ be two systems, $K$ a controller, and define for $u \in [0,1]$:
	\begin{align*}
		\Abar(u) \defeq A_1 + u (A_2 - A_1) \quad \Bbar(u)  \defeq B_1 + u (B_2 - B_2), \quad \Aclbar(u) \defeq \Abar(u) + \Bbar(u) K.
	\end{align*}
	Assume that $K$ stabilizes all instances $(\Abar(u),\Bbar(u))$, so that
	 \begin{align*}
	 \Mbar \defeq \max_{u \in [0,1]} \opnorm{\Pbar(u)} < \infty,\quad \text{ where } \Pbar(u) \defeq \dlyap{\Aclbar(u)}{Q + K^\herm R K}.
	 \end{align*}
	Finally, define $\DelbarAcl  \defeq  \frac{d}{du} \Aclbar(u) =  (A_2 - A_1) + (B_2 - B_1) K$, let $\Sigma \succeq 0$ be a trace class operator, and define the error terms:
	\begin{align*}
	\epssysnum{i}^2 := \traceb{\DelbarAcl ^\herm  \DelbarAcl  \dlyap{(A_i + B_i K)^\herm}{\Sigma}}, \quad i \in \{1,2\}.
	\end{align*}
	Then, setting the operator $\brackalign{X,Y} \defeq \max\{\nucnorm{XWY} : \|W\|_{\op} = 1\}$ the following bound holds,
	 \begin{align*}
	 & \epssysnum{2}^2 \le  \psichange(\epssysnum{1}) \cdot\epssysnum{1}^2\\
	 &\quad \text{where } \psichange(\epsilon) \defeq 2\exp\left( \frac{3}{2}\Mbar \opnorm{\DelbarAcl} \sqrt{\log_+ \left( \frac{e\cdot \brackalign{\DelbarAcl^\herm \DelbarAcl,\Sigma}\Mbar^{3}}{\epsilon^2} \right)}\right).
	 \end{align*}
	 Moreover, these relationships hold for any $\epsilon \ge \epssysnum{1}$.
	\end{proposition}
	The proof relies on a careful application  of the self-bounding ODE method. The full proof is given in \Cref{sec:proof_prop_change_of_system}. 
	\subsubsection{Establishing \Cref{eq:nucnorm_bound_desired}\label{sec:nucnorm_bound_desired}}
	\begin{proof}
	Our proof proceeds in multiple steps. First, we majorize the weighted derivative $\expcov P'(t) \expcov$ in terms of two PSD operators $Y_1(t),Y_2(t)$; a careful analysis shows it suffices to bound the operator norm of $Y_1(t)$ and the trace of $Y_2(t)$. The term $Y_1(t)$ is straightforward to control; the term $Y_2(t)$ requires appeal to the change-of-system error bound in \Cref{prop:change_of_system}. The proof concludes with an application of one of our change-of-covariance bounds (\Cref{corollary:change_of_controller}), and some further simplifications. 

	\paragraph{Majorization by $Y_1(t), Y_2(t)$} Define for all $t \in [0,1]$ and $\alpha > 0$ the quantities: 
	\begin{align*}
	E_1(t) := \Acl(t)^\herm P(t) \Acl(t), \quad E_2(t) := \Delta_{\Acl}(t)^\herm P(t) \Delta_{\Acl}(t), \quad E_{[\alpha]}(t) := \frac{\alpha}{2} E_1(t) + \frac{1}{2\alpha}E_2(t).
	\end{align*}

	By the AM-GM inequality in \Cref{lemma:psd_am_gm}, $E(t) \preceq E_{[\alpha]}(t)$ for all $t \in [0,1]$ and $\alpha > 0$ (recall that $E(t)$ is defined in \Cref{eq:E_t_defn}).  Since the solution of the Lyapunov equation preserves PSD order (\Cref{lemma:basic_lyapunov_facts}),
	\begin{align*}
	 \underbrace{-\expcov \dlyap{\Acl(t)}{E_{[\alpha]}(t)} \expcov}_{\defeq -Y_{[\alpha]}(t)} \preceq {\expcov P'(t) \expcov} \preceq \underbrace{\expcov \dlyap{\Acl(t)}{E_{[\alpha]}(t)} \expcov}_{\defeq Y_{[\alpha]}(t)}.
	\end{align*}
Now, observe that since $\dlyap{\cdot}{\cdot}$ is linear in its second argument, we can express
	\begin{align*}
	Y_{[\alpha]}(t) &= \frac{\alpha}{2}Y_1(t) + \frac{1}{2\alpha}Y_2(t), \quad \text{where}\\
	&\quad Y_1(t) := \expcov \dlyap{\Acl(t)}{\Acl(t)^\herm P(t) \Acl(t)} \expcov\\
	&\quad Y_2(t) := \expcov \dlyap{\Acl(t)}{\Delta_{\Acl}(t)^\herm P(t) \Delta_{\Acl}(t)} \expcov.
	\end{align*}
	From \Cref{lemma:symmetrized_holder}, we show that the PSD domination of $\expcov P'(t) \expcov$ by the weighted linear combinations of the form $Y_{[\alpha]}(t)$ imply that 
	\begin{align}
	\normHS{\expcov P'(t) \expcov} \leq \sqrt{\|Y_1(t)\|_{\op} \cdot \trace[Y_2(t)]} \label{eq:some_inequality}.
	\end{align}
	We now proceed to bound the terms $\|Y_1(t)\|_{\op}$ and $\trace[Y_2(t)]$   individually. We let $M =  1.2 \opnorm{\Pst}$, which upper bounds $\sup_{t \in [0,1]} \opnorm{P(t)}$ by \Cref{lemma:uniform_bound_p}.

	\paragraph{Bounding $ \|Y_1(t)\|_{\op} $} By \Cref{lemma:lyapunov_series_bound} we have that $(\Acl(t)^\herm)^j P(t)\Acl(t)^j \preceq P(t) (1 - \opnorm{P(t)}^{-1})^j$. Therefore, using our uniform upper bound on $\norm{P(t)}$ from \Cref{lemma:uniform_bound_p}, the following inequality holds for $T_1$:
	\begin{align}
		  \|Y_1(t)\|_{\op}  &= \opnorm{\expcov \dlyap{\Acl(t)}{\Acl(t)^\herm P(t) \Acl(t)} \expcov} \nonumber\\ 
		 &=  \opnorm{\expcovone \sum_{j=0}^\infty \left(\Acl(t)^\herm\right)^{j+1} P(t) \Acl(t)^{j+1}} \nonumber\\
		  &\leq  \opnorm{\expcov P(t) \expcov} \sum_{j=0}^\infty (1 - \opnorm{P(t)}^{-1})^{j+1} \nonumber\\
		& \leq  \opnorm{\expcovone} \opnorm{P(t)}^2 \nonumber\\
		& \leq  \opnorm{\expcovone} \maxp^2. \label{eq:some_inequality_two}
	\end{align}

	\paragraph{Bounding $\traceb{Y_2(t)}$ via change of system}  Commuting traces and rewriting $\sum_{j=0}^\infty \Acl(t)^j \expcovone
	\left(\Acl(t)^\herm\right)^j$ as $\dlyap{\Acl(t)^\herm}{\expcovone}$, we have 
	\begin{align}
		\traceb{Y_2(t)} &=  \traceb{\expcov \dlyap{\Acl(t)}{\Delta_{\Acl}(t)^\herm P(t) \Delta_{\Acl}(t)} \expcov} \nonumber \\
		&=  \traceb{ \dlyap{\Acl(t)^\herm}{\expcovone}\Delta_{\Acl}(t)^\herm P(t) \Delta_{\Acl}(t) } \nonumber \\
		&\leq \maxp \cdot \traceb{  \Delta_{\Acl}(t)  \dlyap{\Acl(t)^\herm}{\expcovone}\Delta_{\Acl}(t)^\herm}.\label{eq:t2_riccati}
	\end{align}
	In the last line, we used that for any $X \succeq 0$, and operators $X,A$, $\traceb{X A^\herm Y A} = \traceb{AXA^\herm Y} \le \opnorm{Y} \traceb{AXA^\herm}$. 

	Note however that our estimation guarantees hold under the true system $\Ast,\Bst$, whereas the above bound holds under a closed loop system involving the matrices $A(t),B(t)$. To this end, we invoke the change of system guarantee, \Cref{prop:change_of_system}.  We instantiate \Cref{prop:change_of_system} with the following substitutions:
	\newcommand{\epsalign}{\epsilon_{\mathrm{algn}}}
	\begin{itemize}
		\item Take $K \gets K(t)$, $(A_1,B_1) \gets (A(t),B(t))$, and $(A_2,A_2) \gets (\Ast,\Bst)$.
		\item By \Cref{lemma:uniform_bound_p}, $\Mbar$ can be upper bounded by $M \defeq 1.2\|\Pst\|_{\op} \defeq 1.2M_{\star}$.
		\item $\DelbarAcl =  \Ast - A(t) + (\Bst - B(t)) K(t)) = - t\Delta_{\Acl}(t)$. Hence, bounding $\sup_{t \in [0,1]} \|K(t)\|_{\op} \le \sqrt{M}$:
		\begin{align}
		\|\DelbarAcl \|_{\op} &\le \|\Delta_{\Acl}(t)\|_{\op} \le \|\Ast - A(t)\|_{\op} + \|\Bst - B(t)\|_{\op}\|K(t)\|_{\op} \nonumber\\
		&\le \epsunifop(1 + \|K(t)\|_{\op}) \le \epsunifop(1 + \sqrt{M}) \le 2 \sqrt{M} \epsunifop  \le 2.4 \sqrt{\Mpst}\label{eq:Deltop_ub}.
		\end{align}
		\item Define the aligned error,
		\begin{align} 
		\epsalign^2 := \max_{t \in [0,1]}\brackalign{\Delta_{\Acl}(t)^\herm \Delta_{\Acl}(t), \expcovone} .\label{eq:epsalign}
		\end{align}
		\item With these substitutions, we can take
		\begin{align}
		\psi(\epsilon) \defeq 2\exp\left( 3 \cdot 1.2^{3/2} \cdot  \Mpst^{3/2} \epsunifop \sqrt{\log_+ \left(1.2^3 e\Mpst^{3} \cdot \frac{  	\epsalign^2  }{\epsilon^2} \right)}\right) \label{eq:psideff_eps}.
		\end{align}
	\end{itemize}
	Applying these substitutions,  \Cref{prop:change_of_system} entails
	\begin{align}
	\traceb{Y_2(t)} &\le M \traceb{\Delta_{\Acl}(t)^\herm  \Delta_{\Acl}(t)  \dlyap{(A(t) + B(t) K(t))^\herm}{\expcovone}} \le M \psi(\epsilon) \epsilon, \quad \forall \epsilon \ge T_3  \label{eq:tracebY2}\\
	&\quad \text{ where } T_3 \defeq \traceb{\Delta_{\Acl}(t)^\herm  \Delta_{\Acl}(t)  \dlyap{(\Ast + \Bst K(t))^\herm}{\expcovone}} \label{eq:T3_deff}.
	\end{align}

	\paragraph{Bounding $T_3$ via change of controller }
	Focusing on the remaining trace term  $T_3$, by \Cref{corollary:change_of_controller} we can switch the controller $K(t)$ inside the Lyapunov operator to $\Kinit$ and pay a multiplicative constant,
	\begin{align}
	T_3 &= \traceb{\Delta_{\Acl}(t)^\herm  \Delta_{\Acl}(t)  \dlyap{(\Ast + \Bst K(t))^\herm}{\expcovone}} \nonumber\\
	& \leq \calC_K  \cdot \traceb{\Delta_{\Acl}(t)^\herm  \Delta_{\Acl}(t)  \dlyap{(\Ast + \Bst \Kinit)^\herm}{\expcovone}}, \label{eq:T3_first_bound}
	\end{align}
	where, for $P_0 = \PinftyK{\Kinit}{\Ast}{\Bst}$ and $\Mpnot := \opnorm{P_0}$, we define
	\begin{align*}
	\calC_K \defeq 2\left(1 + \frac{64 \max_{t \in [0,1]}\opnorm{\Kinit - K(t)}^2}{\sigma^2_\matu} \opnorm{\expcovone} \Mpnot^3 \log(2 \Mpnot)^2 \right).
	\end{align*}

	Next, we recall that $\expcovone = \dlyap{(\Ast + \Bst \Kinit)^\herm}{\Linit}$ where $\Linit = \Bst \Bst \uvar + \covw$ is equal to the steady-state covariance induced by the initial controller $K_0$. By \Cref{lemma:repeated_dlyap}, we can rewrite the solution to Lyapunov equation as,
	\begin{align*}
	\dlyap{(\Ast + \Bst\Kinit)^\herm}{\expcovone} = \dlyapm{(\Ast + \Bst\Kinit)^\herm}{\Linit}{1}.
	\end{align*}
By applying \Cref{lemma:dlyapm_bound_trace}, and recalling that $\Mpnot=\opnorm{P_0}$,$\brackalign{X,Y} := \max\{\nucnorm{XWY} : \|W\|_{\op} = 1\}$, we can upper bound $\traceb{\Delta_{\Acl}(t)^\herm  \Delta_{\Acl}(t)  \dlyapm{(\Ast + \Bst \Kinit)^\herm}{\Linit}{1}}$ as follows,
	\begin{align*} 
	& \le n \cdot \traceb{\Delta_{\Acl}(t)^\herm  \Delta_{\Acl}(t)  \dlyap{(\Ast + \Bst \Kinit)^\herm}{\Linit}}  
	+ (n+1)  \brackalign{\Delta_{\Acl}(t)^\herm\Delta_{\Acl}(t),\Linit}\|P_0\|_{\op}^3 \exp(- \|P_0\|_{\op}^{-1}n)\\
	&= n \cdot \traceb{\Delta_{\Acl}(t)^\herm  \Delta_{\Acl}(t)  \expcovone} + (n+1) \brackalign{\Delta_{\Acl}(t)^\herm\Delta_{\Acl}(t),\Linit} \Mpnot^3 \exp(- \Mpnot^{-1} n)\\
	&\le n \cdot \traceb{\Delta_{\Acl}(t)^\herm  \Delta_{\Acl}(t)  \expcovone}  + (n+1) \epsalign^2 \Mpnot^3 \exp(- \Mpnot^{-1} n).
	\end{align*}
	In the last inequality, we used that by \Cref{lem:brackalign_domination}, $ \brackalign{\Delta_{\Acl}(t)^\herm\Delta_{\Acl}(t),\expcovone} \le \epsalign^2 $  where $\epsalign^2$ is defined in \Cref{eq:epsalign}. Next, by AM-GM and the fact that $\traceb{XY} \le \traceb{X'Y}$ for $0 \preceq X \preceq X'$ and $Y \succeq 0$,
	\begin{align}
	\traceb{\Delta_{\Acl}(t)^\herm  \Delta_{\Acl}(t)  \expcovone} &= \traceb{\left(\Ahat - \Ast + K(t)(\Bhat -\Bst )\right)^{\herm } \left(\Ahat - \Ast + K(t)(\Bhat -\Bst )\right) \expcovone} \nonumber\\
	&\le 2\traceb{\left((\Ahat - \Ast)^{\herm}(\Ahat - \Ast) + K(t)^\herm (\Bhat -\Bst )^\herm (\Bhat - \Bst) K(t)\right)   \expcovone}\nonumber\\
	&= 2\|(\Ahat - \Ast) \expcov\|_{\HS}^2 + 2\| (\Bhat - \Bst) K(t) \expcov\|_{\HS}^2\nonumber\\
	&\le 2\|(\Ahat - \Ast) \expcov\|_{\HS}^2 +  2 \max_{t \in [0,1]}\|K(t)\|_{\op}^2 \| (\Bhat - \Bst)  \expcov\|_{\HS}^2\nonumber\\
	&\le 2\|(\Ahat - \Ast) \expcov\|_{\HS}^2 +  2M\| (\Bhat - \Bst)  \expcov\|_{\HS}^2\nonumber\\
	&\le 2\|(\Ahat - \Ast) \expcov\|_{\HS}^2 +  2.4 \Mpst\| \Bhat - \Bst\|_{\HS}^2  \|\expcovone\|_{\op} \defeq \epsgam^2\label{eq:traceb_stream},
	\end{align}
	where above, we used the fact that $\|K(t)\|_{\op}^2  \le \|P(t)\|_{\op}$ by \Cref{lem:Pk_bound}, which by \Cref{lemma:uniform_bound_p} is at most $M \defeq 1.2 \Mpst$ for all $t \in [0,1]$. Hence,
	\begin{align*}
	\traceb{\Delta_{\Acl}(t)^\herm  \Delta_{\Acl}(t)  \dlyapm{(\Ast + \Bst \Kinit)^\herm}{\Linit}{1}} \le n\epsgam^2  + (n+1) \epsalign^2 \Mpnot^3 \exp(- \Mpnot^{-1} n).
	\end{align*}
	To optimize over $n$, select $n = \ceil{\Mpnot \log \frac{ \Mpnot^3 \epsalign^2}{\epsgam^2}}$. Then, going back to \Cref{eq:T3_first_bound} and putting things together, we get that:
	\begin{align}
	T_3 &\le \calC_K \traceb{\Delta_{\Acl}(t)^\herm  \Delta_{\Acl}(t)  \dlyapm{(\Ast + \Bst \Kinit)^\herm}{\Linit}{1}}\nonumber\\
	&\le \calC_K(2n+1)\epsgam^2 \nonumber\\
	&\le 3 \calC_K \Mpnot\epsgam^2\log_+ \left(\frac{e \Mpnot^3 \epsalign^2}{\epsgam^2}\right) \defeq \bar{T}_3. \label{eq:T3_bar}
	\end{align}
	\paragraph{Concluding the proof}
	Combining \Cref{eq:tracebY2,eq:T3_deff,eq:T3_bar}, and $M = 1.2 \Mpst$ gives
	\begin{align}
	\max_{t \in [0,1]} \trace[Y_2(t)] &\le M \psi(\bar{T}_3) \bar{T}_3 = 3.6 \cdot \calC_K \Mpst\Mpnot\epsgam^2\log_+ \left(\frac{e \Mpnot^3 \epsalign^2}{\epsgam^2}\right)   \psi(\bar{T}_3). \label{eq:Y2_penultimate}
	\end{align}
	Now, recall the definition from \Cref{eq:psideff_eps} that
	\begin{align*}
		\psi(\epsilon) \defeq 2\exp\left( 3 \cdot 1.2^{3/2} \cdot  \Mpst ^{3/2} \epsunifop \sqrt{\log_+ \left(1.2^3 e \Mpst^{3} \cdot \frac{  	\epsalign^2  }{\epsilon^2} \right)}\right).
		\end{align*}
	Since this quantity is decreasing in $\epsilon$, and since $\bar{T}_3 \ge 3\Mpnot\calC_K \epsgam^2 \ge 6 \Mpnot \epsgam^2$ (recall $\calC_K \geq 2$),
	\begin{align*}
	\psi(\bar{T}_3) &\le \psi(3\Mpnot \calC_K\epsgam^2) \\
	&\le 2\exp\left( 3 \cdot 1.2^{3/2} \cdot  \Mpst^{3/2} \epsunifop \sqrt{\log_+ \left(1.2^3 e \Mpst^{3} \cdot \frac{  	\epsalign^2  }{6\Mpnot \epsgam^2} \right)}\right)\\
	&\le 2\exp\left( 4 \Mpst^{3/2} \epsilon_{\op} \sqrt{\log_+ \left( \frac{\Mpst^2 \epsalign^2 }{ \epsgam^2 } \right)}\right), 
	\end{align*}
	where in the last inequality, we use the simplifications $3 \cdot 1.2^{3/2} \le 4$, $1.2^3 \cdot e/6 \le 1$, and $\Mpst\le\Mpnot$, since $\Mpnot$ is the norm of a suboptimal value function. Hence, continuing from \Cref{eq:Y2_penultimate},
	\begin{align}
	\max_{t \in [0,1]} \trace[Y_2(t)] &\lesssim  \calC_K M \Mpnot\epsgam^2\log_+ \left(\frac{e \Mpnot^3 \epsalign^2}{\epsgam^2}\right)\exp\left( 4 \Mpst^{3/2} \epsilon_{\op} \sqrt{\log_+ \left( \frac{\Mpst^2 \epsalign^2 }{ \epsgam^2 } \right)}\right).
	\end{align}
	Recall from \Cref{eq:traceb_stream} that $\epsgam^2 :=  2\|(\Ahat - \Ast) \expcov\|_{\HS}^2 +  2.4 \Mpst\| \Bhat - \Bst\|_{\HS}^2  \|\expcovone\|_{\op} $. A similar computation to that used to derive the bound in \Cref{eq:traceb_stream} (this time, invoking the domination property in \Cref{lem:brackalign_domination}) lets us bound
	\begin{align*}
	\epsalign^2 &:= \max_{t \in [0,1]}\brackalign{\Delta_{\Acl}(t)^\herm \Delta_{\Acl}(t), \expcovone} \\
	&\le 2\brackalign{(\Ahat - \Ast)^\herm (\Ahat - \Ast), \expcovone} + 2.4 \Mpst\brackalign{(\Bhat - \Bst)^\herm (\Bhat - \Bst), \expcovone} \\
	&\le 2\|\Ahat -\Ast\|_{\op}^2 \traceb{\expcovone}  + 2.4 \Mpst\| \Bhat - \Bst\|_{\HS}^2  \|\expcovone\|_{\op}.
	\end{align*}
	Using the elementary inequality $\frac{a+b}{c+b} \le 1 + \frac{a}{c+b}$ for $a,b,c \ge 0$, we obtain 
	\begin{align*}
	\frac{\epsalign^2 }{\epsgam^2} \le 1 + \frac{2\|\Ahat -\Ast\|_{\op}^2 \traceb{\expcovone}}{\epsgam^2} \defeq \kappagam. 
	\end{align*}
	The above bound on $\trace[Y_2(t)] $ then simplifies to: 
	\begin{align}
	\max_{t \in [0,1]} \trace[Y_2(t)] &\lesssim  \calC_K M \Mpnot \epsgam^2\log \left(e \Mpnot^3 \kappagam\right)\exp\left( 4 \Mpst^{3/2} \epsilon_{\op} \sqrt{\log_+ \left( \Mpst^2 \kappagam \right)}\right).
	\end{align}
	Combining with \Cref{eq:some_inequality,eq:some_inequality_two} with $M = 1.2 \Mpst$ yields
	\begin{align*}
	\normHS{\expcov P'(t) \expcov} &\leq \sqrt{\|Y_1(t)\|_{\op} \cdot \trace[Y_2(t)]} \\
	&\lesssim  \sqrt{\Mpnot\Mpst^3\calC_K\|\expcovone\|_{\op} \epsgam^2 \log \left(e \Mpnot^3 \kappagam\right)}\exp\left( 2 \Mpst^{3/2} \epsilon_{\op} \sqrt{\log_+ \left( \Mpst^2 \kappagam \right)}\right).
	\end{align*}

	\paragraph{Simplifying the bound}
	Recall that $M = 1.2\|\Pst\|_{\op}$, and $\Mpst := \|\Pst\|_{\op}$. Hence, $\Mpnot \ge \Mpst$, since $P_0 \succeq P_{\star}$, as $P_{\star}$ is the optimal value function for the pair $(\Ast,\Bst)$. Hence, from \Cref{lem:Pk_bound}. $\Mpnot \ge \Mpst \ge 1$, $\opnorm{\Kinit}^2 \le \|P_0\|_{\op} \defeq \Mpnot$, and $\opnorm{K(t)}^2 \le 1.1 \opnorm{\Pst} = 1.1 \Mpst$. We note also that
	\begin{align*}
	&\opnorm{\Kinit - K(t)}^2 \le 2\opnorm{\Kinit}^2 + 2\opnorm{K(t)}^2 \le 2(\Mpnot^2 + M^2) = 2(\Mpnot^2 + 1.2^2 \Mpst^2),
	\end{align*}
	which is less than or equal to $9 \Mpnot^2$. Noting $\Mpnot \ge 1$, we can bound
	\begin{align*}
	\calC_K &\lesssim \left(1 + \frac{\Mpnot^2}{\sigma^2_\matu} \opnorm{\expcovone} \Mpnot^3 \log(2 \Mpnot)^2 \right) \\
	&\lesssim \Mpnot^5 \log(2\Mpnot)^2 \left(1 + \frac{\opnorm{\expcovone} }{\sigma^2_\matu}  \right).
	\end{align*} 
	Therefore, we get that $\normHS{\expcov P'(t) \expcov}$ is upper bounded by: 
	\begin{align*}
	 \lesssim   \sqrt{\Mpnot^{6} \Mpst^{3}\left(1 + \tfrac{\opnorm{\expcovone} }{\sigma^2_\matu}  \right)\|\expcovone\|_{\op} \log \left(e \Mpnot^3 \kappagam\right)}  \cdot \log(2\Mpnot) \cdot  \exp\left( 2 M_{\star}^{3/2} \epsilon_{\op} \sqrt{\log_+ \left( \Mpst^2 \kappagam \right)}  \right)\cdot \epsgam.
	\end{align*}
	To conclude, we bound
	\begin{align*}
	\log(2\Mpnot) \cdot  \sqrt{\log \left(e \Mpnot^3 \kappagam\right)} &\le \log(2\Mpnot)\left(\sqrt{\log \left(e \kappagam\right)} +  \sqrt{\log(\Mpnot^3)}\right)\\
	& \lesssim \Mpnot \sqrt{\log_+(\kappagam)}.
	\end{align*}
Therefore, the bound further simplifies to: 
	\begin{align*}
	& \lesssim  \Mpnot^4 \Mpst^{3/2}  \sqrt{\left(1 + \tfrac{\opnorm{\expcovone} }{\sigma^2_\matu}  \right)\|\expcovone\|_{\op}    \log_+ \left( \kappagam\right)}\exp\left( 2 M_{\star}^{5/2} \epsilon_{\op} \sqrt{\log_+(\kappagam)}  \right)\cdot \epsgam\\
	&\lesssim  \Mpnot^4 \Mpst^{3/2}  \sqrt{\left(1 + \tfrac{\opnorm{\expcovone} }{\sigma^2_\matu}  \right)\|\expcovone\|_{\op}    \log_+ \left( \kappagam\right)} \cdot \phi(\kappagam)^{2 \Mpst^{5/2} \epsilon_{\op}} \cdot \epsgam,
	\end{align*}
	where $\phi(z) = \exp(\sqrt{\log z})$.
     \paragraph{Further simplifications in finite dimensions} To conclude, we remark on how $\kappagam $ can be replaced by $1 + \mathrm{cond}(\expcovone)$, where $\mathrm{cond}(\expcovone)$ denotes the condition number, in finite dimensions. To see this, we note that the term $\opnorm{\Ahat - \Ast}^2 \traceb{\expcovone}$ in $\kappagam$ arose from an upper bound on $\brackalign{(\Ahat - \Ast)^\herm(\Ahat - \Ast),\expcovone}$. By \Cref{lem:brackalign_domination}, one can similarly bound $\brackalign{(\Ahat - \Ast)^\herm(\Ahat - \Ast),\expcovone} \le \opnorm{\expcovone}\normHS{\Ahat - \Ast}^2$. In finite dimensions with invertible $\expcovone$, we can therefore compute: 
     \begin{align*}
     \normHS{\Ahat - \Ast}^2 \le \lambda_{\min}(\expcovone)^{-1} \normHS{(\Ahat - \Ast)\expcovone^{1/2}}^2 \le \frac{\epsgam^2}{2 \lambda_{\min}(\expcovone)}.
     \end{align*}
     Thus we can replace: 
     \begin{align*}
     \log (1 + \frac{2\brackalign{(\Ahat - \Ast)^\herm(\Ahat - \Ast),\expcovone}}{\epsgam^2}) \le \log(1 +\frac{2\opnorm{\expcovone}\epsgam^2}{2 \epsgam^2 \lambda_{\min}(\expcovone)} ) = \log(1 + \mathrm{cond}(\expcovone)).
     \end{align*}
\end{proof}
	\subsubsection{Proof of \Cref{prop:change_of_system} \label{sec:proof_prop_change_of_system}}
	Before beginning the proof, let us recall the setup. We let $(A_1,B_1)$ and $(A_2,B_2)$ be two systems, and  $K$ a controller, and define for $u \in [0,1]$
	\begin{align*}
		\Abar(u) \defeq A_1 + u (A_2 - A_1) \quad \Bbar(u)  \defeq B_1 + u (B_2 - B_1), \quad \Aclbar(u) \defeq A(u) + B(u) K.
	\end{align*}
	Assume that $K$ stabilizes all instances $(\Abar(u),\Bbar(u))$, so that
	 \begin{align*}
	 \Mbar \defeq \max_{u \in [0,1]} \opnorm{\Pbar(u)} < \infty,\quad \text{ where } \Pbar(u) \defeq \dlyap{\Aclbar(u)}{Q + K^\herm R K}.
	 \end{align*}
	Finally, define $\DelbarAcl  \defeq  \frac{d}{du} \Aclbar(u) =  (A_2 - A_1) + (B_2 - B_1) K$.

	\begin{proof}
	The proof is based on analyzing the behavior of the curve,
	\begin{align*}
	z(u):[0,1]\to \R \defeq \traceb{\DelbarAcl^\herm \DelbarAcl \dlyap{\Aclbar(u)^\herm}{\Sigma}}.
	\end{align*}
	Since $\Mbar < \infty$, $\Aclbar(u)$ is stable for all $u \in [0,1]$, and hence the Lyapunov operator is well defined implying that $z(u)$ is finite.

	In order to simplify our presentation, we let $Z(u) = \dlyap{\Aclbar(u)^\herm}{\Sigma}$. Since $\Aclbar(u)$ is a linear curve supported on stable operators, \Cref{lemma:dlyap_derivative} ensures that $Z(u)$ is continuously Frechet-differentiable, with Frechet derivative equal to
	\begin{equation}
	\label{eq:zu_derivative}
	 Z'(u) = \dlyap{\Aclbar(u)^\herm}{ \DelbarAcl Z(u) \Aclbar(u)^\herm + \Aclbar(u) Z(u) \DelbarAcl^\herm}.
	\end{equation}
	From part (a) of \Cref{lem:derivative_and_trace}, $z(u)$ is continuously differentiable (as a real-valued curve), and its derivative is $z'(u) = \trace[\DelbarAcl^\herm \DelbarAcl, Z'(u)]$. By applying the PSD AM-GM inequality twice (\Cref{lemma:psd_am_gm}), we have that for any $\alpha > 0$,
	\begin{equation}
	\label{eq:twice_amgm}
	Z'(u) = \DelbarAcl Z(u) \Aclbar(u)^\herm + \Aclbar(u) Z(u) \DelbarAcl^\herm \preceq \frac{1}{2} \alpha \cdot \DelbarAcl Z(u) \DelbarAcl^\herm + \frac{1}{2} \alpha^{-1} \Aclbar(u) Z(u) \Aclbar(u)^\herm.
	\end{equation}
	Combining \Cref{eq:twice_amgm} and \Cref{eq:zu_derivative}, and optimizing over $\alpha$, we can upper bound the derivative of $z(u)$,
	\begin{align}
	z'(u) &=  \traceb{\DelbarAcl^\herm \DelbarAcl \dlyap{\Aclbar(u)^\herm}{ \DelbarAcl Z(u) \Aclbar(u)^\herm + \Aclbar(u) Z(u) \DelbarAcl^\herm}} \nn\\
		&\leq \left(\underbrace{\traceb{\DelbarAcl^\herm \DelbarAcl\dlyap{\Aclbar(u)^\herm}{\DelbarAcl Z(u) \DelbarAcl^\herm}}}_{\defeq R_1} \right)^{1/2}\\
		&\qquad \times \left(\underbrace{\traceb{\DelbarAcl^\herm \DelbarAcl \dlyap{\Aclbar(u)^\herm}{\Aclbar(u) Z(u) \Aclbar(u)^\herm}}}_{\defeq R_2}\right)^{1/2} \label{eq:zprime}.
	\end{align}
We now bound each of $R_1$ and $R_2$ individually.
	\paragraph{Bounding $R_1$} Beginning with $R_1$,
	\begin{align*}
	&\traceb{\DelbarAcl^\herm \DelbarAcl\dlyap{\Aclbar(u)^\herm}{\DelbarAcl Z(u) \DelbarAcl^\herm}} \\
	&\quad \leq \opnorm{\DelbarAcl}^2 \traceb{ \dlyap{\Aclbar(u)^\herm}{\DelbarAcl Z(u) \DelbarAcl^\herm}} \\
	&\quad \leq \opnorm{\DelbarAcl}^2 \cdot \opnorm{\dlyap{\Aclbar(u)}{I}} \traceb{\DelbarAcl Z(u) \DelbarAcl^\herm} \\
	&\quad \leq \opnorm{\DelbarAcl}^2 \opnorm{\Pbar(u)} \traceb{\DelbarAcl Z(u) \DelbarAcl^\herm}. \\
	&\quad = \opnorm{\DelbarAcl}^2 \opnorm{\Pbar(u)} \cdot z(u)\\
	&\quad = \opnorm{\DelbarAcl}^2 \Mbar \cdot z(u)
	\end{align*}
	In the second line, we have used \Cref{lemma:trace_dlyap_bound}. Furthermore, the second to last line is justified by the following observation. Since $I \preceq Q$ and $I \preceq R$,
	\begin{align*}
	\dlyap{\Aclbar(u)}{I} \preceq \dlyap{\Aclbar(u)}{Q + K^\herm R K} = \Pbar(u).
	\end{align*}
	\paragraph{Bounding $R_2$}  We first notice that,
	\begin{align*}
	\dlyap{\Aclbar(u)^\herm}{\Aclbar(u) Z(u) \Aclbar(u)^\herm} &= \sum_{j=0}^\infty \Aclbar(u)^{j+1} \dlyap{\Aclbar(u)^\herm}{\Sigma}  (\Acl^\herm)^{j+1} \\
	& = \Aclbar(u) \cdot \dlyap{\Aclbar(u)^\herm}{\dlyap{\Aclbar(u)^\herm}{\Sigma}} \cdot  \Aclbar(u)^\herm \\
	& = \dlyap{\Aclbar(u)^\herm}{\dlyap{\Aclbar(u)^\herm}{\Sigma}} - \dlyap{\Aclbar(u)^\herm}{\Sigma} \\
	& = \dlyapm{\Aclbar(u)^\herm}{\Sigma}{1} - \dlyap{\Aclbar(u)^\herm}{\Sigma} \\
	& \preceq \dlyapm{\Aclbar(u)^\herm}{\Sigma}{1}.
	\end{align*}
	The third line the calculation above follows from definition of the solution to the Lyapunov equation, i.e \Cref{eq:lyapunov_equation}. The second to last line is justified by \Cref{lemma:repeated_dlyap}. Setting $X \defeq \DelbarAcl^\herm \DelbarAcl \succeq 0$,
	\begin{align*}
	R_2 = \traceb{X\dlyap{\Aclbar(u)^\herm}{\Aclbar(u) Z(u) \Aclbar(u)^\herm}} \le \traceb{X \dlyapm{\Aclbar(u)^\herm}{\Sigma}{1}}.
	\end{align*}
Recall $\brackalign{X,Y} \defeq \max\{\nucnorm{XWY} : \|W\|_{\op} = 1\}$. By \Cref{lemma:dlyapm_bound_trace}, the following bound holds for any $n \geq 0$:
	\begin{align*}
	R_2 \le \traceb{X \dlyapm{\Aclbar(u)^\herm}{\Sigma}{1}} &\le n \cdot \traceb{X\dlyap{\Aclbar(u)^\herm}{\Sigma}} \\
	&\qquad + (n+1)\brackalign{X,\Sigma}\opnorm{\Pbar(u)}^{3} \exp(-\opnorm{\Pbar(u)}^{-1} n) \\
	&= n \cdot z(u) \\
	&\qquad+ (n+1)\brackalign{X,\Sigma}\opnorm{\Pbar(u)}^{3} \exp(-\opnorm{\Pbar(u)}^{-1} n) \\
	&\leq n \cdot z(u) + (n+1)\brackalign{X,\Sigma}\Mbar^{3} \exp(-\Mbar^{-1} n).
	\end{align*}

	\paragraph{Concluding the proof}
	Combining our bounds for $R_1$ and $R_2$,  then for any $n \ge 0$, we have
	\begin{align*}
	z'(u) &\leq \sqrt{\left(\opnorm{\DelbarAcl}^2 \Mbar \cdot z(u)\right)\left(n \cdot z(u) + (n+1)\brackalign{X,\Sigma}\Mbar^{3} \exp(-\Mbar^{-1} n) \right) }\\
	&\overset{(i)}{\leq} \sqrt{n \Mbar}\opnorm{\DelbarAcl} z(u) + \sqrt{n\Mbar}\opnorm{\DelbarAcl} \left(z(u)^{1/2} \cdot \sqrt{ \frac{n+1}{n}\brackalign{X,\Sigma}\Mbar^{3} \exp(-\Mbar^{-1} n) }\right)\\
	&\overset{(ii)}{\leq} \sqrt{n \Mbar}\opnorm{\DelbarAcl} z(u) + \sqrt{n\Mbar}\opnorm{\DelbarAcl} \left(\frac{z(u)}{2} + \frac{n+1}{2n}\brackalign{X,\Sigma}\Mbar^{3} \exp(-\Mbar^{-1} n) \right)\\
	&\le \underbrace{\frac{3}{2}\sqrt{n \Mbar}\opnorm{\DelbarAcl}}_{:=a} \cdot z(u) + \underbrace{\sqrt{n\Mbar} \opnorm{\DelbarAcl}\cdot \brackalign{X,\Sigma}  \Mbar^{3}\exp(-\Mbar^{-1} n)}_{:=b},
	\end{align*}
	where $(i)$ uses concavity of the square-root, and $(ii)$ applies AM-GM.
	In other words, the scalar function $z(u)$ exhibits the self-bounding property $z'(u) \le a z(u) + b$ for $a,b$ defined above. Hence, for any slack parameter $\eta > 0$,\footnote{This is useful to ensure strict domination of one ODE by another} a standard ODE comparison inequality (see, e.g. \cite{simchowitz2020naive}, Lemma D.1) implies that $z(u) \le \tilde{z}(u)$, where $\tilde{z}(u)$ solves the analogous ODE with equality:
	\begin{align}
	\tilde{z}(u) = a \tilde{z}'(u) + b + \eta, \quad \tilde{z}(0) = z(0)+\eta \label{eq:tilde_z}.
	\end{align}
The differential equation above has solution:
	\begin{align*}
	\tilde{z}(u) \leq \left(z(0) + \eta + \frac{b+\eta}{a}\right)\exp(a\cdot u)  - \frac{b + \eta}{a}.
	\end{align*}
	Taking $\eta \to 0$ and bounding $z(u) \le \tilde{z}(u)$ yields
	\begin{align*}
	z(1) \leq z(0) \exp(a) + \frac{b}{a} \exp(a).
	\end{align*}
	Substituting in the relevant quantities, the following bound holds for any $n \ge 1$:
	\begin{align*}
	z(1) \le \exp\left( \frac{3}{2}\sqrt{n \Mbar}\opnorm{\DelbarAcl} \right) \left(z(0) +  \brackalign{X,\Sigma}  \Mbar^{3}\exp(-\Mbar^{-1} n)\right),
	\end{align*}
	Taking $n =\ceil{\Mbar  \log \left( \frac{\brackalign{X,\Sigma}\Mbar^{3}}{z(0)} \right)} \le \Mbar  \log_+ \left( \frac{e\brackalign{X,\Sigma}\Mbar^{3}}{z(0)} \right)$ lets us bound 
	\begin{align*}\brackalign{X,\Sigma}\Mbar^{3}\exp(-\Mbar^{-1} n) \le z(0),
	\end{align*} and hence
	\begin{align*}
	z(1)   \le 2\exp\left( \frac{3}{2}\Mbar \opnorm{\DelbarAcl} \sqrt{\log_+ \left( \frac{e\cdot\brackalign{X,\Sigma}\Mbar^{3}}{z(0)} \right)}\right) z(0).
	\end{align*}
	Recalling $X = \DelbarAcl^\herm \DelbarAcl$, and recognizing $z(1)$ as $\epssysnum{2}$ and $z(0)$ as $\epssysnum{1}$, the bound follows. Note also that this bound on the derivative is necessarily non-decreasing in the $\epssysnum{2}$ argument. 
	\end{proof}

	\subsubsection{Proof of \Cref{lemma:uniform_bound_p} \label{sec:lem:unif_p_bound}}
	\begin{proof} 
	We apply the first bullet point of \Cref{prop:uniform_perturbation_bound} to establish both parts of the lemma. 
	\paragraph{Part 1:} Take $(A_1,B_1) = (\Ast,\Bst)$, and $(A_2,B_2) = (A(t),B(t))$. Since $(A_2,B_2)$ lies on the segment joining $(\Ast,\Bst)$ and $(\Ahat,\Bhat)$, we have
	\begin{align*}
	\max\{\|A_1 - A_2\|_{\op}, \|B_1 - B_2\|_{\op}\} \le \max\{\|\Ast - \Ahat\|_{\op}, \|\Bst - \Bhat\|_{\op}\} \le \epsilon_{\op}
	\end{align*}
	Hence, if $\epsilon_{\op} \le \frac{\eta}{16\opnorm{P_1}^{3}}$ for $\eta = 1/11$,  \Cref{prop:uniform_perturbation_bound} implies $\opnorm{P(t)} \le (1+\eta)\opnorm{\Pst} \le 1.1\opnorm{\Pst}$.

	\paragraph{Part 2} For part $2$, take $(A_1,B_1) = (A(t),B(t))$ and $(A_2,B_2) = (A(u\cdot t),B(u\cdot t))$. Again, it holds that $\max\{\|A_1 - A_2\|_{\op}, \|B_1 - B_2\|_{\op}\} \le \epsilon_{\op}$. Then, if $\epsilon_{\op} \le \eta / (16\opnorm{P_1}^{3})$, where $P_1 = \Pinfty{A_1}{B_1}$,  \Cref{prop:uniform_perturbation_bound} implies that 
	\begin{align*}
	\opnorm{\PinftyK{K(t)}{A(u \cdot t)}{B(u \cdot t)}}  \le (1 + \eta)\opnorm{P_1}.
	\end{align*} From part $1$ of the present lemma, $\opnorm{P_1} \le (1+\eta)\opnorm{\Pst}$. Hence, if $\epsilon_{\op} \le \frac{\eta}{16(1+\eta)^3\opnorm{\Pst}^{3}}$, we have $\opnorm{\PinftyK{K(t)}{A(u \cdot t)}{B(u \cdot t)}} \le (1+\eta)^2 \opnorm{\Pst}$. Computing $(1+\eta)^2 \le 1.2$, and noting that $16(1+\eta)^3/\eta \le 229$ concludes the proof.

	\end{proof}

\newpage


\section{Technical Lemmas \label{sec:technical_lemmas}}
This section states and prove the main technical tools used throughout the paper. \Cref{sec:technical_linalg} contains linear algebraic tools, \Cref{subsec:lyapunov_theory} gives tools for controlling terms involving Lyapunov operators, \Cref{sec:technical_spectral_comparison} states and proves a comparison theorem between the eigendecay of a PSD operator $\Lambda$ and its image $\dlyap{A}{\Lambda}$. Finally, \Cref{sec:Frechet} addresses the relevant differentiability considerations that arise in infinite dimensional spaces.

\subsection{Linear Algebra \label{sec:technical_linalg}}

	\begin{lemma}\label{lemma:outer_product_bound} Let $Z \succeq 0$ and $Y_1,\dots,Y_T$ be bounded linear operators on a Hilbert space $\hilbx$. Then,
	\begin{align*}
	\left(\sum_{t=1}^{T}Y_t\right) Z \left(\sum_{t=1}^{T}Y_t\right)^{\herm} \preceq 2 \sum_{t=1}^T t^2 \, Y_t Z Y_t^{\herm}.
	\end{align*}
	\end{lemma}
	\begin{proof} Let $\matx$ be any vector in $\hilbx$, then
	\begin{align*}
	\textstyle\innerHx{\matx}{\left(\sum_{t=1}^{T}Y_t\right) Z \left(\sum_{t=1}^{T}Y_t\right)^{\herm} \matx} &= \textstyle\innerHx{Z^{1/2} \left(\sum_{t=1}^{T}Y_t\right)^{\herm} \matx}{Z^{1/2} \left(\sum_{t=1}^{T}Y_t\right)^{\herm} \matx} \\
	&= \textstyle\normHx{\sum_{t=1}^T Z^{1/2} \left(\sum_{t=1}^{T}Y_t\right)^{\herm} \matx}^2.
	\end{align*}
	Therefore it suffices to show that, for any  $\matx_1,\dots,\matx_T \in \hilbx$, $\normHx{\sum_{t=1}^T \matx_t}^2 \le 2\sum_{t=1}^T t^2 \cdot \normHx{\matx_t}^2$. We argue by Cauchy Schwartz,
	\begin{align*}
	\normHx{\sum_{t=1}^T \matx_t}^2 = \normHx{\sum_{t=1}^T \frac{1}{t} \cdot t \cdot \matx_t}^2 \le \left(\sum_{t=1}^T \frac{1}{t^2}\right) \cdot \sum_{t=1}^T \normHx{t \cdot \matx_t}^2.
	\end{align*}
	Since $\sum_{t=1}^T \frac{1}{t^2} < \frac{\pi^2}{6} \le 2$, the bound follows.
	\end{proof}

%
	\begin{lemma}
	\label{lemma:log_det_trace}
	Let M be a positive, semi-definite linear operator, then 
	\[
	\log \det(I + M) \leq \traceb{M}
	\]	
	\end{lemma}
	\begin{proof}
	If $M$ has eigenvalues $\{\sigma_i\}_{i=1}^\infty$, then $I + M$ has eigenvalues $\{1 + \sigma_i\}_{i=1}^\infty$. Since the determinant of a linear operator is equal to the product of its eigenvalues, we have that, 
	\[
	\log \det(I +M) = \sum_i \log(1 + \sigma_i) \leq \sum_i \sigma_i = \traceb{M},
	\] 
	where we have used the numerical inequality $\log(1 + x) \leq x$ for all $x \geq 0$.
	\end{proof}

	\begin{lemma}
	\label{lemma:psd_am_gm}
	Let $X, Y, P:\hilbx \rightarrow \hilbx$ be linear operators and let $P$ be positive semi-definite, then for any $\alpha > 0$ we have that 
	\[
	X P Y^\herm \preceq \frac{\alpha}{2} XP X^\herm  + \frac{1}{2\alpha}  Y P Y^\herm.
	\]
	\end{lemma}
	\begin{proof}
	Letting $\matv \in \hilbx$, the proof follows by direct application of Cauchy-Schwarz and the AM-GM inequality:
	\begin{align*}
		\innerHx{\matv}{XPY^\herm \matv}  &= \innerHx{\sqrt{\alpha}P^{\frac{1}{2}} X^\herm \matv}{\frac{1}{\sqrt{\alpha}} P^{\frac{1}{2}}Y^\herm \matv} \\
		&\leq \normhilbx{ \sqrt{\alpha}   P^{\frac{1}{2}} X^\herm \matv} \normhilbx{ \frac{1}{\sqrt{\alpha}} P^{\frac{1}{2}}Y^\herm \matv} \\ 
		&\leq \frac{\alpha}{2} \normhilbx{P^{\frac{1}{2}} X^\herm \matv}^2 + \frac{1}{2\alpha} \normhilbx{P^{\frac{1}{2}}Y^\herm \matv}^2\\
		&= \frac{\alpha}{2} \innerHx{\matv}{XPX^\herm \matv} + \frac{1}{2\alpha}  \innerHx{\matv}{YPY^\herm \matv}.
	\end{align*}
	\end{proof}

	\begin{lemma}
	\label{lemma:trace_norm_ineq}
	Let $X: \hilbx \rightarrow \hilbx$ be a self-adjoint operator and let $Y \in \psdhil$ be a trace class, positive semi-definite operator. If, $-Y \preceq X  \preceq Y$ then $\| X \|_{\mathsf{tr}} \leq \traceb{Y}$.
	\end{lemma}
	\begin{proof} 
	Let $X = \sum_{j=1}^\infty \lambda_j \matq_j\matq_j^\herm$ denote the spectral decomposition of $X$. Then, 
	\begin{align*}
	\nucnorm{X} &= \sum_{j\ge 1}|\lambda_j| = \sum_{j\ge 1}|\matq_j^\herm X \matq_j|,
	\end{align*}
	since $-Y \preceq X \preceq Y$, $|\matq_j^\herm X \matq_j| \le \matq_j^\herm Y \matq_j$ for each $j$. Thus, $\nucnorm{X}  \le \sum_{j=1}^{\infty} \matq_j^\herm Y \matq_j = \traceb{Y}$, since the elements $\matq_j$ form an orthonormal basis.
	\end{proof}

	\begin{lemma}\label{lemma:symmetrized_holder} Let $X,Y_1,Y_2$ be symmetric operators with $Y_1,Y_2 \succeq 0$. For all $\alpha > 0$, define $Y_{[\alpha]} := \frac{\alpha}{2} Y_1 + \frac{1}{2\alpha} Y_2$, and suppose that $-Y_{[\alpha]} \preceq X \preceq Y_{[\alpha]}$ for any $\alpha > 0$. Then, 
	\begin{align*}
	\|X\|_{\HS} \le \sqrt{\|Y_1\|_{\op}\traceb{Y_2}}.
	\end{align*}
	\end{lemma}
	\begin{proof} Let $X = \sum_{j=1}^\infty \lambda_j \matq_j\matq_j^\herm$ denote the spectral decomposition of $X$. Then, 
	\begin{align*}
	\normHS{X}^2 &= \sum_{j\ge 1}\lambda_j^2 = \sum_{j\ge 1}(\matq_j^\herm X \matq_j)^2.
	\end{align*}
	For any fixed $j$, we have $|\matq_j^\herm X \matq_j| \le \inf_{\alpha > 0} \matq_j^\top Y_{[\alpha]}\matq $, since for all $\alpha$, $Y_{[\alpha]} \preceq X \preceq Y_{[\alpha]}$. Moreover, we compute
	\begin{align*}
	|\matq_j^\herm X \matq_j| \le \inf_{\alpha > 0}  \matq_j^\herm Y_{[\alpha]} \matq_j, &=  \inf_{\alpha> 0} \frac{\alpha}{2} \matq_j^\herm Y_1 \matq_j + \frac{1}{2\alpha} \matq_j^\herm Y_2 \matq_j \\
	&= \sqrt{ \matq_j^\herm Y_1 \matq_j \cdot \matq_j^\herm Y_2 \matq_j} \le \sqrt{\|Y_1\|_{\op}}\sqrt{\matq_j^\herm Y_2 \matq_j},
	\end{align*}
	where we note that $\min_{\alpha > 0} \frac{a}{2\alpha} + \frac{b}{2\alpha} = \sqrt{ab}$ for nonegative $a, b \ge 0$. Combining the above two displays, 
	\begin{align*}
	\|X\|_{\HS}^2 &\le  \sum_{j=1}^{\infty} \left(\sqrt{\|Y_1\|_{\op}}\sqrt{\matq_j^\herm Y_2 \matq_j},\right)^2 = \opnorm{Y_1} \sum_{j=1}^{\infty}\matq_j^\herm Y_2 \matq_j.
	\end{align*}
	Since $\matq_j$ are an orthonormal basis, $\sum_{j=1}^{\infty}\matq_j^\herm Y_2 \matq_j = \traceb{Y_2}$. The bound follows.
	\end{proof}

	\begin{lemma}\label{lem:brackalign_domination} Define the operator $\brackalign{X,Y} \defeq \max\{\nucnorm{XWY} : \|W\|_{\op} = 1\}$. Then for any $ X' \succeq X \succeq0 $,  $\brackalign{X,Y} \le \brackalign{X',Y}$. Similarly, for $Y' \succeq Y \succeq 0$, then $\brackalign{X,Y'} \le \brackalign{X,Y}$.
	\end{lemma}
	\begin{proof} 
	Let us prove the first point; the second is analogous. 
	Fix $\epsilon > 0$, and let $W$ be such that $\nucnorm{XWY} \ge \brackalign{X,Y} - \epsilon$. Let $WY = U\Sigma V^\herm$ denote the singular value decomposition of $WY$. Then, $\nucnorm{XWY} = \nucnorm{XU\Sigma V^\herm } = \nucnorm{XU\Sigma U^\herm}$, since $VU^\herm$ is orthonormal, and thus conjugating by it does not alter the trace norm. Since $X' \succeq X \succeq 0 $ and $U\Sigma U^\herm \succeq 0$,
	\begin{align*}
	\nucnorm{XU\Sigma U^\herm} &= \traceb{X U\Sigma U^\herm} \le \traceb{X'U\Sigma U^\herm} \\
	&= \nucnorm{X'U\Sigma U^\herm} = \nucnorm{X'U\Sigma V^\herm} = \nucnorm{X'WY} \le \brackalign{X',Y}.
	\end{align*}
	The bound follows.
	\end{proof}

\subsection{Lyapunov Theory}

\label{subsec:lyapunov_theory}

	\begin{lemma}\label{lem:Pk_bound} Let $(A_1,B_1)$ be a stabilizable instance with stabilizing controller $K_1$, and let $P_1= \PinftyK{K_1}{A_1}{B_1}$ be the associated value function.  Then, if $R\succeq I$ and $Q \succeq I$,  it holds that $P_1 \succeq I$ and $\opnorm{K_1} \le \opnorm{P_1}^{1/2}$. In particular, $\Pinfty{A_1}{B_1} \succeq I$, and $\opnorm{\Kinfty{A_1}{B_1}} \le {\opnorm{\Pinfty{A_1}{B_1}}}^{1/2}$. 
	\end{lemma}
	\begin{proof}
	We have the identity:
	\begin{align*}
	P_1 = \PinftyK{K_1}{A_1}{B_1}=  \dlyap{A_1+B_1K_1}{Q + K_1^\top R K_1} \succeq Q + K_1 R K_1
	\end{align*} 
	Thus, $P_1 \succeq Q \succeq I$, and since $R \succeq I$, $K_1^\top K_1 \preceq K_1^\top R K_1 \preceq P_1$, so that $\opnorm{K_1} \le \opnorm{P_1}^{1/2}$.
	\end{proof}

	\begin{lemma} Let $Y$ be a trace class operator, and suppose $\dlyap{X^\herm}{I}$ is bounded. Then, 
	\label{lemma:trace_dlyap_bound}
		\begin{align*}
		\nucnorm{\dlyap{X}{Y}} \leq \opnorm{\dlyap{X^\herm}{I}}\nucnorm{Y}.
		\end{align*}
		In particular, if $Y \succeq 0$ is PSD, then 
		\begin{align*}
		\trace[\dlyap{X}{Y}] \leq \opnorm{\dlyap{X^\herm}{I}}\trace[Y].
		\end{align*}
	\end{lemma}
	\begin{proof}
		We write out $Y$ in its spectral decomposition, $Y = \sum_{i=0}^\infty v_i \otimes v_i \cdot \lambda_i$, and use the form of the Lyapunov solution to get that
		\begin{align*}
		\dlyap{X}{Y} = \sum_{j=0}^\infty (X^\herm)^j Y X^j  &= \sum_{i=0}^\infty \sum_{j=0}^\infty  (X^\herm)^j (v_i \otimes v_i) X^j \cdot \lambda_i \\
		&\preceq \sum_{i=0}^\infty \sum_{j=0}^\infty  (X^\herm)^j (v_i \otimes v_i) X^j \cdot |\lambda_i| \defeq Z. 
		\end{align*}
		Similarly, 
		\begin{align*}
		\dlyap{X}{Y}  \succeq - \sum_{i=0}^\infty \sum_{j=0}^\infty  (X^\herm)^j (v_i \otimes v_i) X^j \cdot |\lambda_i| = -Z.
		\end{align*}
		Therefore, from \Cref{lemma:trace_norm_ineq}
		\begin{align*}
			\nucnorm{\dlyap{X}{Y}} &\le \trace[Z] = \sum_{i=0}^\infty \nucnorm{\dlyap{X}{v_i \otimes v_i}} \cdot |\lambda_i| \\
			& = \sum_{i=0}^\infty |\innerHx{v_i}{\dlyap{X^\herm}{I}v_i}| \cdot |\lambda_i|\\
			& \leq \opnorm{\dlyap{X^\herm}{I}} \sum_{i=0}^\infty |\lambda_i|.\\
		\end{align*}
	Noting that $\sum_{i=0}^\infty \lambda_i = \nucnorm{Y}$ finishes the first part of the proof. For the second part, notice that if $Y \succeq 0 $, then the nuclear and trace norms norms coincide. 

	\end{proof}

	\begin{lemma}
	\label{lemma:lyapunov_series_bound}
		Let $A$ be stable and $P = \dlyap{A}{\Sigma}$ be the corresponding value function for some $\Sigma \succeq I$, then for all integers $j \geq 0$, 
	\begin{align*}
	(A^\herm)^j P A^j \preceq P (1 - \opnorm{P}^{-1})^j.
	\end{align*}
	In particular, for all $j \ge 0$, 
	\begin{align*}
		\|A^j\|_{\op}^2 \le  \|P\|_{\op} (1 - \|P\|_{\op}^{-1})^j.
	\end{align*}
	Hence, for any over $\Sigma' \succeq 0 $, $\|\dlyap{A}{\Sigma}\|_{\op} \le \|\Sigma\|_{\op} \|P\|_{\op}^2$.
		\end{lemma}
	\begin{proof}
		
	Since $P$ satisfies the Lyapunov equation, $A^\herm P A - P +\Sigma =0$, we can write for any $x \in \hilspace$:
	\begin{align*}
		\innerHx{x}{A^\herm P A x} & = \innerHx{x}{Px} - \innerHx{x}{\Sigma x} \nn \\
		& = \innerHx{x}{Px}\left(1 - \frac{\innerHx{x}{\Sigma x}}{\innerHx{x}{Px}} \right) \nn \\ 
		& \leq \innerHx{x}{Px} (1 - \frac{1}{\opnorm{P}}).
	\end{align*}
	In the last line, we have use the assumption that $\Sigma \succeq I$. Hence, $A^\herm P A \preceq P (1 - \opnorm{P}^{-1})$ where $\opnorm{P} > 1$. Repeating this argument, we can in fact show that, 
	\begin{align*}
	(A^\herm)^j P A^j \preceq P (1 - \opnorm{P}^{-1})^j.
	\end{align*}
	For the second guarantee, since $\Sigma \succeq I$, we have $P \succeq \Sigma \succeq I$. Hence, 
	\begin{align*}
		\opnorm{A^j}^2 = \opnorm{(A^j)^\herm A^j} \leq  \opnorm{(A^j)^\herm P A^j} \leq  \|P\|_{\op}(1 - \|P\|_{\op}^{-1})^j.
	\end{align*}	
	\end{proof}
	For the next lemma, we recall  the higher order Lyapunov operator from \Cref{def:higher_order_dlyap},
	\begin{align*}
		\dlyapm{A}{\Lambda}{m} \defeq \sum_{j=0}^\infty (A^\herm)^j \Lambda A^j (j+ 1)^m.
	\end{align*}
	\begin{lemma}
	Let $A$ be a stable linear operator and $\Sigma$ be self-adjoint, the the following identity holds: 
	\label{lemma:repeated_dlyap}
	\[
	\sum_{j=0}^\infty A^j \dlyap{A^\herm}{\Sigma} (A^\herm)^{j} = \sum_{j=0}^\infty A^j \Sigma (A^\herm)^j \cdot (j+1). 
	\]
	Equivalently, 
	\[
	\dlyap{A^\herm}{\dlyap{A^\herm}{\Sigma}}=\dlyapm{A^\herm}{\Sigma}{1}.
	\]
	\end{lemma}

	\begin{proof}
	To simplify the proof, let $\Gamma = \dlyap{A^\herm}{\Sigma}$.
	Since $\Gamma = \dlyap{A^\herm}{\Sigma}$ is the solution to the Lyapunov equation $A \Gamma A^\herm + \Sigma - \Gamma = 0$, we have that 
	\[
	A \Gamma A^\herm  \Sigma = \Gamma - \Sigma.
	\]
	By repeating this argument, we can in fact show that: 
	\[
	(A)^j \Gamma (A^\herm)^j = \Gamma - \sum_{i=0}^{j-1} A^i \Gamma (A^\herm)^i.
	\]
	Using the fact that $\Gamma = \sum_{j=0}^\infty A^j \Sigma (A^\herm)^j$, it follows that $(A)^j \Gamma (A^\herm)^j = \sum_{i=j}^\infty A^i \Sigma (A^\herm)^i$. Therefore, we can rewrite $\dlyap{A^\herm}{\dlyap{A^\herm}{\Sigma}}$ as follows,
	\begin{align*}
		\sum_{j=0}^\infty A^j \Gamma (A^\herm)^{j} & = \sum_{j=0}^\infty \sum_{i=j}^\infty A^i \Sigma (A^\herm)^i \\
		& = \sum_{j=0}^\infty A^j \Sigma (A^\herm)^j \cdot (j+1).
	\end{align*}
	which is exactly $\dlyapm{A^\herm}{\Sigma}{1}$.
	\end{proof}

	\begin{lemma}
	\label{lemma:dlyapm_bound}
	Let $A$ be stable and $P = \dlyap{A}{Q + K^\herm R K}$ be the corresponding value function, then for all integers $n \geq 0$, 	\\

	For $m=1:$
	\begin{align*}
		\dlyapm{A^\herm}{\Sigma}{1} \preceq n \cdot \dlyap{A^\herm}{\Sigma} + (n+1)\opnorm{\Sigma} \opnorm{P}^{3} \exp(-\opnorm{P}^{-1} n) \cdot I
	\end{align*}

	For $m=2:$
	\begin{align*}
	\dlyapm{A^\herm}{\Sigma}{2} \preceq n^2 \cdot \dlyap{A^\herm}{\Sigma} + (n^2 + 2n + 2) \opnorm{\Sigma} \opnorm{P}^{4} \exp(-\opnorm{P}^{-1} n) \cdot I
	\end{align*}
		\end{lemma}
	\begin{proof}
		We begin by expanding out the definition of $\dlyapm{A(t)^\herm}{\Sigma}{m}$, 
		\begin{align}
			\dlyapm{A^\herm}{\Sigma}{m} &= \sum_{j=0}^\infty A^j \Sigma (A^\herm)^j \cdot(j+1)^m \nonumber\\ 
			& \preceq n^m  \sum_{j=0}^{n-1}A^j \Sigma (A^\herm)^j \nonumber + \sum_{j=n}^\infty A^j \Sigma (A^\herm)^j \cdot(j+1)^m\\
			& \preceq n^m \dlyap{A^\herm}{\Sigma} + \sum_{j=n}^\infty A^j \Sigma (A^\herm)^j \cdot(j+1)^m , \label{eq:first_dlyap_high_order} 
		\end{align}
		where we have let the sum go to infinity in the first term.
		Focusing on the second term, 
		\begin{align*}
				\sum_{j=n}^\infty A^j \Sigma (A^\herm)^j \cdot(j+1)^m &\preceq I \cdot \opnorm{\Sigma} \sum_{j=n}^\infty \opnorm{ A^j (A^\herm)^j  } \cdot (j+1)^m  \\
				 &=  I \cdot \opnorm{\Sigma} \sum_{j=n}^\infty \opnorm{(A^\herm)^j A^j} \cdot (j+1)^m  \tag{$\opnorm{NN^\herm} = \opnorm{N^\herm N}$}\\  
		 & \preceq I \cdot \opnorm{\Sigma} \sum_{j=n}^\infty \opnorm{(A^ \herm)^j P A^j} \cdot (j+1)^m \tag{ $P \succeq I$}  \\
		 &\preceq \opnorm{\Sigma} \opnorm{P} \sum_{j=n}^\infty (j+1)\cdot (1 - \opnorm{P}^{-1})^j \cdot I. \tag{\Cref{lemma:lyapunov_series_bound}}
	\end{align*}
	The result then follows by applying the following two identities regarding geometric series, which hold for $c \in (0,1)$,
	\begin{align*}
		\sum_{j=n}^\infty (1-c)^j \cdot (j+ 1) & = \frac{(1-c)^n(cn + 1)}{c^2} \\ 
		\sum_{j=n}^\infty (1-c)^j \cdot (j+ 1)^2 & = \frac{(1-c)^n(c^2n^2 +2cn - c)}{c^3}, 
	\end{align*}
	and using the fact that $(1-c)^t \leq \exp(-c t)$.
	\end{proof}

	\begin{lemma}\label{lemma:dlyapm_bound_trace}
	Let $A$ be stable and $P = \dlyap{A}{Q + K^\herm R K}$ be the corresponding value function, then for all integers $n \geq 0$, and all PSD, bounded operators $X$, and $\Sigma \succeq 0$,
	\begin{align*}
		\traceb{X\dlyapm{A^\herm}{\Sigma}{1}} \le n \cdot \traceb{X\dlyap{A^\herm}{\Sigma}} + (n+1)\brackalign{X,\Sigma}\opnorm{P}^{3} \exp(-\opnorm{P}^{-1} n),
	\end{align*}
	where $\brackalign{X,Y} \defeq \max\{\nucnorm{XWY} : \|W\|_{\op} = 1\}$.
	\end{lemma}
	\begin{proof} From \Cref{eq:first_dlyap_high_order} and the fact that $\traceb{WY} \le \traceb{WZ}$ for $W \succeq 0$ and $Y \preceq Z$,
	\begin{align*}
	\traceb{X\dlyapm{A^\herm}{\Sigma}{1}} &\le n \traceb{X\dlyap{A^\herm}{\Sigma}} + \traceb{X\sum_{j=n}^\infty A^j \Sigma (A^\herm)^j \cdot(j+1)}.
	\end{align*}
	We can then bound,
	\begin{align}
	\traceb{X A^j \Sigma (A^\herm)^j} \le \|A^j\|_{\op}^2 \nucnorm{X \cdot \frac{A^j}{\|A^j\|_{\op}} \cdot \Sigma} \le \|A^j\|_{\op}^2\brackalign{X,\Sigma} =\|A^j (A^{j})^\herm\|_{\op} \brackalign{X,\Sigma},  
	\end{align}
	where we note the above bound holds even if $A^j = 0$. The bound now follows from the computation given in the proof of \Cref{lemma:dlyapm_bound}.
	\end{proof}

\subsection{Spectrum Comparison under Lyapunov Operator \label{sec:technical_spectral_comparison}}

	\begin{lemma}[Iterated Weyl's Eigenvalue Inequality]\label{lem:weyl_ineq} Let $X_1,\dots,X_n$ be a sequence of PSD operators, and let $\lambda_j(\cdot)$ denote the $j$-eigenvalue. Then, $\lambda_{j}(\sum_{i=1}^n X_i) \le \sum_{i=1}^n \lambda_{\ceil{j/n}}(X_i)$. 
	\end{lemma}
	\begin{proof} Consider $n = 2$. Then, for any indices $k_1,k_2$ such that $k_1 + k_2 - 1 \le j$,  $\lambda_i(X_1 + X_2) \le \lambda_{k_1}(X_1) + \lambda_{k_2}(X_2)$. Hence, for general $n$,  
	\begin{align*}
	\lambda_i\left(\sum_{i=1}^n X_i\right) \le \lambda_{k_{1:n-1}}\left(\sum_{i=1}^{n-1}X_i\right) + \lambda_{k_n}(X_n),
	\end{align*} 
	where $k_n + k_{1:n-1} \le j + 1$. Iterating, 
	\begin{align*}
	\lambda_{k_{1:n-1}}\left(\sum_{i=1}^{n-1}X_i\right) \le \lambda_{k_{n-1}}\left(X_{n-1}\right) + \lambda_{k_{1:n-2}}\left(\sum_{i=1}^{n-2} X_i\right), 
	\end{align*} where $k_{n-1} + k_{1:n-2} \le k_{1:n-1} + 1$. Continuing, we have that for any $k_{1},\dots,k_n$ with $\sum_{i=1}^n k_i \le j + (n-1)$, 
	\begin{align*}
	\lambda_j\left(\sum_{i=1}^n X_i\right) \le \sum_{i=1}^n \lambda_{k_i}\left(X_i\right).
	\end{align*}
	Taking $k_i = \ceil{j/n}$, we have $\sum_{i=1}^n k_i = n\ceil{j/n}$. Now, if $j/n$ is integral, then $n \ceil{j/n} = j$. Otherwise, $\ceil{j/n} = \ceil{(j-1)/n} \le 1 + \frac{j-1}{n}$, so $n\ceil{j/n} \le n + j - 1$. In either case,  $\sum_{i=1}^n k_i \le j + (n-1)$, as needed.
	\end{proof}
	\begin{lemma}\label{lem:sv_multiplication} Let $X$ be a bounded operator on a Hilbert space $\calH$, and $\Lambda \succeq 0$ be a PSD operator on $\calH$. Then $\sigma_j(X\Lambda X^\herm) \le \|X\|^2_{\op} \sigma_j(\Lambda)$.
	\end{lemma}
	\begin{proof} Since $\sigma_j(X\Lambda X^\herm)  = \sigma_j(X\Lambda^{1/2})^2$, it suffices to show that for any two bounded operators $A,B$, $\sigma_j(AB) \le \|A\|\sigma_j(B)$. This follows using the variation representation of singular values:
	\begin{align*}
	\sigma_j(AB) &= \min_{ \text{ subspaces } \calV \subseteq \calH \text{ of dimension } j }\left( \max_{ \matq  \in \calV : \|\matq\|_{\calH} = 1} \|AB\matq\|_{\calH} \right)\\
	&\le \|A\|_{\op} \cdot \min_{ \text{ subspaces } \calV \subseteq \calH \text{ of dimension } j }\left( \max_{ \matq  \in \calV : \|\matq\|_{\calH} = 1} \|B\matq\|_{\calH} \right)\\
	&= \|A\|_{\op}\sigma_j(B).
	\end{align*}
	\end{proof}

	\newcommand{\ceiljnpl}[1][j]{\ceil{\frac{#1}{n+1}}}
	\begin{lemma}
	\label{lemma:comparing_spectrum}
	Let $A$ be a stable operator, $\Lambda \succeq 0$, and let $\Sigma = \dlyap{A^\herm}{\Lambda}$, and $P = \dlyap{A}{I}$. Then, for all indices $j, n \ge 1$, we can bound
	\begin{align*}
	\sigma_j (\Sigma) &\le \|P\|_{\op}^2 \left(\sigma_{\ceiljnpl}(\Lambda) + (1- \|P\|_{\op}^{-1})^n \opnorm{\Lambda}\right)\\
	\sum_{j \ge k} \sigma_j (\Sigma) &\le (n+1)\|P\|_{\op}^2 \left(\sum_{j \ge \ceil{\frac{k}{n+1}}} \sigma_{j}(\Lambda) + (1- \|P\|_{\op}^{-1})^n \traceb{\Lambda}\right).
	\end{align*}
	Moreover, from monotonicity of $\dlyapname$ (\Cref{lemma:basic_lyapunov_facts}), the above also holds for $P = \dlyap{A}{Z}$ for any $Z \succeq I$. In particular, if $n \ge \|P\|_{\op} \log \frac{2\opnorm{\Lambda}}{\lambda }$, then,
	\begin{align*}
	d_{\lambda}(\Sigma) \le (n+1) d_{\lambda/(2\|P\|_{\op}^2)}(\Lambda).\\
	\end{align*}
	\end{lemma}
	\begin{proof} Fix an integer $n$, and define the matrices
	\begin{align*}
	X_i :=  A^i \Lambda (A^\herm)^i, \quad Y_n := \sum_{i > n} X_i. 
	\end{align*}
	By \Cref{lem:weyl_ineq},
	\begin{align}
	\sigma_j (\Sigma) = \sigma_j(\sum_{i=1}^n X_i + Y_n) \le \sum_{i=1}^n \sigma_{\ceiljnpl}(X_i) + \sigma_{\ceiljnpl}(Y_n).
	\end{align}
	Now, for any index $k$, we can bound
	\begin{align*}
	\sigma_k(X_i) = \sigma_k( A^i \Lambda (A^\herm)^i) = \sigma_k(A^i \Lambda^{1/2}) \le \|A^i\|_{\op}^2 \sigma_k(\Lambda).
	\end{align*} 
	Hence,
	\begin{align}
	\sigma_j (\Sigma) &\le \left(\sum_{i=1}^n \|A^i\|_{\op}^2\right) \sigma_{\ceiljnpl}(\Lambda) + \sigma_{\ceiljnpl}(Y_n) \nonumber\\
	&\le \|P\|_{\op}^2 \sigma_{\ceiljnpl}(\Lambda) + \sigma_{\ceiljnpl}(Y_n),  \label{eq:sigma_j_bound}
	\end{align} 
	where in the last line we use \Cref{lemma:lyapunov_series_bound} to bound $\|A^i\|_{\op}^2$ by a geometric series. 

	Now, let us bound the operator norm and trace of $Y_n$. Again, by \Cref{lemma:lyapunov_series_bound}, we have
	\begin{align}
	\|Y_n\|_{\op} &= \opnorm{\sum_{i > n}  A^i \Lambda (A^\herm)^i } \le \sum_{i > n}\|A^i\|_{\op}^2 \opnorm{\Lambda} \le \|P\|_{\op}^2(1- \|P\|_{\op}^{-1})^n \opnorm{\Lambda} \label{eq:Yn_op_bound}
	\end{align}
	and similarly
	\begin{align}
	\|Y_n\|_{\mathsf{tr}} &\le \|P\|_{\op}^2(1- \|P\|_{\op}^{-1})^n \traceb{\Lambda}. \label{eq:Yn_tr_bound}
	\end{align}
	Combining these bounds, we have
	\begin{align}
	\sigma_j (\Sigma) &\le \|P\|_{\op}^2 \sigma_{\ceiljnpl}(\Lambda) + \sigma_{\ceiljnpl}(Y_n),  \tag{by \Cref{eq:sigma_j_bound}}\\
	&\le \|P\|_{\op}^2 \sigma_{\ceiljnpl}(\Lambda) + \opnorm{Y_n},  \tag{by \Cref{eq:Yn_op_bound}}\\
	&\le \|P\|_{\op}^2 \left(\sigma_{\ceiljnpl}(\Lambda) + (1- \|P\|_{\op}^{-1})^n \opnorm{\Lambda}\right). \nonumber
	\end{align}
	and, 
	\begin{align}
	\sum_{j \ge k} \sigma_j (\Sigma) &\le \|P\|_{\op}^2 \sum_{j \ge k} \sigma_{\ceiljnpl}(\Lambda) + \sum_{j \ge k} \sigma_{\ceiljnpl}(Y_n)  \tag{by \Cref{eq:sigma_j_bound}}\\
	&\overset{(i)}{\le} (n+1)\|P\|_{\op}^2 \sum_{j \ge \ceiljnpl[k]} \sigma_{j}(\Lambda) + (n+1)\sum_{j \ge \ceiljnpl[k]} \sigma_{j}(Y_n)\\
	&\le (n+1)\|P\|_{\op}^2 \sum_{j \ge \ceiljnpl[k]} \sigma_{j}(\Lambda) + (n+1)\traceb{Y_n} \nonumber\\
	&\le (n+1)\|P\|_{\op}^2 \left(\sum_{j \ge \ceiljnpl[k]} \sigma_{j}(\Lambda) + (1- \|P\|_{\op}^{-1})^n \traceb{\Lambda}\right), \tag{by \Cref{eq:Yn_tr_bound}}
	\end{align}
	concluding the proof. Here, in inequality $(i)$, we used that summing over the indices $\ceiljnpl{j}$ for $j \ge k$ sums over only indices $j' \ge \ceiljnpl[k]$, and includes each such index at most $n$ times. 

	To see why the last statement holds, for $a \defeq  (n+1)d_{\lambda / (2 \opnorm{P})}$,
	by the first statement, 
	\begin{align*}
		\sigma_a(\Sigma) \leq \opnorm{P}^2 \left( \sigma_{d_{\lambda / (2\opnorm{P})}}(\Lambda) + \exp(-n \opnorm{P}^{-1}) \opnorm{\Lambda}\right).
	\end{align*} 
	By definition of $d_\lambda$, $\opnorm{P}^2 \sigma_{d_{\lambda / (2\opnorm{P})}}(\Lambda) \leq \lambda/2$ and for $n \geq \opnorm{P} \log(2\traceb{\Lambda}.\lambda)$ the second term is also smaller than $\lambda /2$. Hence $d_\lambda(\Sigma)$ is at most $(n+1) d_{\lambda/(2\|P\|_{\op}^2)}(\Lambda)$.

	\end{proof}

\subsection{Frechet Differentiability\label{sec:Frechet}}

	\newcommand{\Banspace}{\mathscr{B}}
	\newcommand{\Banop}{\mathscr{B}_{\mathrm{op}}}
	\newcommand{\Symop}{\mathscr{S}_{\mathrm{op}}}
	\newcommand{\calT}{\mathcal{T}}

	\begin{definition} Let $\Banop$ denote the Banach space consisting of all operator-norm bounded operators on $\hilspace$, and let  $\Symop$ denote the subspace of self-adjoint bounded operators on $\hilspace$. 
	\end{definition}

	\begin{lemma}\label{lem:dlyap_banspace} Let $X \in \Banop$ have spectral radius $\rho(X) < 1$, and for $Y \in \Symop$, define $\calT_X[Y] = Y - X^\herm Y X $. Then, $\calT_X[\cdot]$ is a bounded linear operator on $\Symop$, and has an inverse which is also bounded. Hence, the solution to $\calT_X[Y] = Q$ is bounded for any $Q \in \Symop$, and given by $\dlyap{X}{Q}$.
	\end{lemma}
	\newcommand{\rmD}{\mathrm{D}}

	\begin{definition}[Frechet Derivative]\label{defn:frechet_differentiable} Let $\Banspace_1,\Banspace_2$ be a banach space with norms $\|\cdot\|_{\Banspace_1}$ and $\|\cdot\|_{\Banspace_2}$. Given a subset $\calU \subset \Banspace_1$ , we say $f: \calU \to \Banspace_2$ is \emph{Frechet Differentiable} at $x \in \Banspace_1$ if there exists linear mapping $\rmD f(x): \Banspace_1 \to \Banspace_2$ such that 
	\begin{align*}
	\lim_{\|h\|_{\Banspace_1} \to 0} \frac{\| f(x+h ) - \rmD f(x)[h])\|_{\Banspace_2}}{\|h\|} = 0.
	\end{align*}
	We say that $f$ is \emph{Frechet continuously differentriable} at $x$ if $f$ is Frechet differentiable in an open neighborhood about $x$, and $x \mapsto \rmD f(x)[h]$ is a continuous in the dual norm $\|g\|_{\Banspace_1^*} := \sup_{\|h\|_{\Banspace} = 1} g(h)$.
	\end{definition}
	An important case is where $f(t):\R \to \Banspace$ is a curve. In this case, the Frechet Derivative exists if there exists an $f'(t)$ such that
	\begin{align*}
	\lim_{\epsilon \to 0} \frac{\| f(t+\epsilon ) - f'(t)\|_{\Banspace}}{\epsilon} = 0.
	\end{align*}

	\begin{lemma}[Implicit Function Theorem \citep{whittlesey1965analytic}]\label{lem:imp_func_thm} Let $\Banspace_1,\Banspace_2$ be a two Banach spaces, and let $f : \Banspace_1 \times \Banspace_2 \to \Banspace_2$ be a continuously Frechet differentiable function at $(x_0,y_0) \in \Banspace_1 \times \Banspace_2$. Suppose that the mapping $ h_2 \mapsto \rmD f(x_0,y_0)[0,h_2]$ is a Banach space isopmorphism on $\Banspace_2$ (i.e, it admits a a bounded linear inverse). Then, there exists a neighborhood $\calU$ around $x_0$ and $\calV$ around $y_0$ and a Frechet-differentiable function $g: \calU \to \calV$ on which $g(x)$ is the unique solution to  $f(x,g(x)) = 0$.
	\end{lemma}

	\begin{lemma} Let $(A_1,B_1)$ and $(A_2,B_2)$ be two instance, and define the curve $A(t) = A_1 + t(B_2 - B_1)$ and $B(t) = B_1 + t(B_2 - B_1)$. If $(A(t),B(t))$ is stabilizable for all $t \in [0,1]$,
	\begin{itemize}
		\item $P(t) := \Pinfty{A(t)}{B(t)}$ is Frechet differentiable on $[0,1]$ in the operator norm $\|\cdot\|_{\op}$.
		\item The Frechet derivative of $P(t)$  is given by
		\begin{align}
	P'(t) = \dlyap{\Acl(t)}{E(t)} \text{ where } E(t) \defeq \Acl(t)^\herm P(t) \Delta_{\Acl}(t) + \Delta_{\Acl}(t)^\herm P(t) \Acl(t), \label{eq:E_of_t}
	\end{align}
	where $\Delta_{\Acl}(t) \defeq (A_2 - A_1) + (B_2 - B_1)K_\infty\left(
	A(t), B(t)\right)$.
	\end{itemize}
	\label{lemma:p_derivative}
	\end{lemma}
	\begin{proof} The result is the infinite-dimensional analogue of  Lemma 3.1 in \cite{simchowitz2020naive}. Define the function $\calF(t,P):= \R \times \Symop$ as 
	\begin{align*}
	\calF(t,P) := A(t)^\herm P A(t) - P + Q - A(t)^\herm P B(t)(R + B(t)^\herm P B(t))^{-1} B(t)^\herm P A(t)
	\end{align*}
	Then, $P(t)$ solves $\calF(t,P) = 0$. One can verify that $\calF(t,P)$ is a Frechet continuously differentiable function\footnote{Indeed, the first term in a polynomial in bounded operators, and the second term is a product of surge polynomials, and the matrix inverse $(R + B(t)^\herm P B(t))^{-1}$. Since $R + B(t)^\herm P B(t) \succeq R \succ 0$ is finite dimensional, one can express perturbations of $(R + B(t)^\herm P B(t))^{-1}$ as a convergence series, which can be used to verify continuous differentiability. } The computation of Lemma 3.1 in \cite{simchowitz2020naive} (carried out now in infinite dimensions) 
	shows that,
	\begin{align*}
	\left\{\frac{\rmd}{\rmd t}\calF(t,P) \cdot (\updelta t) + \frac{\rmd}{\rmd P}\calF(t,P) \cdot (\updelta P)\right\} \big{|}_{P = P(t)}=  \mathcal{T}_{A(t) + B(t)K(t)}[\updelta P] + E(t) \cdot \updelta t,
	\end{align*}
	where $K(t) = \Kinfty{A(t)}{B(t)}$, and where $E(t)$ is defined in \Cref{eq:E_of_t}, and $\mathcal{T}_{X}[Y] = Y - X^\herm Y X $. Since $(A(t),B(t))$ is stabilizable, $A(t) + B(t)K(t)$ is stable and thus  $\mathcal{T}_{A(t) + B(t)K(t)}[\cdot]$ is a bounded linear operator  on $\Symop$ with bounded inverse. It follows from the \Cref{lem:imp_func_thm} that $P(t)$ admits a Frechet Derivative $P'(t)$, and its derivative must $\mathcal{T}_{A(t) + B(t)K(t)}[\updelta P] + E(t) = 0$. Hence, by \Cref{lem:dlyap_banspace}, $P'(t)$ is given by \Cref{eq:E_of_t}.
	\end{proof}

	\begin{lemma}
	\label{lemma:dlyap_derivative} Let $A(t) = A_1 + (A_2 - A_1) t$ be a curve on $[0,1]$ with $A(t) \in \Banop$ stabilizable. Set $\Delta = A_2 - A_1$.Then, $t \mapsto \dlyap{A(t)}{\Sigma}$ is Frechet differentiable in the operator norm, and 
	\begin{align}
	\frac{\rmd}{\rmd t} \dlyap{A(t)}{\Sigma} = \dlyap{A(t)}{N}  \label{eq:N_of_t}
	\end{align} 
	where $N \defeq \Delta^\herm \dlyap{A(t)}{\Sigma} A(t) + A(t)^\herm \dlyap{A(t)}{\Sigma} \Delta$.
	\end{lemma}
	\begin{proof}
	Set $P(t) := \dlyap{A(t)}{\Sigma}$. Recall the operator $\calT_{X}[Y] = Y - X^\herm Y X$. Then, $P(t)$ solves $\calF(t,P) = 0$, where $\calF(t,P)  = \calT_{(A(t)}[P] - \Sigma$. It is straightforward to check that $\calF(t,P)$ is Frechet continuously differentiable. A direct computation reveals
	\begin{align*}
	\left\{\frac{\rmd}{\rmd t}\calF(t,P) \cdot (\updelta t) + \frac{\rmd}{\rmd P}\calF(t,P) \cdot (\updelta P)\right\} \big{|}_{P = P(t)}=  \mathcal{T}_{A(t)}[\updelta P] +N(t) \cdot \updelta t,
	\end{align*}
	where $N(t)$ is defined in the lemma. Moreover, since $A(t)$ is stabilizable, $\mathcal{T}_{A(t)}$ is a bounded linear operator with bounded inverse (\Cref{lem:dlyap_banspace}). Hence, the implicit function theorem (\Cref{lem:imp_func_thm}) shows that $P(t)$ admits a Frechet Derivative $P'(t)$, and its derivative must satisfy $\mathcal{T}_{A(t) + B(t)K(t)}[\updelta P] + N(t) = 0$. Consequently, by \Cref{lem:dlyap_banspace}, $P'(t)$ is given by \Cref{eq:E_of_t}.
	\end{proof}

	\begin{lemma}[Traces of Derivatives on $\Symop$]\label{lem:derivative_and_trace} Let $P(t):[0,1] \to \Symop$ be a continuously Frechet-differentiable curve. Moreover, let $\Gamma$ be a fixed trace-class operator on $\hilspace$. Then, 
	\begin{itemize}
		\item[(a)] The function $f(t) = \traceb{\Gamma P(t)}$ is a continously differentiable function, with derivative $f'(t) = \traceb{\Gamma P'(t)}$.
		\item[(b)] If $\Gamma$ is also PSD,
		\begin{align*}
		\nucnorm{\Gamma^{1/2}(P(1)-P(0))\Gamma^{1/2}} \le \max_{t \in [0,1]}\nucnorm{\Gamma^{1/2}P'(t)\Gamma^{1/2}}.
		\end{align*}
		\item[(c)] Similarly, if $\Gamma$ is PSD,
		\begin{align*}
		\normHS{\Gamma^{1/2}(P(1)-P(0))\Gamma^{1/2}} \le \max_{t \in [0,1]}\normHS{\Gamma^{1/2}P'(t)\Gamma^{1/2}}.
		\end{align*}
	\end{itemize}
	\end{lemma}
	\begin{proof} For any trace-class $\Gamma$, the mapping $P \mapsto \traceb{\Gamma P}$ is a bounded linear functional on the space $\Symop$. Hence, the map commutes with Frechet derivatives, and thus part (a) follows. 

	To prove part (b), write $\nucnorm{\Gamma^{1/2}(P(1)-P(0))\Gamma^{1/2}}  = \max_{X:\|X\| = 1} \trace(X \Gamma^{1/2} (P(1)-P(0)\Gamma^{1/2}) = \trace( \tilde{\Gamma} (P(1) -P(0)))$, where $\tilde{\Gamma} = \Gamma^{1/2} X \Gamma^{1/2}$. One can check that $\tilde{\Gamma}$ is also trace class (see, e.g. \cite{dragomir2014some}) and hence $f(t) =  \traceb{ \tilde{\Gamma} P(t)}$ is a continuously differentiable function by the first part of the lemma, with derivative $f'(t) = \traceb{ \tilde{\Gamma} P'(t)}$. Thus, from the mean-value theorem, $\trace( \tilde{\Gamma} (P(1) -P(0))) = f'(t) = \traceb{ \tilde{\Gamma} P'(t)}$ for some $t \in [0,1]$. Thus,
	\begin{align*}
	\nucnorm{\Gamma^{1/2}(P(1)-P(0))\Gamma^{1/2}} &= \max_{X:\|X\|_{\op} = 1} \trace\left[ \Gamma^{1/2} X \Gamma^{1/2} (P(1) - P(0))\right]\\
	&\le \max_{X:\|X\|_{\op} = 1} \max_{t \in [0,1]}\trace\left[ \Gamma^{1/2} X \Gamma^{1/2} P'(t)\right].
	\end{align*}
	Swapping the maxima and rearranging the trace, we see that the resulting term is just 
	\begin{align*}
	\max_{X:\|X\|_{\op} = 1}\max_{t \in [0,1]}\trace\left[  X \Gamma^{1/2} P'(t) \Gamma^{1/2}\right] = \max_{t \in [0,1]}\nucnorm{  \Gamma^{1/2} P'(t)\Gamma^{1/2} }.
	\end{align*}
	The proof of part (c) is nearly identical, except that the constraint on the variational parameter $X$ is strengthed to $\normHS{X} \le 1$.
	\end{proof}

\newpage
\part{Estimation Rates \label{part:learning}}
\section{Estimating System Operators  \label{app:estimation}}

In this section, we prove estimation rates for the system operators $\Ast, \Bst$ estimated via ridge regression on data collected under a single trajectory. In particular, throughout this section, as per our definition of the dynamics in \Cref{eq:dynamics_def}, we assume that data $\{(\matx_t, \matu_t,) \}_{t=1}^{T+1}$ is generated according to, 
\begin{align*}
	\matx_{t+1} = \Aclst \matx_t + \Bst\matu_t +\matw_t,
\end{align*}
where $\matu_t = \Kinit \matx_t + \matv_t$. Here, $\Kinit$ is a stabilizing controller, $\matv_t \iidsim \calN(0, \uvar I)$, and $\matw_T \iidsim \calN(0, \covw)$. As before, we let $\expcovone \defeq \Covst{\Kinit, \uvar} = \lim_{t \rightarrow} \E [\matx_t \otimes \matx_t] = \dlyap{(\Ast + \Bst \Kinit)^\herm}{\Bst \Bst^\herm \uvar + \covw}$ be the stationary state covariance when inputs are chosen according to the exploratory policy above. We define $\Aclst \defeq \Ast + \Bst \Kinit$ and estimate $\Aclst, \Bst$ via two separate regressions:
\begin{align}
\label{eq:acl_estimation}
	\Aclhat & \defeq \argmin_{\Acl} \frac{1}{T} \sum_{t=1}^T \frac{1}{2} \norm{\matx_{t+1} - \Acl \matx_t}^2 + \frac{\lambda}{2} \normHS{\Acl}^2\\
	\Bhat  &\defeq \argmin_{B}\frac{1}{T}\sum_{t=1}^{T}\frac{1}{2}\norm{\matx_{t+1} - B \matv_t}^2.
\label{eq:b_estimation}
\end{align}
The following lemma provides closed form expressions for the estimation error.
\begin{lemma}
\label{lemma:estimation_error}
Let $\Aclhat$ and $\Bhat$ be defined as in \Cref{eq:acl_estimation} and \Cref{eq:b_estimation}, then
\begin{align*}
\Aclst - \Aclhat &=   \lambda \Aclst \left( \frac{1}{T}\sum_{t=1}^T \matx_t \otimes \matx_t + \lambda I \right)^{-1} - \left( \frac{1}{T} \sum_{t=1}^T \left(\Bst\matv_t + \matw_t\right) \otimes \matx_t \right) \left( \frac{1}{T}\sum_{t=1}^T \matx_t \otimes \matx_t + \lambda I \right)^{-1} \\
\Bst - \Bhat &= -  \left (\sum_{t=1}^T \left(\Aclst \matx_t + \matw_t \right) \otimes \matv_t \right) \left(\sum_{t=1}^T \matv_t \otimes \matv_t \right)^{-1}.
\end{align*}
\end{lemma}
\begin{proof}
For the first statement, by taking the first order optimality conditions for the optimization problem in \eqref{eq:acl_estimation}, we have that,
\[
\Aclhat = \left( \Aclst \left(  \frac{1}{T}\sum_{t=1}^T \matx_t \otimes \matx_t \right) + \frac{1}{T} \sum_{t=1}^T \left(\Bst\matv_t + \matw_t\right) \otimes \matx_t  \right) \left( \frac{1}{T}\sum_{t=1}^T \matx_t \otimes \matx_t + \lambda I \right)^{-1}.
\]
Subtracting out the following quantity from both sides,
\[
\Aclst \left( \frac{1}{T}\sum_{t=1}^T \matx_t \otimes \matx_t + \lambda I \right) \left( \frac{1}{T}\sum_{t=1}^T \matx_t \otimes \matx_t + \lambda I \right)^{-1}
\]
and multiplying by $-1$ we get the first identity. The second follows directly from computing the optimality conditions for \eqref{eq:b_estimation}.
\end{proof}

\newpage
\subsection{Estimating $\Bst$}
We start by presenting the estimation guarantees for $\Bst$ since they are  significantly easier to prove than those for $\Ast$. This is because the covariates $\matv_t$ are independent and finite dimensional.

\begin{proposition}[B estimation]
\label{proposition:B_estimation}
Let  $T \gtrsim \max \left\{ \dimu  \log
\left(\dimu \right), \; \dimu \log\left( \frac{1}{\delta} \right)   \right\}$ then with probability $1-\delta$,
\begin{align*}
\normHS{\Bst - \Bhat}^2 \lesssim   \frac{ \dimu  \left( \traceb{\expcovone} +  \normHS{\expcovone} \log\left( \frac{\dimu}{\delta} \right) \right)
}{\uvar T}.
\end{align*}
\end{proposition}

\begin{proof}
Using \Cref{lemma:estimation_error} we know that the error in estimating $\Bst$ has the following form, which we can upper bound as follows,
\begin{align*}
	 \normHS{\sum_{t=1}^T \left(\Aclst \matx_t + \matw_t \right) \otimes \matv_t \left(\sum_{t=1}^T \matv_t \otimes \matv_t \right)^{-1}}^2 &\leq \frac{1}{\lambda_{\min} \left( \sum_{t=1}^T \matv_t \otimes \matv_t  \right)^2}
  \normHS{\sum_{t=1}^T \left(\Aclst \matx_t + \matw_t \right) \otimes \matv_t}^2. 
 \end{align*}
 
 \paragraph{Upper bounding Hilbert-Schmidt norm} Focusing on the Hilbert-Schmidt norm term, we let $\mate_1 \dots \mate_{\dimu}$ be the standard basis for $\R^{\dimu}$. Then, we can expand the Hilbert Schmidt norm as, 
 \begin{align}
 \label{eq:hs_norm_B}
 \normHS{\sum_{t=1}^T \left(\Aclst \matx_t + \matw_t \right) \otimes \matv_t}^2 = \sum_{i=1}^{\dimu} \underbrace{\norm{\sum_{t=1}^T \left(\Aclst \matx_t + \matw_t \right) \inner{\matv_t}{\mate_i}}^2}_{\defeq E_i}.
 \end{align} 
 Now, we deal with each $E_i$ individually. To do so, we notice that $\sum_{t=1}^T \left(\Aclst \matx_t + \matw_t \right) \inner{\matv_t}{\mate_i}$ is a mean zero gaussian vector in $\hilbx$ with covariance operator equal to, 
 \begin{align*}
 	\E \left[ \sum_{j,k=1}^T \left(\Aclst \matx_j + \matw_j \right) \otimes \left(\Aclst \matx_k + \matw_k \right)^\herm \inner{\matv_j}{\mate_i}\inner{\matv_k}{\mate_i} \right] & = \sum_{t=1}^T \left(\Aclst \E[\matx_t \otimes \matx_t] \Aclst^\herm  + \covw \right)\uvar  \\
 	& \preceq \uvar T \cdot  \expcovone. 
 \end{align*} 
The first equality follows from the fact that for $j \neq k$, $\inner{\matv_j}{\mate_i}$ and $\inner{\matv_k}{\mate_i}$ are independent and both mean zero.  Therefore, only the diagonal terms remain and by definition of $\matv_t$, $\E \inner{\matv_t}{\mate_i}^2 = \uvar$ for all $i$ and $t$. The upper bound in the second line is a consequence of the fact that, by definition of the dynamics and $\mathsf{dlyap}$, $\E \left[ \matx_t \otimes \matx_t \right] \preceq \expcovone$. Moreover, since $\expcovone$ is the solution to a Lyapunov equation, 
\begin{align*}
\Aclst \expcovone \Aclst^\herm  = \expcovone - \covw - \uvar \Bst \Bst^\herm,
\end{align*}
the $\covw$ cancel out. Now, by applying the Hanson-Wright inequality (\Cref{lemma:hansonwright}), with probability $1-\delta_i$,  $E_i$ is upper bounded by, 
\[
2 T \uvar \cdot  \traceb{\expcovone} + 5 T \uvar \cdot \log(1/\delta_i) \normHS{\expcovone}.
\]
Letting $\delta_i = \frac{1}{2\dimu} \delta$, we get that with probability, $1-\delta / 2$, the expression in \eqref{eq:hs_norm_B} is less than or equal to
\[
5 \dimu T  \uvar \cdot \left( \traceb{\expcovone} +   \log\left( \frac{2\dimu}{\delta} \right) \normHS{\expcovone}\right).
\]

\paragraph{Lower bounding minimum eigenvalue } To finish off the proof, we lower bound the minimum eigenvalue of $\sum_{t=1}^T \matv_t \otimes \matv_t$. Recall that $\matv_t \iidsim \calN(0, \uvar I)$ for all $t$. Hence, by \Cref{lemma:max_covariance_lower_bound} (and our lower bound on $T$), with probability $1-\delta/2$,
\[
\lambda_{\min}\left(\sum_{i=1}^n \matv_i \otimes \matv_i \right)^2 \geq (9 / 1600)^2 T^2 \sigma_{\matu}^4.
\]
Here, we have used that $M$, as defined in \Cref{lemma:max_covariance_lower_bound}, is upper bounded by $\uvar \dimu$ and that $\sigma_{\min}(\uvar I) = \uvar$.
\paragraph{Wrapping up}
Putting everything together, with probability $1-\delta$, for $T$ larger than the stated threshold,  
\[
\normHS{\Bst - \Bhat}^2 \lesssim \frac{ \dimu  \left( \traceb{\expcovone} +  \log\left( \frac{\dimu}{\delta} \right) \normHS{\expcovone}\right)
}{\uvar T}.
\]
\end{proof}

\subsection{Estimating $\Ast$}
\label{subsec:A_estimation}
For this subsection, we define the following the following quantities to simplify our notation,
\begin{align*}
\expcovonehat &\defeq \frac{1}{T} \sum_{t=1}^T \matx_t \otimes \matx_t \\ 
V_\lambda  &\defeq \left\{\matq \in \hilbx: \inner{\matq}{\expcovone \matq} \geq \lambda \right\} \\ 
t_0 &\defeq c_0 \opnorm{P_0}\log_+ \left(\lambda^{-1} \opnorm{P_0}^2  \opnorm{\Linit} \right),  
\end{align*}
where $P_0 \defeq \PinftyK{\Kinit}{\Ast}{\Bst}$ and $c_0$ is a universal constant.

Let $\projlambda$ denote the orthogonal projection operator onto the subspace $V_\lambda$. Similarly, let $\projnlambda$ denote the projection operator onto the orthogonal complement of $V_\lambda$. Recall our definition of $\Linit \defeq \Bst \Bst^\herm \uvar + \covw$. Lastly, we will make extensive reference to $d_\lambda$ and $\Ctail$ as defined in \Cref{theorem:main_regret_theorem}. In particular, if we let $(\sigma_j)_{j=1}^\infty = (\sigma_j(\expcovone))_{j=1}^\infty $ be the eigenvalues of $\expcovone$ then,
\begin{align*}
\textstyle d_\lambda \defeq |\{\sigma_j : \sigma_j \geq \lambda \}|, \quad \Ctail \defeq \frac{1}{\lambda}\sum_{ j >  d_\lambda} \sigma_j.
\end{align*}
Lastly, recall $\Ast \defeq \Ast + \Bst \Kinit$.

\begin{proposition}[A estimation]
\label{prop:A_estimation}
Let $T$ be such that
\[
	 T  \gtrsim \max \left\{ d_\lambda  \log_+
\left( \traceb{\expcovone} \lambda^{-1} \right), \; d_\lambda \log_+\left( 1 / \delta \right)   \right\} +  \opnorm{P_0}\log_+ \left(\lambda^{-1} \opnorm{P_0}^2  \opnorm{\Linit} \right).
\]
Then, with probability $1-\delta$,
\[
\normHS{(\Aclhat - \Aclst) \expcov}^2 \lesssim  \lambda \normHS{\Aclst}^2 + \frac{\logtracenorm}{T}(d_\lambda + \Ctail)\log_+ \left(  \frac{ \traceb{\expcovone}}{\delta \lambda}\right).
\]
Furthermore, the above guarantee holds for any $\overline{\Acl}$ equal to
\[
\overline{\Acl} \defeq \argmin_{\Acl \in \mathcal{B}} \inner{(\Acl - \Aclhat)}{(\expcovonehat + \lambda I)(\Acl - \Aclhat)}, \]
where $\mathcal{B}$ is a compact and convex set containing $\Aclst$.
\end{proposition}

\begin{proof}
Using the identity from \Cref{lemma:estimation_error},
\begin{align*}
(\Aclst - \Aclhat) \left( \frac{1}{T}\sum_{t=1}^T \matx_t \otimes \matx_t + \lambda I \right)^{1/2} &=   \lambda \Aclst \left( \frac{1}{T}\sum_{t=1}^T \matx_t \otimes \matx_t + \lambda I \right)^{-1/2} \\
&- \left( \frac{1}{T} \sum_{t=1}^T \left(\Bst\matv_t + \matw_t\right) \otimes \matx_t \right) \left( \frac{1}{T}\sum_{t=1}^T \matx_t \otimes \matx_t + \lambda I \right)^{-1/2}.
\end{align*}
Taking the Hilbert-Schmidt norm of both sides, we get that
\[
\normHS{(\Aclst - \Aclhat) \left( \frac{1}{T}\sum_{t=1}^T \matx_t \otimes \matx_t + \lambda I \right)^{1/2}}^2
\]
is less than or equal to,
\[
2 \underbrace{\normHS{\lambda \Aclst \left( \frac{1}{T}\sum_{t=1}^T \matx_t \otimes \matx_t + \lambda I \right)^{-1/2}}^2}_{N_1} + 2 \underbrace{\normHS{\left( \frac{1}{T} \sum_{t=1}^T \left(\Bst\matv_t + \matw_t\right) \otimes \matx_t \right) \left( \frac{1}{T}\sum_{t=1}^T \matx_t \otimes \matx_t + \lambda I \right)^{-1/2}}^2}_{N_2}.
\]
Since projections onto convex sets are non-expansive, incorporating a projection step as in the definition of $\overline{\Acl}$, doesn't change the above guarantee since
\[
\normHS{(\Aclst - \overline{\Acl}) \left(\expcovonehat + \lambda I \right)^{1/2}}^2 \leq \normHS{(\Aclst - \Aclhat) \left(\expcovonehat + \lambda I \right)^{1/2}}^2.
\]

\paragraph{Bounding the bias} Next, we bound each of these terms separately. For the first, we have that,
\[
	N_1 \leq \lambda \normHS{\Aclst}^2.
\]
\paragraph{Bounding the noise} For the second term, we apply \Cref{lemma:noise_bound_hs}, to get that with probability $1-\delta/2$,
\begin{align*}
	N_2 &\leq \frac{1}{T} \normHS{\left( \sum_{t=1}^T \left(\Bst\matv_t + \matw_t\right) \otimes \matx_t \right) \left(\sum_{t=1}^T \matx_t \otimes \matx_t + \lambda I \right)^{-1/2}}^2 \\
	& \lesssim \frac{1}{T}\logtracenorm \cdot  \left(d_\lambda  +  \Ctail \right) \log_+\left(\frac{\traceb{\expcovone}}{\delta \lambda} \right).
\end{align*}
Therefore, with probability $1-\delta/2$, 
\begin{align}
\normHS{(\Aclhat - \Aclst)(\expcovonehat + \lambda I)^{1/2}}^2 \lesssim \lambda \normHS{\Aclst}^2 + \frac{1}{T}\logtracenorm \cdot  \left(d_\lambda  +  \Ctail \right) \log_+\left(\frac{\traceb{\expcovone}}{\delta \lambda} \right)  .
\label{eq:final_error_bound_A}
\end{align}
Now, since we have chosen $T$ to be large enough, we apply \Cref{prop:empirical_cov_psd_bound} to get that with probability $1-\delta/2$
\[
\expcovone \preceq  c \cdot (\expcovonehat + \lambda I)
\]
where $c$ is a universal constant. This finishes the proof since, 
\begin{align*}
	\normHS{(\Aclhat - \Aclst) \expcov}^2 &= \traceb{(\Aclhat - \Aclst) \expcovone (\Aclhat - \Aclst)} \\
	&\leq c\cdot  \traceb{(\Aclhat - \Aclst) (\expcovonehat + \lambda I) (\Aclhat - \Aclst)} \\ 
	& = c\cdot \normHS{(\Aclhat - \Aclst)(\expcovonehat + \lambda I)^{1/2}}^2.\\ 
\end{align*}
 The final result then follows by applying a union bound and combining this last inequality with  \eqref{eq:final_error_bound_A}.
\end{proof}

\begin{lemma}
\label{lemma:noise_bound_hs}
The following inequality holds with probability $1-\delta$,
\[
\normHS{\sum_{t=1}^T \left(\Bst \matv_t + \matw_t \right) \otimes \matx_t \left( \sum_{t=1}^T \matx_t \otimes \matx_t +  \lambda T \cdot I \right)^{-1/2}}^2  \lesssim  \logtracenorm \cdot \log_+\left( \frac{\traceb{\expcovone}}{\delta \lambda} \right) \left(d_\lambda  +  \Ctail  \right)\]
\end{lemma}
\begin{proof}
Define $\matz_t \defeq \Bst \matv_t + \matw_t$. Note that $\matz_t \iidsim \calN(0, \Linit)$, where $\Linit$ is again defined as,
\[
\Linit = \Bst \Bst^\herm \uvar + \Sigma_\matw.
\]
Now, let $\sum_{j=1}^\infty \matq_j \otimes \matq_j \cdot \sigma_j$ be the eigendecomposition of this linear operator $\Linit$ where the $\matq_j$ form an orthonormal basis of $\hilbx$. By definition of the Hilbert-Schmidt norm,
\begin{equation}
\label{eq:hs_norm_noise}
		\normHS{\sum_{t=1}^T \matz_t \otimes \matx_t \left( \sum_{t=1}^T \matx_t \otimes \matx_t +  \lambda T \cdot I\right)^{-1/2}}^2 = \sum_{j=1}^\infty \underbrace{\norm{\left( \sum_{t=1}^T \matx_t \otimes \matx_t +  \lambda T \cdot I \right)^{-1/2} \sum_{t=1}^T \matx_t \inner{\matz_t}{\matq_j} }^2}_{\defeq E_j}.
\end{equation}
Due to our choice of basis, we have that $\inner{\matz_t}{\matq_j}$ is a zero-mean sub-Gaussian random variable with sub-Gaussian parameter $\sigma_j$. In particular, since the $\matz_t$ are drawn i.i.d from $\calN(0, \Linit)$, we have that for all $\gamma \in \R$,
\[
\E \exp(\gamma \inner{\matq_j}{\matz_t}) = \exp\left(\frac{\gamma^2}{2} \inner{\Linit \matq_j}{\matq_j}\right) = \exp\left(\frac{\gamma^2}{2} \sigma_j\right).
\]
\paragraph{Self-normalized inequality} Using this decompostion, we can bound each $E_j$ via a self-normalized martingale bound for vectors in Hilbert space. In particular by \Cref{lemma:self_normalized}, with probability $1-\delta_j$,
\begin{align*}
E_j &\leq 2 \sigma_j \log \left( \frac{\det \left(I + \frac{1}{\lambda T }\sum_{t = 1}^T \matx_t \otimes \matx_t \right)^{1/2}}{\delta_j} \right) \\
& = \sigma_j \log \det\left(I + \frac{1}{\lambda T }\sum_{t = 1}^T \matx_t \otimes \matx_t \right) + 2\sigma_j \log(1/\delta_j).
\end{align*}
Setting $\delta_j =\frac{3}{\pi^2} j^{-2} \cdot \delta$, with probability $1-\delta_j$, $E_j$ has the following upper bound,
\begin{align*}
E_j &\leq \sigma_j \log \det\left(I + \frac{1}{\lambda T }\sum_{t = 1}^T \matx_t \otimes \matx_t \right)   + 2 \sigma_j \log\left(\frac{\pi^2}{3\delta} \right)+ 4 \sigma_j \log(j).
\end{align*}
Using a union bound and summing up over all $j$, we get that with probability $1-\sum_{j=1}^\infty\delta_j = 1 -\delta / 2$,
\begin{align*}
\normHS{\sum_{t=1}^T \matz_t \otimes \matx_t \left( \sum_{t=1}^T \matx_t \otimes \matx_t +  \lambda T \cdot I \right)^{-1/2}}^2  &\leq  \left( \log \det\left(I + \frac{1}{\lambda T }\sum_{t = 1}^T \matx_t \otimes \matx_t \right) + 2 \log\left(\frac{\pi^2}{3\delta} \right) \right) \sum_{j=1}^\infty \sigma_j\\
 &+ 4 \sum_{j=1}^\infty \sigma_j \log(j).
\end{align*}
\paragraph{Bounding log determinant \& simplifying} Lastly, we can apply \Cref{lemma:log_det_bound},  to conclude that with probability $1-\delta/2$,
\[
\log \det\left(I + \frac{1}{\lambda T }\sum_{t = 1}^T \matx_t \otimes \matx_t \right) \lesssim d_\lambda \log\left(1 + \traceb{\expcovone} \frac{2}{\delta \lambda} \right) +  \lambda^{-1}\log_+(1/\delta) \cdot \traceb{\projnlambda \expcovone \projnlambda}.
\]
Now using the fact that $\sum_{j=1}^\infty \sigma_j = \traceb{\Bst \Bst^\herm \uvar  + \covw} = \traceb{\Linit}$, the target quantity is less than or equal to a universal constant times: 
\begin{align*}
\traceb{\Linit} \left(d_\lambda \log\left(1 + \traceb{\expcovone} \frac{2}{\delta \lambda} \right) +  \lambda^{-1}\log_+(1/\delta) \cdot \traceb{\projnlambda \expcovone \projnlambda} +  \log\left(\frac{\pi^2}{3\delta} \right) \right) +  \logtracenorm.
\end{align*}
Furthermore, by definition of $\logtracenorm$, $\traceb{\Linit} \leq \logtracenorm$. A short calculation shows this above satisfies, 
\begin{align*}
\lesssim  \logtracenorm \cdot \log_+\left( \frac{\traceb{\expcovone}}{\delta \lambda} \right) \left(d_\lambda  +   \Ctail  \right).
\end{align*}
\end{proof}

\begin{lemma}
\label{lemma:log_det_bound}
With probability $1-\delta$, 

\[
\log \left( \det \left(I + \frac{1}{\lambda T }\sum_{t = 1}^T \matx_t \otimes \matx_t \right) \right) \leq d_\lambda \log\left(1 + \traceb{\expcovone} \frac{2}{\delta \lambda} \right) + 7 \lambda^{-1}\log_+(2/\delta) \cdot \traceb{\projnlambda \expcovone \projnlambda}.
\]

\end{lemma}
\begin{proof}
Let $\sum_{i=1}^\infty \matq_i \otimes \matq_i \cdot \sigma_i$ be the eigendecomposition of $\frac{1}{\lambda T }\sum_{t = 1}^T \matx_t \otimes \matx_t$ and recall the following definitions.
\begin{align*}
V_\lambda &= \{\matz \in \hilbx: \inner{\matz}{\expcovone \matz} \geq \lambda \}\\
d_\lambda &= \mathsf{dim}(V_\lambda).
\end{align*}
\paragraph{Partitioning the spectrum} We recall that for any self-adjoint, PSD linear operator $M$, $\det(M)$ is equal to the product $\Pi_{i=1}^\infty \sigma_i$ of the eigenvalues $\{\sigma_i\}_{i=1}^\infty$ of $M$. Therefore, $\log \det(M)$ is equal to $\sum \log(\sigma_i)$ and $\log \det(I + M) = \sum \log(1 + \sigma_i)$. Using this identity, we now bound our target quantity as,
\begin{align*}
\log \left( \det \left(I + \frac{1}{\lambda T }\sum_{t = 1}^T \matx_t \otimes \matx_t \right) \right) & =  \sum_{i=1}^\infty \log(1 + \sigma_i) \\
& = \sum_{i: \matq_i \in V_\lambda} \log(1 + \sigma_i) + \sum_{i: \matq_i \notin V_\lambda} \log(1 + \sigma_i)\\
& \leq d_\lambda \log\left(1 + \sum_{i=1}^\infty \sigma_i \right)  + \sum_{i: \matq_i \notin V_\lambda}  \sigma_i \\
&  = d_\lambda \log\left(1 + \traceb{\frac{1}{\lambda T }\sum_{t = 1}^T \matx_t \otimes \matx_t}\right) + \traceb{\overline{S}_\lambda \frac{1}{\lambda T }\sum_{t = 1}^T \matx_t \otimes \matx_t \overline{S}_\lambda}.
\end{align*}
To go from the second to the third line, we used the inequality $\log(1 + x) \leq x$ for all $x \geq 0$ and the fact that $\dim(V_\lambda) = d_\lambda$. In the last line, we have used definition of $\overline{S}_\lambda$ as a projection operator. In further detail, 
\begin{align*}
\sum_{i: \matq_i \in V_\lambda} \log(1 + \sigma_i) \leq d_\lambda \log(1 + \sigma_1) \leq  d_\lambda \log\left(1 + \sum_{i=1}^\infty \sigma_i\right). 
\end{align*}
\paragraph{Bounding the top} Now, by Markov's inequality, with probability $1-\delta / 2$,
\[
\traceb{\frac{1}{\lambda T }\sum_{t = 1}^T \matx_t \otimes \matx_t} \leq \frac{2 \sum_{t=1}^T \E[\matx_t \otimes \matx_t]}{\lambda  \delta \cdot T} \leq  \traceb{\expcovone} \frac{2}{\delta \lambda}.  
\]
\paragraph{Bounding the tail} Define $\tilde{\matx}$ as, 
\[
	\tilde{\matx} = \begin{bmatrix}
	\projnlambda \matx_1 \\ 
	\projnlambda \matx_2 \\ 
	\dots \\
	\end{bmatrix}.
\]
Then, 
\begin{align*}
	\traceb{\projnlambda \frac{1}{\lambda T }\sum_{t = 1}^T \matx_t \otimes \matx_t \projnlambda} = \frac{1}{\lambda T}\sum_{t=1}^T \norm{\projnlambda \matx_t}^2 = \frac{1}{\lambda T} \norm{\tilde{\matx}}^2,
\end{align*}
where the last norm is taken in the relevant Hilbert space. Now, we notice that $\tilde{\matx}$ is a zero mean Gaussian. Furthermore, the trace norm of its covariance operator can be upper bounded as follows,
\begin{align*}
	\E \traceb{\tilde{\matx} \otimes \tilde{\matx}} = \sum_{t=1}^T \E \left[\traceb{\projnlambda \matx_t \otimes \matx_t \projnlambda} \right] \preceq T \cdot \traceb{\projnlambda \expcovone \projnlambda}. 
\end{align*}
Therefore, we can apply the Hanson-Wright inequality (\cref{lemma:hansonwright}) to conclude that with probability $1-\delta/2$, 
\begin{align*}
	\traceb{\projnlambda \frac{1}{\lambda T }\sum_{t = 1}^T \matx_t \otimes \matx_t \projnlambda} &\leq 2 \lambda^{-1} \traceb{\projnlambda \expcovone \projnlambda} + 5 \lambda^{-1} \log(2/\delta) \traceb{\projnlambda \expcovone \projnlambda} \\ 
	&\leq 7 \lambda^{-1}\log_+(2/\delta) \traceb{\projnlambda \expcovone \projnlambda}.
\end{align*}

\paragraph{Finishing the proof} In conclusion, by combining the previous parts, we get that with probability $1-\delta$,
\[
\log \left( \det \left(I + \frac{1}{\lambda T }\sum_{t = 1}^T \matx_t \otimes \matx_t \right) \right) \leq d_\lambda \log\left(1 + \traceb{\expcovone} \frac{2}{\delta \lambda} \right) + 7 \lambda^{-1}\log_+(2/\delta) \cdot \traceb{\projnlambda \expcovone \projnlambda}.
\]

\end{proof}

\begin{proposition}
\label{prop:empirical_cov_psd_bound}
	For $t_0$ and $d_\lambda$ defined as in introduction to  \Cref{subsec:A_estimation},    if
	\[
T - t_0 \gtrsim \max \left\{  d_\lambda  \log_+
\left( \frac{\traceb{\expcovone}}{\lambda} \right), \; d_\lambda \log_+\left( 1 / \delta \right)   \right\}
	\]
	then with probability $1-\delta$,
	\[
	\expcovone \preceq c \cdot \left( \expcovonehat + \lambda I \right)
	\]
	where $c$ is a universal constant.
\end{proposition}
\begin{proof}
	Recall our definition of  $V_\lambda$ as the subspace of $\hilbx$ corresponding to the directions where the state covariance $\expcovone$ has large eigenvalues,
	\[
	V_\lambda = \left\{\matq: \inner{\matq}{\expcovone \matq} \geq \lambda \right\}.
	\]
	Furthermore, we review the definitions of $\projlambda$, the orthogonal projection operator onto the subspace $V_\lambda$, and $\projnlambda$, the orthogonal projection operator on the complement of $V_\lambda$. Since $\expcovone$ is a trace class operator, $V_\lambda$ is a finite dimensional subspace. Therefore, $d_\lambda = \mathsf{dim}(V_\lambda) < \infty$.
  
  For any vector $\matz \in \hilbx$, we have that,
	\begin{align*}
		\inner{\matz}{\expcovone \matz} &\leq 2 \inner{S_\lambda \matz}{\expcovone S_\lambda \matz} + 2\inner{\overline{S}_\lambda \matz}{\expcovone \overline{S}_\lambda \matz} \\
		& \leq 2 \inner{S_\lambda \matz}{\expcovone S_\lambda \matz} + 2 \lambda.
	\end{align*}
In order to complete the proof, it suffices to show that,
\[
S_\lambda \expcovone  S_\lambda \preceq c \cdot S_\lambda \expcovonehat S_\lambda.
\]
In short, we have reduced the proof to showing a PSD upper bound for finite dimensional operators. By \Cref{lemma:xt_covariance_lower_bound}, for $t \geq t_0$, we have that $S_\lambda \matx_t \sim \calN(0, \Sigma_t)$ for $\Sigma_t \succeq \frac{\lambda}{2} I $. Therefore, for $T - t_0$ greater than the lower bound in the statement of the proposition, we can apply \Cref{lemma:max_covariance_lower_bound}, to conclude that with probability $1-\delta$,
\[
S_\lambda \expcov  S_\lambda \preceq c \cdot S_\lambda \expcovonehat S_\lambda.
\]
This concludes the proof.
\end{proof}

\begin{lemma}
\label{lemma:xt_covariance_lower_bound}
For $t \geq t_0$, $S_\lambda \E \left[\matx_t \otimes \matx_t \right] S_\lambda \succeq \frac{\lambda}{2} I$, where $I \in \R^{d_\lambda \times d_\lambda}$.
\end{lemma}
\begin{proof}
Let $\matv \in V_\lambda$ be a unit vector. Then, by definition of $V_\lambda$, $\inner{\matv}{\expcovone \matv} \geq \lambda$. Furthermore, by properties of the dynamical system,
\begin{align*}
	\E  \left[ \matx_t \otimes \matx_t \right]   &=   \sum_{j=0}^{t-2}\Aclst^j \left(\Bst \Bst \uvar + \Sigma_\matw \right) \left(\Aclst^j\right)^\herm.
\end{align*}
Before moving on, we recall the form of the steady-state covariance operator $\expcovone$,
\[\expcovone = \dlyap{\Aclst^\herm}{\Bst \Bst \uvar + \Sigma_\matw} = \sum_{j=0}^{\infty}\Aclst^j \left(\Bst \Bst \uvar + \Sigma_\matw \right) \left(\Aclst^j\right)^\herm.
\]
By the previous two equations, we have that:
\[
\inner{\matv}{\expcovone \matv} -  \inner{\matv}{\E \matx_t \otimes \matx_t \matv}=  \inner{\matv}{\sum_{j=t-1}^\infty \Aclst^j \left(\Bst \Bst \uvar + \Sigma_\matw \right) \left(\Aclst^j\right)^\herm \matv}.
\]
Therefore for any $\matv \in V_\lambda$,
\begin{align*}
	\inner{\matv}{\E \matx_t \otimes \matx_t \matv} & = \inner{\matv}{\expcovone \matv} -  \inner{\matv}{\sum_{j=t-1}^\infty \Aclst^j \left(\Bst \Bst \uvar + \Sigma_\matw \right) \left(\Aclst^j\right)^\herm \matv} \\
	&\geq  \lambda - \inner{\matv}{\sum_{j=t-1}^\infty \Aclst^j \left(\Bst \Bst \uvar + \Sigma_\matw \right) \left(\Aclst^j\right)^\herm \matv}.
\end{align*}
To finish the proof, we show that for any $\matv \in V_\lambda$, the second term is no smaller than $-\lambda/2$. To do so, we proceed as in \Cref{lemma:dlyapm_bound}. Recall that $\Bst \Bst \uvar + \Sigma_\matw = \Linit$,
\begin{align*}
\inner{\matv}{\sum_{j=t-1}^\infty \Aclst^j \left(\Bst \Bst \uvar + \Sigma_\matw \right) \left(\Aclst^j\right)^\herm \matv} &\leq \opnorm{\Bst \Bst \uvar + \Sigma_\matw} \sum_{j=t-1}^{\infty} \opnorm{\Aclst^j (\Aclst^\herm)^j}\\
& = \opnorm{\Linit} \sum_{j=t-1}^{\infty} \opnorm{(\Aclst^\herm)^j \Aclst^j}\\
& \leq \opnorm{\Linit} \sum_{j=t-1}^{\infty} \opnorm{(\Aclst^\herm)^j P_0 \Aclst^j} \\ 
& \leq \opnorm{\Linit} \opnorm{P_0} \sum_{j=t-1}^\infty (1 - \opnorm{P_0}^{-1})^j \\
& = \opnorm{\Linit} \opnorm{P_0}^2 (1 - \opnorm{P_0}^{-1})^{t-1}\\ 
& \leq \opnorm{\Linit} \opnorm{P_0}^2 \exp\left( -\frac{(t-1)}{\opnorm{P_0}} \right).
\end{align*}
The fourth inequality follows from applying \Cref{lemma:lyapunov_series_bound}. For $t \geq 1 + \ceil{\opnorm{P_0}\log \left(\frac{2}{\lambda} \opnorm{P_0}^2  \opnorm{\Linit} \right)}$, the expression above is smaller than $\frac{\lambda}{2}$. Lastly, we note that since $\opnorm{P_0} > 1$, 
\[
t_0 \defeq 3\opnorm{P_0}\log_+ \left(\frac{2}{\lambda} \opnorm{P_0}^2  \opnorm{\Linit} \right) \geq 1 +  \left\lceil\opnorm{P_0}\log \left(\frac{2}{\lambda} \opnorm{P_0}^2  \opnorm{\Linit} \right)\right\rceil.
\]
\end{proof}

\subsection{Formal Statement and Proof of \Cref{prop:warm_start_estimates}}

%


\newtheorem*{formalprop:warm_start}{\Cref{prop:warm_start_estimates}}
\begin{formalprop:warm_start}
\label{formalprop:warm_start_estimates}
Assume that \Cref{assumption:alignment} holds, and define $\Delta_r \defeq \frac{1}{16}\Cstable^2 - s_{r+1}^2$,  $\lambdast \defeq c \frac{\Delta_r}{ \rho}$, where $\rho, r$ are defined as in \Cref{assumption:alignment} and $c$ is a universal constant. Then, for $T \geq \Tinit$, and $(A_0, B_0)$ computed as in the $\warmstart$ algorithm,  with probability $1-\delta$
\[
		 \max \left\{ \opnorm{A_0- \Ast}, \opnorm{B_0 - \Bst} \right\} \leq \frac{1}{2} \Cstable.
\]
Here,  $\Pinit = \PinftyK{\Kbase}{\Ast}{\Bst}$, $\initcovone \defeq \Covst{\Kbase, \uvar}$, and $\Tinit$ is equal to a universal constant times the maximum of the following three quantities,
\begin{enumerate}[a)]
	\item $ d_{\lambdast} \log_+
\left( \frac{\traceb{\initcovone}}{\delta\lambdast} \right) + \opnorm{\Pinit}\log_+ \left(\frac{1}{\lambdast} \opnorm{\Pinit}^2  \opnorm{\Bst \Bst^\herm \uvar + \covw} \right) $
	\item $\frac{ \logtracenorm}{\lambdast \Delta_{r}} \log_+\left(  \frac{\traceb{\initcovone}}{\delta \lambdast} \right) \left(d_{\lambdast}  +  \calC_{\mathsf{tail}, \lambdast}  \right)$
	\item $\max\{1, \opnorm{\Kbase}^2\} \cdot \frac{ \dimu  \left( \traceb{\initcovone} +  \normHS{\initcovone} \log\left( \frac{\dimu}{\delta} \right) \right)
}{\uvar \Cstable^2}$.
\end{enumerate}
\end{formalprop:warm_start}


\begin{proof}
The overall proof strategy is to show there exists a threshold such that the error in both $\Ast$ and $\Bst$ is small. We deal with operator error separately. Assume $\Aclhat, A_0,$ and $B_0$ are computed as in $\warmstart$ (see \Cref{app:algorithm_descriptions}). Furthermore, redefine $\Aclst \defeq \Ast + \Bst \Kbase$.

\paragraph{$\Bst$ recovery} This part is simple since we have consistent parameter recovery for $\Bst$. Recalling \Cref{proposition:B_estimation}, the following statement holds with probability $1-\delta/2$,
\begin{align*}
\normHS{\Bst - B_0}^2 \lesssim  \frac{ \dimu  \left( \traceb{\initcovone} +  \normHS{\initcovone} \log\left( \frac{\dimu}{\delta} \right) \right)
}{\uvar T}.
\end{align*}
Since $\Tinit \gtrsim \max\{1, \opnorm{\Kbase}^2\} \frac{ \dimu  \left( \traceb{\initcovone} +  \normHS{\initcovone} \log\left( \frac{\dimu}{\delta} \right) \right)
}{\uvar \Cstable^2}$ we have that with probability $1-\delta/2$, $$\normHSil{\Bst - B_0} \leq \frac{1}{4 \max\{1, \opnorm{\Kbase}\}}\Cstable. $$

\paragraph{$\Ast$ recovery} In order to guarantee a close estimate of $\Ast$, we use the error decomposition from \Cref{lemma:estimation_error} and show that this quantity is small in operator norm if the alignment condition holds. For the sake of making notation concise, we let,$\initcovonehat \defeq \frac{1}{T}\sum_{t=1}^T \matx_t \otimes \matx_t$ and $\Aclst = \Ast + \Bst \Kbase$. From our earlier error decomposition, 
\begin{align*}
\opnorm{\Aclhat - \Aclst}^2 &= \opnorm{\lambda \Aclst \left( \initcovonehat + \lambda I \right)^{-1} - \left( \frac{1}{T} \sum_{t=1}^T \left(\Bst\matv_t + \matw_t\right) \otimes \matx_t \right) \left( \initcovonehat + \lambda I \right)^{-1}}^2 \\
& \leq  2 \underbrace{\opnorm{\lambda \Aclst \left( \initcovonehat + \lambda I \right)^{-1}}^2}_{\defeq N_1} +  2 \underbrace{\normHS{\left( \frac{1}{T} \sum_{t=1}^T \left(\Bst\matv_t + \matw_t\right) \otimes \matx_t \right) \left( \initcovonehat + \lambda I \right)^{-1}}^2}_{\defeq N_2}.
\end{align*}
\paragraph{Bounding the noise } The bound on $N_2$ follows a simple application of \Cref{lemma:noise_bound_hs}.
\begin{align*}
	\normHS{\left( \frac{1}{T} \sum_{t=1}^T \left(\Bst\matv_t + \matw_t\right) \otimes \matx_t \right) \left( \initcovonehat + \lambda I \right)^{-1}}^2 &\leq \frac{1}{\lambda T^2} \normHS{\left(  \sum_{t=1}^T \left(\Bst\matv_t + \matw_t\right) \otimes \matx_t \right) \left( \initcovonehat + \lambda I \right)^{-1/2}}^2 \\
	& \leq  \frac{1}{\lambda T} \normHS{\left(  \sum_{t=1}^T \left(\Bst\matv_t + \matw_t\right) \otimes \matx_t \right) \left(T \initcovonehat +  \lambda T  I \right)^{-1/2}}^2.
\end{align*}
Applying \Cref{lemma:noise_bound_hs}, we get that with probability $1-\delta/4$,  
\begin{align*}
 N_2 \lesssim \frac{ \logtracenorm}{\lambda T} \cdot \log_+\left(  \frac{\traceb{\initcovone}}{\delta \lambda} \right) \left(d_\lambda  +  \Ctail  \right).
\end{align*}

\paragraph{Bounding the bias} The bound on $N_1$ follows from the alignment condition. In order to apply it, we first perform the following simplification of the bias term. 
\begin{align}
	\opnorm{\lambda \Aclst \left( \initcovonehat + \lambda I \right)^{-1}}^2 &\leq \lambda \opnorm{\Aclst \left( \initcovonehat + \lambda I \right)^{-1/2}}^2 \nn \\
	& = \lambda  \opnorm{\left(\initcovonehat + \lambda I \right)^{-1/2} \Aclst^\herm \Aclst \left(\initcovonehat + \lambda I \right)^{-1/2} } \label{eq:trace_AA}.
\end{align}
As in the alignment condition, we let $U(\Lambda_r + \Lambda_{/r})V^\herm$ be the SVD of $\Aclst$. Therefore, $\Aclst^\herm \Aclst $ is equal to $V(\Lambda_r^2 + \Lambda_{/r}^2)V^\herm$. Applying the triangle inequality, we can then bound \eqref{eq:trace_AA} by,
\begin{align}
\lambda  \opnorm{\left(\initcovonehat + \lambda I \right)^{-1/2} V \Lambda_{r}^2 V^\herm \left(\initcovonehat + \lambda I \right)^{-1/2} }  + \lambda \opnorm{\left(\initcovonehat + \lambda I \right)^{-1/2}  V\Lambda_{/r}^2V^\herm  \left(\initcovonehat + \lambda I \right)^{-1/2}}.
\label{eq:identifiability_topandbottom}
\end{align}
We can bound the second term above by $\opnormil{\Lambda_{/r}^2} = s_{r+1}^2$. Next, we have chosen $\Tinit$ large enough so that $\initcovonehat + \lambda I$ is a PSD upper bound on $\initcovone$ as per \Cref{prop:empirical_cov_psd_bound}. Together with the alignment condition, we have that with probability $1- \delta/4$,
\begin{align*}
	 V\Lambda_{r}^2V^\herm \preceq \rho \initcovone \preceq c \cdot  \rho (\initcovonehat + \lambda I),
\end{align*} 
for some universal constant $c$. This implies that,
\begin{align*}
\left(\initcovonehat + \lambda I \right)^{-1/2}  V\Lambda_{r}^2V^\herm \left(\initcovonehat + \lambda I \right)^{-1/2} \preceq c \rho I,
\end{align*}
and hence the first term in \Cref{eq:identifiability_topandbottom} is smaller than $c\cdot \lambda \rho$. Putting everything together, we get that with probability $1-\delta/4$, $N_1$ is less than or equal to $c \rho \lambda + s_{r+1}^2$.
Therefore, with probability $1-\delta/2$,
\begin{align}
\label{eq:A_error_before_lambda}
\opnorm{\Aclhat - \Aclst}^2 \leq c_0\frac{ \logtracenorm}{\lambda T} \cdot \log_+\left(  \frac{\traceb{\initcovone}}{\delta \lambda} \right) \left(d_\lambda  +  \Ctail  \right) + c_1\rho \lambda + s_{r+1}^2. 
\end{align}
Now, we let $\Delta_r \defeq  \frac{1}{16}\Cstable^2 - s_{r+1}^2$ which is strictly great than $0$ by \Cref{assumption:alignment}. Setting $\lambda = \lambdast \defeq c'\frac{\Delta_r}{\rho}$ for some universal constant $c'$, we have that $c_1\rho \lambda \leq \frac{1}{2} \Delta_r$. Furthermore, for $T$ such that,
\[
T \gtrsim \frac{\logtracenorm}{\lambdast \Delta_{r}} \log_+\left(  \frac{\traceb{\expcovone}}{\delta \lambdast} \right) \left(d_{\lambdast}  +  \calC_{\mathsf{tail}, \lambdast}  \right), 
\] 
the first term in \Cref{eq:A_error_before_lambda} is less than or equal to $\Delta_r / 2$ and hence $\opnormil{\Aclhat - \Aclst}^2 \leq \frac{1}{16}\Cstable^2$. Computing $A_0 = \Aclhat - B_0 \Kbase$, we get that: 
\begin{align*}
	\opnorm{A_0 - \Ast} & = \opnorm{(\Aclhat  - B_0 \Kbase )\pm \Aclst - \Ast} \\ 
	&  \leq \opnorm{\Aclhat - \Aclst} + \opnorm{(\Bst - B_0) \Kbase} \\ 
	& \leq \frac{1}{2} \Cstable.
\end{align*}
This concludes the proof.
\end{proof}

\subsection{Estimation Lemmas}
\label{subsec:estimation_lemmas}

\begin{lemma}[Theorem 2.6 in \cite{hansonwright}]
\label{lemma:hansonwright}
Let $\matv_i  \in \hilbx$ be independent random vectors in a Hilbert space $\hilbx$ such that $ \matv_i \sim \calN(0, \Sigma_i)$ and $\Sigma_i \preceq \Sigma$ for all $i$. Then,
\[
\Pr\left[ \sum_{i=1}^n\norm{\matv_i}^2 \geq 2n \traceb{\Sigma} + 5t \normHS{\Sigma} \right] \leq \exp\left(-t   \right)
\]
\end{lemma}

\begin{proof}
The lemma is a restatement of Theorem 2.6 in \cite{hansonwright}. Since the $\matv_i$ are Gaussians, the inequality immediately following Equation 4.2 in the proof of Theorem 2.6 in \cite{hansonwright} can be restated as,
\[
\Pr\left[ \sum_{i=1}^n\norm{\matv_i}^2 \geq n \traceb{\Sigma} + t \right] \leq \exp\left(- \lambda t + 2 n \lambda^2 \normHS{\Sigma}^2  \right) \text{ for all } 0 \leq \lambda < \frac{1}{4 \opnorm{\Sigma}}.
\]
Since $\opnorm{\Sigma} \leq \normHS{\Sigma}$, if we set $\lambda = (5 \normHS{\Sigma})^{-1}$, we get that
\[
\Pr\left[ \sum_{i=1}^n\norm{\matv_i}^2 \geq n \traceb{\Sigma} + t \right] \leq \exp\left(- \frac{t}{5 \normHS{\Sigma}}  + \frac{2}{25} n   \right).
\]
Lastly, setting $-t' = - \frac{t}{5 \normHS{\Sigma}}  + \frac{2}{25} n$,
\[
\Pr\left[ \sum_{i=1}^n\norm{\matv_i}^2 \geq n \traceb{\Sigma} + 5t' \normHS{\Sigma} + \frac{2}{5}n \normHS{\Sigma} \right] \leq \exp\left(-t'   \right).
\]
The proof follows since $\normHS{\Sigma} \leq \traceb{\Sigma}$.
\end{proof}

\begin{lemma}[Lemma E.4 in \cite{simchowitz2020naive}]
\label{lemma:max_covariance_lower_bound}
Let $\calF_t$ be a filtration such that $\matz_t \mid \calF_{t-1} \sim \calN(0, \Sigma_{_t})$ where $\Sigma_{t} \in \R^{d \times d}$ is $\calF_{t-1}$-measurable, and $\Sigma_t \succeq \Sigma$. Furthermore, assume that
\[
  \E \;   \traceb{ \frac{1}{T} \sum_{t=1}^T\matz_t \otimes \matz_t}  \leq M.
\]
Then for,
\[
T \geq 223 \max \left\{ 2d \log(\frac{100}{3}) + d  \log
\left( \frac{M}{\lambda_{\min}(\Sigma)} \right), \; (d+1) \log\left( \frac{2}{\delta} \right)   \right\},
\]
with probability $1-\delta$,
\[
\frac{1}{T}\sum_{t=1}^T \matz_t \otimes \matz_t \succeq \frac{9}{1600} \Sigma.
\]
\end{lemma}

\begin{lemma}[Corollary 3.6 in \cite{yasinthesis}]
\label{lemma:self_normalized}
Let $(\calF_k, k \geq 1)$ be a filtration and let $(\matm_k, k\geq 1)$ be an $\hilbx$-valued stochastic process adapted to $\calF_k$, and $(\eta_k, k \geq 2)$ be a real valued martingale difference process adapted to $\calF_k$. Furthermore, assume that $\eta_k$ is conditionally sub-Gaussian in the sense that there exists a $\sigma > 0$ such that, $\E\left[\exp(\eta_k \gamma)\right]  \leq \exp\left( \gamma^2 \sigma / 2\right)$
for all $\gamma \in \R$. Consider the martingale and operator-valued processes, 
\begin{align*}
	S_t \defeq \sum_{k=1}^t \eta_{k+1} \cdot \matm_{k}, \quad  V_t \defeq \sum_{k=1}^{t-1} \matm_k \otimes \matm_k, \quad \overline{V}_t \defeq \lambda I + V_t \text{ for } t \geq 0.
\end{align*}
%
Then, for any $\delta \in (0,1)$ with probability $1-\delta$,
\begin{align*}
	\forall t \geq 2, \quad \norm{\overline{V}_t^{-1/2} S_t} \leq 2 \sigma \log \left( \frac{\det(I + \lambda^{-1} V_t)}{\delta}\right). 
\end{align*}
\end{lemma}

\newpage
\part{Regret Bounds \label{part:regret}}

\section{Algorithm Descriptions: OnlineCE and WarmStart} 
\label{app:algorithm_descriptions}
Below, we let $\mathcal{B} \defeq \{\Acl: \opnorm{\Acl - (A_0 + B_0 K_0)} \leq .5 \Cstable \}$ denote an operator norm ball around the warm start estimate $A_0 + B_0 \Kinit$ and let $\expcovonehat = T^{-1} \sum_{t=1}^T \matx_t \otimes \matx_t$ be the empirical state covariance. The precise description of $\Texplore$ and $\lambda$ may be found in the proof of \Cref{theorem:main_regret_theorem}.
\begin{figure}[h!]
\label{fig:onlinece_description}
\setlength{\fboxsep}{2mm}
\begin{boxedminipage}{\textwidth}
\begin{center}
$\mainalg$ \\
\end{center}
{\bf Input:} Warm start estimates $(A_0, B_0)$, confidence $\delta$, horizon length $T$
\begin{enumerate}
	\item Synthesize controller $\Kinit = \Kinfty{A_0}{B_0}$
	\item Collect data under exploration policy 
	\vspace{1mm}
	
	\hspace{5mm}{{ For} $t = 1,2, \ldots, \Texplore:$
	\begin{itemize}[itemindent=10mm]
	\item Observe state $\matx_t$
	\item Choose input $\matu_t = \Kinit \matx_t + \matv_t$ where $\matv_t \sim \calN(0, I)$
	\end{itemize}}
	\vspace{1mm}
	\item Estimate $\Bst$ 
	\[\Bhat  = \argmin_{B} \sum_{t=1}^{\Texplore}\frac{1}{2\Texplore}\normHx{\matx_{t+1} - B \matv_t}^2\]
\item Estimate $\Ast$
\begin{enumerate}[nolistsep, noitemsep]
	\item Compute initial estimate via ridge regression 
	\begin{align*}
\widetilde{\Acl}  &\defeq \textstyle\argmin_{\Acl} \frac{1}{\Texplore} \sum_{t=1}^{\Texplore} \frac{1}{2} \normHx{\matx_{t+1} - \Acl \matx_t}^2 + \frac{\lambda}{2} \normHS{\Acl}^2	\end{align*}
	\item Project to safe set 
	\begin{align*}
			\Aclhat \defeq \argmin_{\Acl \in \mathcal{B}} \inner{(\Acl - \widetilde{\Acl})}{(\expcovonehat + \lambda I)(\Acl - \widetilde{\Acl})} 
	\end{align*}
	\item Refine estimate $$\Ahat \defeq \Aclhat - \Bhat \Kinit$$ 
\end{enumerate}
\vspace{1mm}
\item Synthesize certainty equivalence controller $\Khat = \Kinfty{\Ahat}{\Bhat}$
\item Choose inputs according to $\Khat$ for remainder of horizon.
\vspace{1mm}
	
	\hspace{5mm}{{ For} $t = \Texplore + 1,\ldots,T:$
	\begin{itemize}[itemindent=10mm]
	\item Observe state $\matx_t$
	\item Choose input $\matu_t = \Khat \matx_t$
	\end{itemize}}
	\vspace{1mm}
\end{enumerate}
\vspace*{2mm}
\vspace*{2mm}
\end{boxedminipage}
\end{figure}

\begin{figure}[h!]
\label{figure:warm_start_description}
\setlength{\fboxsep}{2mm}
\begin{boxedminipage}{\textwidth}
\begin{center}
$\warmstart$ \\
\end{center}
{\bf Input:} Initial controller $\Kbase$ which stabilizes $(\Ast, \Bst)$.
\begin{enumerate}
	\item Collect data under exploration policy 
	\vspace{1mm}
	
	\hspace{5mm}{{ For} $t = 1,2, \ldots, \Tinit:$
	\begin{itemize}[itemindent=10mm]
	\item Observe state $\matx_t$
	\item Choose input $\matu_t = \Kinit \matx_t + \matv_t$ where $\matv_t \sim \calN(0, I)$
	\end{itemize}}
	\vspace{1mm}
	\item Estimate $\Bst$ 
	\[B_0 = \argmin_{B} \sum_{t=1}^{\Texplore}\frac{1}{2\Texplore}\normHx{\matx_{t+1} - B \matv_t}^2\]
\item Estimate $\Ast$
\begin{enumerate}[nolistsep, noitemsep]
	\item Compute initial estimate via ridge regression 
	\begin{align*}
\Aclhat  &\defeq \textstyle\argmin_{\Acl} \frac{1}{\Texplore} \sum_{t=1}^{\Texplore} \frac{1}{2} \normHx{\matx_{t+1} - \Acl \matx_t}^2 + \frac{\lambdast}{2} \normHS{\Acl}^2	\end{align*}
	\item Refine estimate: $A_0 = \Aclhat - B_0 \Kbase$.
\end{enumerate}
\end{enumerate}
\vspace*{2mm}
\vspace*{2mm}
\end{boxedminipage}
\end{figure}

\subsection{Incorporating Data Dependent Conditions for $\warmstart$ \label{app:data_dependent}}

Let $P_0 \defeq \Pinfty{A_0}{B_0}$ denote the optimal value function of the initial estimate $(A_0,B_0)$. From our perturbation bounds on the solution to the DARE (\Cref{prop:uniform_perturbation_bound}), we have that if $$\epsopnot \defeq \max\{\opnorm{A_0 - \Ast}, \opnorm{B_0 - \Bst}\} \le \eta/(16 (1+\eta)^4 \opnorm{\Pinfty{A_0}{B_0}}^3,$$ for some parameter $\eta\in (0,1)$, then $\|\Pst - P_0\|_{\op} \le \eta \|P_0\|_{\op}$. By choosing $\eta$ sufficiently small, we can see that that if $\epsopnot \le \frac{1}{c_1\|P_0\|_{\op}^3}$, then the warm start condition \Cref{cond:unif_close} holds. Hence, by modifying constants, we can replace $\Cstable$ in the warm-start condition with a data-dependent condition $\epsopnot \le \frac{1}{c_1\|P_0\|_{\op}^3}$ (for constant $c_1$), which depends only on the value function of the initial estimate. One can can show that this condition is guaranteed to be met as soon as $\epsopnot \le \frac{1}{c_2\|\Pst\|_{\op}^3}$, which means that the data-dependent warm start condition does not significantly alter what is required from the alignment condition.

\subsection{Implementation via Representer Theorems}

In the case where the states $\matx_t$ are infinite dimensional feature mappings equal to $\phi(\maty_t)$, where $\phi$ is a kernel and $\maty_t$ is a finite dimensional observation, we can employ standard representer theorem arguments for kernel ridge regression in order to efficiently implement the algorithms above. More specifically, the estimates $\Ahat$ and $\Bhat$ can be represented via outer products of the data points $(\matx_t, \matv_t)$. 

Furthermore, since inputs are finite dimensional, we can compute inverses and solve the Riccati equation by iterating the finite horizon version until convergence (see discussion in \citet{fazel2018global}, Appendix A). That is, for $P_1 = Q$,  \citet{hewer1971} shows that the following fixed point iteration is contractive and that $\Pinfty{A}{B}$ is equal to the limit of, 
\[
P_{t+1} = Q + A^\herm P_t A - A^\herm P_tB(R + B^\herm P_t B)^{-1} B^\herm P_t A.
\]

 Having solved the Riccati equation, we can then compute controller by taking products of linear operators, for which we have tractable representations.

\section{Regret Bounds \label{app:regret}}

For the sake of the analysis in this section, we define the quantity $\Regret_T(\calA, \matz)$ as the regret incurred by the algorithm $\calA$ over $T$ time steps starting from (a possibly random) initial state $\matz$.
\begin{lemma}
\label{lemma:main_regret_lemma}
Let $\calA$ be an algorithm that chooses actions according to $\matu_t = K \matx_t + \matv_t$ for $\matv_t \sim \calN(0, \uvar I)$ and let $\matz \sim \calN(0, \Sigma_\matz)$ with $\opnorm{\Sigma_\matz} \leq B_\matz$. Then, with probability $1-\delta$, $\Regret_T(\calA, \matz)$ is less than or equal to,
\begin{align*}
7\log_+(2/\delta) \left( T\left( \uvar \traceb{R} + \traceb{Q_0\dlyap{\Acl^\herm}{\covw + \Bst \Bst^\herm \uvar}} \right) + B_\matz \traceb{\PinftyK{K}{\Ast}{\Bst}}\right) - T J_\star,
\end{align*}
where $Q_0 \defeq Q + K^\herm R K$ and $\Acl \defeq \Ast + \Bst K$.
\end{lemma}

\begin{proof}
By definition, the regret of the algorithm is equal to:
\begin{align}
\label{eq:stochastic_regret}
\sum_{t=1}^{T} \inner{\matx_t}{(Q + K^\herm R K) \matx_t} + \inner{\matv_t}{R \matv_t} - T J_\star.
\end{align}
The lemma follows from first showing that the relevant random variables concentrate around their expectations and then upper bounding these expectations.

\paragraph{Bounding exploration cost} Using the Hanson-Wright inequality (\Cref{lemma:hansonwright}), we argue that with probability $1-\delta/2$, since $\matv_t \iidsim \calN(0, \uvar I)$,
\begin{align*}
	\sum_{t=1}^{T} \inner{\matv_t}{R \matv_t}^2 = \norm{R^{1/2}\matv_t}^2 \leq 7 \uvar T  \log_+(2/\delta) \traceb{R}.
\end{align*}
More precisely, we have applied Hanson-Wright to the series of random variables $\tilde{\matv}_t = R^{1/2}\matv_t$ and used the calculation, $\E \traceb{R \matv_t \otimes \matv_t } =\traceb{R} \uvar$.

\paragraph{Bounding state cost} Now let $Q_0 \defeq Q + K^\herm R K$, we have that,
\begin{align*}
	\sum_{t=1}^{T} \inner{\matx_t}{(Q + K^\herm R K) \matx_t} &= \sum_{t=1}^{T} \norm{Q_0^{1/2} \matx_t}^2 = \norm{\tilde{\matx}}^2, \\
\end{align*}
where $\tilde{\matx}$ is defined as, $\matz_0 \defeq \begin{bmatrix}
		Q_0^{1/2} \matx_1 $
		\dots $
		Q_0^{1/2} \matx_{T}
	\end{bmatrix}^\herm$. Since all the $\matx_t$ are Gaussian, $\tilde{\matx_t}$ is also Gaussian, albeit in a different Hilbert space. Applying Hanson-Wright, we get that with probability $1-\delta/2$,
\begin{align*}
	\norm{\tilde{\matx}}^2 & \leq 7 \log_+ (2/\delta) \E \traceb{ \tilde{\matx} \otimes \tilde{\matx}}.
\end{align*}

\paragraph{Bounding expectation} Letting $\Acl \defeq \Ast + \Bst K$ we have that by definition of the dynamical system, for $j \geq 0$,
\[
\matx_{1 + j} = \Acl^j \matx_{1} + \sum_{k=1}^j \Acl^{k-1} (\Bst \matv_{1 + j -k} +  \matw_{1 + j -k} ).
\]
Therefore,
\[
	\E \left[ \matx_{1 + j} \otimes \matx_{1 + j} \right] \preceq \Acl^{j} \E \left[ \matx_{1} \otimes \matx_{1} \right] \left(\Acl^{\; \herm}\right)^{j} + \dlyap{\Acl^\herm}{\covw + \Bst \Bst^\herm \uvar}.
\]
Using this, we can upper bound $\E \traceb{\tilde{\matx} \otimes \tilde{\matx} }$ as follows,
\begin{align}
	\E \traceb{\tilde{\matx} \otimes \tilde{\matx} } &= \sum_{t = 1}^T \E \traceb{Q_0 \matx_{t} \otimes \matx_t} \nn \\
	&\leq \traceb{ Q_0 \sum_{j=0}^\infty \Acl^j \E \left[\matx_{1} \otimes \matx_{1}\right] \left(\Acl^{\; \herm}\right)^j} + T\traceb{Q_0 \dlyap{\Acl^\herm}{\covw + \Bst \Bst^\herm \uvar}} \label{eq:second_line_regret}\\
	&\leq B_\matz \traceb{\PinftyK{K}{\Ast}{\Bst}} + T \traceb{Q_0\dlyap{(\Ast + \Bst K)^\herm}{\covw + \Bst \Bst^\herm \uvar}}.
\label{eq:third_line_regret}
\end{align}
The final inequality is justified by the following series of manipulations,
\begin{align*}
	\traceb{ Q_0 \sum_{j=0}^\infty \Acl^j \E \left[\matx_{1} \otimes \matx_{1}\right] \left(\Acl^{\; \herm}\right)^j} &\leq  \opnorm{\E \left[\matx_{1} \otimes \matx_{1}\right]}\traceb{ \sum_{j=0}^\infty  \left(\Acl^{\; \herm}\right)^j (Q + K^\herm P K )\Acl^j  } \\
	& = \opnorm{\E \left[\matx_{1} \otimes \matx_{1}\right]} \traceb{\PinftyK{K}{\Ast}{\Bst}}.
\end{align*}
\paragraph{Wrapping up} Combining our results so far, we have that with probability $1-\delta$ the regret as expressed in \Cref{eq:stochastic_regret} is less than or equal to,

\[
7\log_+(2/\delta) \left( T\left( \uvar \traceb{R} + \traceb{Q_0\dlyap{\Acl^\herm}{\covw + \Bst \Bst^\herm \uvar}} \right) + B_\matz \traceb{\PinftyK{K}{\Ast}{\Bst}}\right) - T J_\star.
\]
\end{proof}
\subsection{Proof of \Cref{corollary:regret_warm_start}}
\begin{proof}
The proof of this proposition consists of a simple application of \Cref{lemma:main_regret_lemma} and then bounding the relevant terms to show that the regret is $\mathcal{O}(\Tinit)$. Recall that $\warmstart$ chooses inputs according to $\matu_t = \Kbase \matx_t + \matv_t$. Furthermore, the initial state is exactly 0 so $\B_\matz = 0$ in the statement of the lemma. Therefore, with probability $1-\delta$, for $\Acl = \Ast + \Bst \Kbase$, the regret is smaller than:
\begin{align*}
7 \Tinit \log_+(2/\delta) \left(   \uvar \traceb{R} + \traceb{(Q + \Kbase^\herm R \Kbase)\dlyap{\Acl^\herm}{\covw + \Bst \Bst^\herm \uvar}} \right)  - \Tinit J_\star.
\end{align*}
We note that $\traceb{(Q + \Kbase^\herm R \Kbase)\dlyap{\Acl^\herm}{\covw + \Bst \Bst^\herm \uvar}}  \leq \opnorm{Q + \Kbase^\herm R \Kbase} \traceb{\Covst{\Kbase, \uvar}}$. Since $J_\star \geq 0$, this expression above is then upper bounded by,
\begin{align*}
7 \log_+(2/\delta) \left(   \uvar \traceb{R} + \opnorm{Q + \Kbase^\herm R \Kbase} \traceb{\Covst{\Kbase, \uvar}}\right) \Tinit.
\end{align*}
\end{proof}
\newpage
\subsection{Proof of \Cref{theorem:main_regret_theorem}}
\begin{proof}
By definition of the algorithm in \Cref{app:algorithm_descriptions}, we can split up the regret into two separate phases: an initial $\explore$ phase and then a $\commit$ phase.
\[
\Regret_{\Texplore}(\explore, \matx_1) + \Regret_{T - \Texplore}\left(\commit, \matx_{\Texplore + 1}\right).
\]
The $\explore$ phase corresponds to the regret incurred during the first part of the algorithm wherein inputs are chosen according to $\matu_t  = \Kinit \matx_t + \matv_t$ for $\Texplore$ many iterations. The commit algorithm then chooses inputs $\matu_t = \Khat \matx_t$ for the remaining $T - \Texplore$ time steps.
\paragraph{Exploration regret}We now upper bound the regret in each phase individually. For the exploration phase, we can apply \Cref{lemma:main_regret_lemma}, to get that with probability $1-\delta/4$, for $\Acl = \Ast +\Bst \Kinit$ and $Q_0 = Q +  \Kinit^\herm R \Kinit$, the regret experienced in this phase bounded by,
\begin{align*}
7\log_+(8/\delta)  \Texplore \left( \uvar \traceb{R} + \traceb{Q_0\dlyap{\Acl^\herm}{\covw + \Bst \Bst^\herm \uvar}} \right).
\end{align*}
Since $\matx_{1} \sim \mathcal{N}(0, \expcovone)$ is drawn from the same exploration distribution, the remaining constant term from \Cref{lemma:main_regret_lemma} vanishes since there is no need to consider the change in distribution. As in the proof of the $\warmstart$ regret bound, we observe that,
\begin{align*}
\traceb{(Q +  \Kinit^\herm R \Kinit)\dlyap{\Acl^\herm}{\covw + \Bst \Bst^\herm \uvar}} \leq \opnorm{Q +  \Kinit^\herm R \Kinit}\traceb{\expcovone}.
\end{align*}
Given that the initial estimates satisfy \Cref{cond:unif_close}, by \Cref{lem:Pk_bound} and \Cref{lemma:uniform_bound_p} we have that 
\[
\opnorm{\Kinit}^2 \leq \opnorm{\Pinfty{A_0}{B_0}} \lesssim \opnorm{\Pst}.
\] 
Hence,
\begin{align*}
\opnorm{Q + \Kinit^\herm R \Kinit} \leq \opnorm{Q} + \opnorm{\Kinit}^2 \opnorm{R} \lesssim  M_\star^2.
\end{align*}
Therefore, with probability $1 - \delta/4$,
\begin{align}
\label{eq:exploration_regret_bound}
\Regret_{\Texplore}(\explore, \matx_1) \lesssim \log_+(1/\delta) \Texplore\left( \uvar \traceb{R} + M_\star^2 \traceb{\expcovone} \right).
\end{align}
\paragraph{Bounding regret during commit phase} Moving on to bounding regret during the second phase, our first observation is that due to the projection step, the system estimates $\Ahat, \Bhat$ lie inside an operator norm ball of radius $.5 \Cstable$ around the warm start estimates $(A_0, B_0)$. Since the warm start estimates are themselves $.5\Cstable$ close to $(\Ast, \Bst)$, we conclude that $(\Ahat, \Bhat)$ satisfy \Cref{cond:unif_close} and by \Cref{lemma:uniform_bound_p}, $\Khat = \Kinfty{\Ahat}{\Bhat}$ is guaranteed to be stabilizing for the true system $(\Ast, \Bst)$. Furthermore, $\opnorm{P_0} \lesssim \Mst$ and $\opnormil{\Phat} \lesssim \Mst$.

Next, we use a similar regret analysis as before and apply \Cref{lemma:main_regret_lemma}, to conclude that with probability $1 - \delta / 4$, $\Regret_{T - \Texplore}\left(\commit, \matx_{\Texplore + 1}\right)$ is upper bounded by,
\begin{align}
\label{eq:initial_commit_regret}
7 T \log_+(8/\delta)  \left(\traceb{Q_0\dlyap{\Acl^\herm}{\covw}}  -  \Jstar\right) + 7\log_+(8/\delta)\opnorm{\expcovone} \traceb{\PinftyK{\Kinit}{\Ast}{\Bst}},
\end{align}
where $\Acl = \Ast + \Bst \Khat$ and $Q_0 = Q + \Khat^\herm R \Khat$. Next, since
\begin{align*}
	\traceb{Q_0\dlyap{\Acl^\herm}{\covw}} = \traceb{(Q + \Khat^\herm R \Khat) \Covst{\Khat}} = \Jfunc{\Khat},
\end{align*}
we can rewrite \Cref{eq:initial_commit_regret} as, 
\[
7 T \log_+(8/\delta) (J(\Khat) - J(\Kst))  + 7\log_+(8/\delta)\opnorm{\expcovone} \traceb{\PinftyK{\Kinit}{\Ast}{\Bst}}.
\]
Using \Cref{theorem:end_to_end_bound},
\[
\Jfunc{\Khat} - \Jfunc{\Kst} \lesssim \Mst^{36}  \cdot\Lfactor  \exp(\tfrac{1}{50}\sqrt{\mathcal{L}}) \cdot \epsilon^2 ,
	\quad \text{ where } \Lfactor \defeq \log\left(e + \tfrac{ 2e \|\Ahat - \Ast\|_{\op}^2\traceb{\expcovone}}{\epsilon^2} \right),
\]
where $\epsilon = \max\left \{ \normHS{(\Ahat - \Ast) \expcov}, \normHS{\Bhat - \Bst} \right\}$ and $\uvar$ has been set to 1 as in the description of the $\mainalg$ in \Cref{app:algorithm_descriptions}. To finish the proof, we now plug in our estimation rates from \Cref{part:learning} to upper bound $\epsilon$ and optimize over $\Texplore$.

\paragraph{Plugging in estimation rates} Using \Cref{proposition:B_estimation}, for $\Texplore \gtrsim \dimu \log\left(\frac{\dimu}{\delta}\right)$, with probability $1-\delta/4$,
\begin{align*}
\normHS{\Bst - \Bhat}^2 \lesssim  \frac{ \dimu \traceb{\expcovone} \log\left( \frac{\dimu}{\delta} \right) }{\uvar \Texplore}.
\end{align*}
Incorporating the analogous proposition for $\Ast$, \Cref{prop:A_estimation}, for
\[
\Texplore \gtrsim  d_\lambda \log_+\left(\frac{ \traceb{\expcovone}}{\delta \lambda}\right) +  \opnorm{\Pst}\log_+ \left(\frac{\opnorm{\Pst}^2}{\lambda}   \opnorm{\Bst \Bst^\herm \uvar + \covw} \right),
\]
with probability $1-\delta/4$
\[
\normHS{(\Aclhat - \Aclst) \expcov}^2 \lesssim \lambda \normHS{\Aclst}^2 + \cdot \frac{\logtracenorm}{\Texplore}(d_\lambda + \Ctail)\log_+ \left(  \frac{ \traceb{\expcovone}}{\delta \lambda}\right).
\]
Setting $\lambda = c \frac{\logtracenorm}{ \Texplore \normHS{\Aclst}^2}$ for some universal constant $c$, the second term dominates the first and we get that,
\[
\normHS{(\Aclhat - \Aclst) \expcov}^2 \lesssim \frac{\logtracenorm}{\Texplore}(d_\lambda + \Ctail)\log \left(\frac{ \traceb{\expcovone} \Texplore \normHS{\Aclst}^2}{\logtracenorm \delta}\right).
\]
From our definition of $\epsilon$, we see that it is upper bounded by, the sum of the errors in $\Ast$ and $\Bst$ and hence,
\[
\epsilon^2 \lesssim \frac{\dimu \traceb{\expcovone} + \logtracenorm (d_\lambda + \Ctail)}{\Texplore} \log \left(\frac{ \dimu \traceb{\expcovone} \normHS{\Aclst}^2 T}{\logtracenorm \delta^2}\right),
\]
where above we also upper bounded $\Texplore \le T$.
\paragraph{Wrapping up} All that remains is to optimize over $\Texplore$ to balance the regret between both phases. In particular, if we choose,
\begin{align*}
	\Texplore = \sqrt{\frac{T \cdot \Mst^{36} \left( \dimu \traceb{\expcovone} + \logtracenorm (d_\lambda + \Ctail) \right)}{\uvar \traceb{R} + M_\star^2 \traceb{\expcovone}}},
\end{align*}
we get that with probability $1-\delta$, up to constants and log factors, the total regret is bounded by:
\begin{align*}
\sqrt{(\traceb{R} + M_\star^2 \traceb{\expcovone}) \Mst^{36} \left( \dimu \traceb{\expcovone} + \logtracenorm (d_\lambda + \Ctail) \right)T} \cdot \varphi(T)
\end{align*}
where $\varphi(T) = \exp\left(\sqrt{\log\left(1 + \sqrt{T} \traceb{\expcovone}\right)}\right) $. Simplifying the bound a bit further, we know that $\traceb{R} \leq \opnorm{R}\dimu$. And, by \Cref{lemma:trace_dlyap_bound},
\begin{align*}
	\traceb{\expcovone} &= \traceb{\dlyap{(\Ast + \Bst \Kinit)^\herm}{\Bst \Bst^\herm \uvar+ \covw}} \\
	&\leq \opnorm{\dlyap{(\Ast + \Bst \Kinit)^\herm}{Q + \Kinit^\herm R \Kinit}} \traceb{\Bst \Bst^\herm \uvar+ \covw} \\
	& \leq \Mst (\Mst + \traceb{\covw}) \\
	& \lesssim \Mst^2 \traceb{\covw}.
\end{align*}
Hence, if we let $\dmax \defeq \max\{ \traceb{\covw}, \logtracenorm, \dimu \}$ we get that: 
\[
\traceb{R} + \Mst^2 \traceb{\expcovone} \lesssim \Mst^4 \dmax \text{ and } \dimu \traceb{\expcovone} + \logtracenorm(d_{\lambda} + \Ctail) \lesssim \Mst^2 \dmax^2.
\]
Therefore, we can upper bound the total regret of $\mainalg$ is with high probability $\mathcal{O}_\star \left( \sqrt{ \Mst^{42} \dmax^2(d_\lambda + \Ctail) T} \right)$.
\end{proof}

\subsection{Proof of \Cref{theorem:regret_with_decay_rates}}
\begin{proof}
The proof of the theorem follows by bounding $d_\lambda$ and $\Ctail$ based on the decay rates of $\covw$. We overload notation and define $d_\lambda(\Lambda)$ for $\Lambda \succeq 0 $ to be the number of eigenvalues of $\Lambda$ that are larger than $\lambda$. Likewise, we define $\Ctail(\Lambda)$ to be the sum of the eigenvalues of $\Lambda$ that are smaller than $\lambda$, divided by $\lambda$. The proof follows by applying  two inequalities which follow from \Cref{lemma:comparing_spectrum}. In particular, since $\Bst$ is finite rank and by the warm start property, the following are true for $n \gtrsim n_0 \defeq \opnorm{\Pst} \log \left( \opnorm{\Pst}^2 \logtracenorm / \lambda\right)$,
\begin{align*}
	d_{\lambda}(\expcovone) &\leq n_0 d_{\lambda/(2\|\Pst\|_{\op}^2)}(\covw) + \dimu \opnorm{\expcovone}\\
	\Ctail(\expcovone) & \leq \frac{1}{\lambda} n_0 \opnorm{\Pst}^2\left(\sum_{j \geq \ceil{a}}^\infty \sigma_j(\covw) + \lambda \right) + \dimu \opnorm{\expcovone},
\end{align*}
where $a = d_\lambda\left( \covw \right) / n_0$.
We analyze each case separately.
\paragraph{Polynomial decay} If $\sigma_j(\covw) = j^{-\alpha}$, then a short calculation shows that $d_\lambda(\covw) = \floor{\lambda^{-1/\alpha}}$. Therefore,
\[
d_{\lambda}(\expcovone) \leq n_0 \left( \frac{\lambda}{\opnorm{\Pst}^2} \right)^{-1/\alpha} + \dimu \opnorm{\expcovone}.
\]
Since $\lambda = c \cdot \frac{\logtracenorm}{T \normHS{\Aclst}^2}$ for some constant $c$, $d_\lambda(\expcovone)$ scales no faster than $n_0T^{1/\alpha} + \dimu \opnorm{\expcovone}$. For $\Ctail$, we have that
\[
\sum_{j \geq \ceil{\frac{\lambda^{-1/\alpha}}{n_0}}} j^{-\alpha} \leq \int_{\frac{\lambda^{-1/\alpha}}{n_0}} j^{-\alpha} = \frac{1}{\alpha -1} \lambda^{(\alpha - 1) / \alpha} n_0^{\alpha - 1}.
\]
Therefore, after dividing by $\lambda$, we get that $\Ctail$ scales as $n_0\opnorm{\Pst}^2 T^{1/\alpha} \normHS{\Aclst}^2 + \dimu \opnorm{\expcovone}$. Since $\opnorm{\expcovone} \lesssim M_\star^3$, we get that $d_\lambda + \Ctail $ are $\widetilde{\mathcal{O}}(M_\star^4 T^{1/\alpha})$.




\paragraph{Exponential decay} Moving on to the case where the eigenvalues of $\covw$ decay exponentially fast, a similar calculation to the previous one shows that $d_\lambda(\covw) = \floor{\alpha^{-1} \log(1/\lambda)}$. Therefore,
\begin{align*}
	d_\lambda(\expcovone) \lesssim \frac{n_0}{\alpha} \log
	\left( \frac{\opnorm{\Pst}^2}{\lambda}\right) + \dimu \opnorm{\expcovone}.
\end{align*}
Hence $d_\lambda(\expcovone)$ scales as $\widetilde{\mathcal{O}}(n_0 + \dimu M_\star^3)$. For the tail term, we observe that $\sum_{j=1}^\infty \exp(-\alpha j) = 1 / (\exp(\alpha) -1)$. This term therefore scales no faster than $n_0 \opnorm{\Pst}^2 + \dimu \opnorm{\expcovone}$. This shows that $d_\lambda + \Ctail$ are $\widetilde{\mathcal{O}}(M_\star^3 \dimu)$.

\paragraph{Finite dimension} In finite dimension with full rank noise, it is clear that for large enough $T$, $d_\lambda(\expcovone)$ is bounded by $\dimx$ and that $\Ctail$ is equal to 0. Furthermore, using now standard analysis such as the ones present in \cite{simchowitz2018learning}, it follows that $\normHSil{\Ahat - \Ast}$ goes to 0 at the same rate as $\normHSil{(\Ahat - \Ast)\expcov}$. Therefore, the terms depending on $\Lfactor$ in \Cref{theorem:main_regret_theorem} become $\mathcal{O}(1)$.

\end{proof}

\subsection{Combining WarmStart and OnlineCE \label{sec:further_remarks_regret}}
In the analysis of $\mainalg$ (\Cref{theorem:main_regret_theorem}), we assumed that $\matx_1 \sim \calN(0,\expcovone)$ was distributed according to the steady state distribution of induced by the exploration policy. This assumption can be relaxed, since any initial distribution over $\matx_1$ will converge exponentially quickly to the steady steady in Wasserstein distance due to a mixing argument (variants of this argument are ubiquitous in the analysis of online LQR, and for brevity we omit them. The curious reader can see appendices of \cite{dean2018regret,abeille2020efficient} for examples. The mixing time will be a polynomial in $\|\PinftyK{K_0}{\Ast}{\Bst}\|_{\op}$, which we show is $\lesssim \opnorm{\Pst}$.

Hence to stitch the two regret bounds together, we simply run the initial phase to garner estimates $(A_0,B_0)$, begin to play controller $K_0 = \Kinfty{A_0,B_0}$, allow a constant-length burnin for the state to converge to the distribution of $\expcovone$, and then execute $\mainalg$. Again, for the sake of brevity, we omit the details.

As a final remark, when stitching both algorithms together, one could in principle omit synthesizing the controller $\Kinit$ as outlined in the first step of the $\mainalg$ algorithm and run the entire exploration phase just using $\Kbase$. Doing so would only increase the constants $\Mst$ since they would now depend on $\opnorm{\PinftyK{\Kbase}{\Ast}{\Bst}}$. The asymptotics of the algorithm would remain unchanged. However, the projection step onto a safe set around $(A_0, B_0)$ is crucial for our analysis in order to ensure that the certainty equivalent controller is stabilizing for the true system.

\newcommand{\calB}{\mathcal{B}}
\section{Lower Bound \label{app:lower_bound}}
In this section, we state and prove lower bounds demonstrating the necessity of finite input dimension; these results follow from applications of the lower bound due to \cite{simchowitz2020naive}.  Our first bound is as follows:


\newtheorem*{thm:lb_formal}{\Cref{thm:main_lb_input}}
\begin{thm:lb_formal}
Let $c,c' > 0$ denote universal constants. Fix any trace bound $\gamma \ge 1$ and input dimension $\dimu \in \N$ with $\dimu \ge \sqrt{\log(1+\gamma)}$. Consider the set $\calU$ of instances with state dimension $\dimx = \floor{\gamma}$ defined by
\begin{align*}
 \calU := \{(A,B) := \|A - \frac{1}{2}I\|_{\HS}  \le \frac{1}{4}, \quad \|B\|_{\HS} \le \frac{1}{4} \}.
\end{align*}
Then, the LQR regret with cost matrices $Q = I_{\dimx}$, $R = I_{\dimu}$, and noise $\|\Sigmaw\|_{\op} = 1$, $\trace[\Sigmaw] \le \gamma$ satisfies 
\begin{align*}
\min_{{\Alg}}\max_{(A,B) \in \calU} \Exp_{A,{B}}[\mathrm{Regret}_T( \Alg)] \ge c \cdot \begin{cases} T & T \in [c' \gamma \log(1+\gamma), \gamma \dimu^2]\\
\sqrt{\gamma \dimu^2 \cdot T } & T \ge \gamma \dimu^2 
 \end{cases}.
\end{align*}
\end{thm:lb_formal}
We prove the bound in the following subsection. The bound considers instances that lie in a finite dimensional Hilbert space of dimension $\dimx = \floor{\gamma}$. In particular, all the instances in the packing are operator norm bounded, Hilbert-Schmidt, and in fact finite rank. The difficulty introduced by high-dimensional inputs is one of a needle in the haystack: to find the optimal control policy, the learner needs to learn to align their controller with the true $\Bst$ matrix, and doing so incurs dependence on ambient dimension. In essence, this is because the learner is free to pick any direction she chooses, so the complexity of the problem behaves more like, say, a linear bandit problem in dimension $\dimu$ than a statistical learning problem which admits a more refined notion of intrinsic dimension. 

In \Cref{sec:omitted_proofs}, we state and sketch the proof of a lower bound that holds \emph{even if} all the instances are controllable, demonstrating that little can be done to remove the finite dimensionality requirement of inputs. 
\begin{remark}[Local Mimimax Bounds] The bound stated above considers the minimax regret over a ball of constant radius. In contrast, the  lower bound for finite dimensional online LQR due to \cite{simchowitz2020naive} is stronger in the following ways: (a) it considers \emph{local minimax} regret over alternatives with a vanishingly small radius ($\mathcal{O}(1/\sqrt{T})$) of a nominal instance $(A_\star,B_\star)$ and (b) it establishes lower bounds for \emph{any} sufficiently nondegenerate $(A_\star,B_\star)$. 

Since our lower bound builds on theirs, we can in fact strengthen \Cref{thm:main_lb_input} to hold for a set $\mathcal{U}$ contained in a ball of radius $\mathcal{O}(1/\sqrt{T})$, thereby accomplishing point (a). Regarding point (b): we chose to state our lower bound for a fixed instance to clarify the dependence on dimension without requiring further dependence on problem parameters. In principle, similar guarantees can be established for perturbations of arbitrary instances $(A_\star,B_\star)$, but these would require additional work to clarify both (i) the norm of the value function for such an instance, and (ii) the singular values $\sigma_{m}(\Ast + \Bst \Kst)$ of the optimal solution to the closed loop matrix (in view of \citet[Theorem 1]{simchowitz2020naive}).  By addressing both terms, one could show, for example, that even in vanishly small neighborhoods of Hilbert-Schmidt, but still infinite dimensional $\Ast$, the $\dimu$-term in the regret is unavoidable. In the interest of brevity, we leave these calculations to future work.
\end{remark}

\subsection{Proof of \Cref{thm:main_lb_input}}
Throughout, we let $c_i, i > 1$ denote universal constants. Recall that $\gamma \ge 1$ is the trace bound, and $\dimx = \floor{\gamma}$ is the state dimension. We select $\Sigmaw = I_{\dimx}$, which has trace $\traceb{\Sigmaw} = \dimx \le \gamma$.

 Our lower bound follows from specializing the lower bound due to \citet{simchowitz2020naive}. To begin, we state a variant of their main lower bound.
\begin{proposition}[Variant of Theorem 1 in \citet{simchowitz2020naive}]\label{thm:lb_finite_dim}  Let  $c_1,c_2,p > 0$ denote universal constants. Consider a finite dimensional LQR system $(\Ast,\Bst)$, with finite input dimension $\dimu$, state dimension $\dimx$, cost matrices $R,Q \succeq I$, optimal controller $\Kst$, value function $\Pst$, and noise $\Sigmaw = I_{\dimx}$. Suppose $\nu := \sigma_{\min}(\Ast + \Bst \Kst)/\|R + \Bst^\top \Pst \Bst\|_{\op} > 0$. Then, defining the convex set,
\begin{align}
\calB \defeq \{(A,B) : \|A - \Ast\|_{\HS}^2 +\|B - \Bst\|_{\HS}^2\} \le \frac{1}{16},  \label{eq:AB_bound_lb}
\end{align}
it holds that,
\begin{align*}
\min_{\Alg}\max_{(A,B) \in \calB} \Exp_{A,B}[\mathrm{Regret}_T(\Alg)] \ge c_2 \sqrt{\dimx \dimu^2  T } \cdot \frac{\min\{1,\nu^2\}}{\|\Pst\|_{\op}^2},
\end{align*}
provided that $T \ge c_1\|\Pst\|^p_{\op}\min\{\dimu^2 \dimx, \frac{\dimx\max\{1,\nu^4\}\max\{1,\|\Bst\|_{\op}^4\}}{\dimu^2}, \dimx \log(1 + \dimx \|\Pst\|_{\op})\} $.
\end{proposition}
\begin{proof}

\Cref{thm:lb_finite_dim} is obtained by specializing the proof of the lower bound from Theorem 1 in \citet{simchowitz2020naive} to $m = \dimu$, and noting that the instances in the construction of the lower bound lie in an ball which satisfies the conditions of Lemma 4.1 their work (note that, by enlarging the constant term in the unspecified polynomial in that lemma, we can ensure that the constant is sufficiently small). 
\end{proof}
First, for simplicity, we specialize \Cref{thm:lb_finite_dim} with a concrete instance. Take $\Ast = \frac{1}{2} I$, $R = Q = I$ and let $\Bst = \mathbf{0}$.
\begin{lemma} Let $\Pst$ denote the value function for the instance $(\Ast,\Bst)$, $\Kst$ the optimal controller, and $\Aclst := \Ast + \Bst \Kst$ the optimal closed loop systems. Then, $\Pst = \frac{4}{3} I $, $\Kst = 0 $, and $\Aclst = \Ast = \frac{1}{2} I$.
\end{lemma}
\begin{proof} Since $\Bst = \mathbf{0}$, $\Aclst = \Ast$. Moreover, zero $\Bst$ means that the control inputs do not affect the system. Hence, optimal performance is optimized by selecting no control input, so as to minimize input cost; that is, $\Kst = 0$. Hence, the value function is $\dlyap{\Aclst}{Q} = \sum_{j \ge 0} (\frac{1}{2}I)^{j} I (\frac{1}{2})^j =I \cdot \sum_{j \ge 0} 4^{-j} = \frac{1}{1-1/4} I = \frac{4}{3}I$.
\end{proof}
Invoking \Cref{lem:specialized_riccati}, and using $\dimu^2 \ge \log(1+\gamma) \ge \log(1+\dimx)$ to simplify terms, we find that for universal (i.e. dimension independent) constants $c_3,c_4$, the instances centered in $\calB := \{(A,B) : \|A - \Ast\|_{\HS}^2 +\|B\|_{\HS}^2\} \le \frac{1}{4}\}$ satisfy 
\begin{align}
\min_{\Alg}\max_{(A,B) \in \calB} \Exp_{A,B}[\mathrm{Regret}_T(\Alg)] \ge c_3 \sqrt{\dimx \dimu^2 \cdot T } , \quad \forall T \ge c_4\dimu^2 \dimx  \label{eq:T_lb}.
\end{align}

To conclude, we address the case where $T \le c_4 \dimu^2 \dimx$. For $T \ge 4c_4 \dimx \log(1+\dimx)$, consider a smaller input dimension $d := \sqrt{\floor{T/c_4\dimx}}$. Let $\tilde{\calB}_d := \{(A,\tilde B) \in \R^{\dimx^2} \times \R^{\dimx d} : \|A - \Ast\|_{\HS}^2 +\|\tilde{B}\|_{\HS}^2\} \le \frac{1}{4}\}$ denote the analogous set of local instances to the above construction with input dimension $d$. From the choice of $d$ and condition on $T$, we have $d^2 \ge \log(1+\dimx)$ and $T \ge c_4 d^2 \dimu$, so the above lower bound established in dimension $\dimu$ entails that, for another universal constant $c_5$,
\begin{align*}
\min_{\tilde{\Alg}}\max_{(A,\tilde{B}) \in \tilde{\calB}_d} \Exp_{A,\tilde{B}}[\mathrm{Regret}_T(\tilde \Alg)] \ge c_3 \sqrt{\dimx d^2 \cdot T } \ge c_5 T,
\end{align*}
where the last inequality follows from the choice of $d := \sqrt{\floor{T/c_4\dimx}}$. Stated otherwise, it holds that for all $T \in [4 c_4 \dimx \log(1+\dimx), c_4\dimx \dimu^2]$, 
\begin{align*}
\exists \; d \in [1,\dimu] \text{ such that } \min_{\tilde{\Alg}}\max_{(A,\tilde{B}) \in \tilde{\calB}_d} \Exp_{A,\tilde{B}}[\mathrm{Regret}_T(\tilde \Alg)] \ge c_5 T.
\end{align*}
We now embed the above instances of input dimension $d$ into input dimension $\dimu$ via the following lemma.

\begin{lemma}\label{lem:redux_lemma} Fix a state dimension $\dimx$, and an input dimension $d \le \dimu$. Given a matrix $\tilde{B} \in \R^{\dimx \times d}$, let $\iota(\tilde{B})$ denote its canonical embedding into $\R^{\dimx \times \dimu}$ by padding the remaining $\dimu - d$ columns with zeros. Overloading notation, given a subset of $ \tilde{\calU} \subset \R^{\dimx^2} \times\R^{\dimx \dimu}$, define its embedding $$\iota(\tilde{\calU}) := \{(A,B): ~ \exists (A,\tilde{B}) \in \tilde{\calU} \text{ with } B = \iota(\tilde{B})\}.$$ 
Then,
\begin{align*}
\inf_{\Alg}\max_{(A,B) \in \calU} \Exp_{A,B,}[\mathrm{Regret}_T(\Alg)] \ge \inf_{\tilde{\Alg}}\max_{(A,\tilde{B}) \in \tilde{\calU}} \Exp_{A,\tilde{B}}[\mathrm{Regret}_T(\tilde\Alg)], 
\end{align*}
where on both sides, the noise covariance is $\Sigmaw = I$ and the state cost $Q = I_{\dimx}$. On the left hand side, the input cost is $I_{\dimu}$, and on the right, $R = I_{d}$. 
\end{lemma}
Thus, letting $\bar{\calB} := \bigcup_{d \in [1,\dimu]} \iota(\tilde{\calB}_d)$ denote the union of all the embeddings of the sets $\bar{\calB}$, it holds that for all $T \in [4 c_4 \dimx \log(1+\dimx), c_4\dimx \dimu^2]$, 
\begin{align*}
\min_{{\Alg}}\max_{((A,B) \in \bar{\calB}} \Exp_{A,{B}}[\mathrm{Regret}_T( \Alg)] \ge c_5 T.
\end{align*}
Moreover, since $\bar{\calB} \supset \calB$, incorporating \Cref{eq:T_lb} we conclude that, 
\begin{align*}
\min_{{\Alg}}\max_{((A,B) \in \bar{\calB}} \Exp_{A,{B}}[\mathrm{Regret}_T( \Alg)] \ge \begin{cases} c_5 T & T \in [4 c_4 \dimx \log(1+\dimx), c_4\dimx \dimu^2]\\
c_3 \sqrt{\dimx \dimu^2 \cdot T } & T \ge c_4\dimx \dimu^2 
 \end{cases}.
\end{align*}
Replacing $\dimx$ with $\floor{\gamma}$, absorbing constants, and simplifying, we find that for universal constants $c_6,c_7$:
\begin{align*}
\min_{{\Alg}}\max_{((A,B) \in \bar{\calB}} \Exp_{A,{B}}[\mathrm{Regret}_T( \Alg)] \ge c_6 \begin{cases} T & T \in [c_7 \gamma \log(1+\gamma), \gamma \dimu^2]\\
\sqrt{\gamma \dimu^2 \cdot T } & T \ge \gamma \dimu^2 
 \end{cases}.
\end{align*}
Finally, we observe that
\begin{align*}
\bar{\calB} \subset \calU := \{(A,B) := \|A - \frac{1}{2}I\|_{\HS}  \le \frac{1}{4}, \quad \|B\|_{\HS} \le \frac{1}{4} \}.
\end{align*}

\begin{proof}[Proof of \Cref{lem:redux_lemma}] We begin with the following claim.

\begin{claim}\label{claim_dimension_embedding} Let $\Alg$ be an algorithm which interacts with instances $(A,B) \in \calU$. Then, there is an algorithm $\tilde{\Alg}$ which interacts with instances $(A,\tilde B) \in \tilde \calU$ for which,
\begin{align*}
\forall A,\tilde B, B = \iota (\tilde B), \quad \Exp_{A,\tilde B, \tilde \Alg} [\sum_{t=1}^T \|\matx_t\|^2 + \|\matu_t\|^2] \le \Exp_{A,B, \Alg} [\sum_{t=1}^T \|\matx_t\|^2 + \|\matu_t\|^2],
\end{align*}
with equality if $\Alg$ always plays inputs $\matu_t$ whose last $\dimu - d$ coordinates are $0$. 
\end{claim}
\begin{proof}
Given an input $\matu_t \in \R^{\dimu}$, write $\matu_t = (\tilde \matu_t, \mathring{\matu}_t)$ as the decomposition of $\matu_t$ into its first $d$, and last $\dimu - d$ coordinates. Observe that, for instances $(A,B) \in \calU$, the last $\dimu - d$ coordinates $\mathring{\matu}_t$ \emph{do not} affect the dynamics. Hence, the iterates $(\matx_t,\tilde{\matu}_t)$ produced by $\Alg$ on $(A,B)$ in $\calU$ coincide with the iterates obtained by the  algorithm $\tilde{\Alg}$ which, given an instances $(A,\tilde B) \in \tilde \calU$ proceeds as follows: 
\begin{itemize}
\item $\tilde{\Alg}$ maintains ``internal'' inputs $\bar{\matu}_t$ corresponding to the inputs that would have been selected by the original algorithm $\Alg$ with input dimension $\dimu$.
\item For each $t$, $\tilde{\Alg}$ feeds $\Alg$ the past iterates $\matx_{1:t},\bar{\matu}_{1:t-1}$, and receive internal input $\bar{\matu}_{t}$.
\item Then, $\tilde{\Alg}$ plays the input $\matu_t \in \R^{d}$ obtained by projecting $\bar{\matu}_t$ onto its first $d$ coordinates. 
\end{itemize}
Given an instance $(A,\tilde{B}) \in \tilde U$, and its embedding $(A,B) = (A,\iota(\tilde{B}))$,  the iterates $(\matx_t,\tilde{\matu}_t)$ produced by $\Alg$ on $(A,B)$ have the same distribution as the iterates $(\matx_t,\matu_t) \in \R^{\dimx} \times \R^{d}$ produced by $\tilde{\Alg}$. Hence, for all $ A,\tilde B, B = \iota (\tilde B)$, it holds that
\begin{align*}
\Exp_{A,\tilde B, \tilde \Alg} [\sum_{t=1}^T \|\matx_t\|^2 + \|\matu_t\|^2] &= \Exp_{A,B, \Alg} [\sum_{t=1}^T \|\matx_t\|^2 + \|\tilde\matu_t\|^2] \\
&\le\Exp_{A,B, \Alg} [\sum_{t=1}^T \|\matx_t\|^2 + \|\matu_t\|^2],
\end{align*}
with equality if remaining $\dimu - d$ coordinates of the inputs prescribed by $\Alg$ are identically $0$ for all $t$.
\end{proof}
Arguing along the lines of \Cref{claim_dimension_embedding}, we can also see that the optimal infinite horizon control policy for $(A,B) \in \calU$ also only selects inputs supported on the first $d$ coordinates, and thus 
\begin{align*}
J^{\star}_{A,B} = J^{\star}_{A,\tilde B}, \quad \forall (A,\tilde B) \in \calU, B = \iota(\tilde B).
\end{align*}
Consequently,
\begin{align*}
\max_{(A,B) \in \calU} \Exp_{A,B}[\mathrm{Regret}_T(\Alg)] &= \max_{(A,B) \in \calU} \Exp_{A,B,\Alg}[\sum_{t=1}^T\|\matx_t\|^2 +\|\matu_t\|^2] - TJ^{\star}_{A,B}\\
&\ge \max_{(A,B) \in \calU} \Exp_{A,\tilde{B},\tilde \Alg}[\sum_{t=1}^T\|\matx_t\|^2 +\|\matu_t\|^2] - TJ^{\star}_{A,\tilde{B}}, \quad \text{ where } B = \iota(\tilde B)\\
&\ge\max_{(A,B) \in \tilde\calU} \Exp_{A,\tilde B}[\mathrm{Regret}_T(\tilde\Alg)],
\end{align*}
as needed.
\end{proof}

\subsection{Lower Bound that Maintains Controllability\label{sec:omitted_proofs}}
In this section, we state a lower bound that maintains controllability, in order to demonstrate how   controllability does not ameliorate the requirement that the input dimension $\dimu$ be bounded. To capture this scenario, set 
\begin{align}
\dimx \le \dimu, \quad \Ast = \frac{1}{2}I, \quad \Bst = \begin{bmatrix} I_{\dimx} & 0_{\dimx} \end{bmatrix} \label{eq:bst}.
\end{align}

\begin{theorem}\label{thm:lb_controllable} Let $c,c'$ be universal constants.  Let $c,c' > 0$ denote universal constants. Fix any trace bound $\gamma \ge 1$ and input dimension $\dimu \in \N$ with $\dimu \ge \gamma$. Consider the set $\calU$ of instances with state dimension $\dimx = \floor{\gamma}$ defined by
\begin{align*}
 \calU := \left\{(A,B) := \|A - \frac{1}{2}I\|_{\HS}  \le \frac{1}{4}, \quad \|B - \Bst\|_{\HS} \le \frac{1}{4} \right\}, \quad \text{where } \Bst \text{ is in \Cref{eq:bst}.}
\end{align*}
Then, the LQR regret with cost matrices $Q = I_{\dimx}$, $R = I_{\dimu}$, $\|\Sigmaw\|_{\op} = 1$, and $\trace[\Sigmaw] \le \gamma$, satisfies,
\begin{align*}
\min_{{\Alg}}\max_{(A,B) \in \calU} \Exp_{A,{B}}[\mathrm{Regret}_T( \Alg)] \ge \sqrt{\gamma \dimu^2 \cdot T },
\end{align*}
for all $T \ge c' \gamma \dimu^2 $. In particular, if $T \propto \gamma \dimu^2 $, the minimax regret on $\calU$ is linear in $T$.
\end{theorem}
Note that, for all instances $(A,B) \in \calU$, not only is $A$ stable ($\|A\|_{\op} \le 3/4$), but the column space has rank $\dimx$, and smallest singular value at least $3/4$ (since the first $\dimx$ columns of $\Bst$ are the identity, and all matrices are in $\calU$ are a bounded perturbation thereof). Hence, the systems $(A,B) \in \calU$ are all one-step controllable. Nevertheless, the regret still scales with $\dimu^2$.

The proof of \Cref{thm:lb_controllable} is nearly the same as that of \Cref{thm:main_lb_input}; the main difference is verifying the bounds on $\Pst$ and $\sigma_{\min}$ required to instantiate \Cref{thm:lb_finite_dim}.
\begin{lemma}\label{lem:specialized_riccati} Regardless of the choice of dimension, we have that $\|\Pst\|_{\op} \le 4/3$, and $\sigma_{\min}(\Ast + \Bst \Kst) \ge 1/5$.
\end{lemma}
\begin{proof}[Proof of \Cref{lem:specialized_riccati}] For simplicity, we drop the  stars in the subscript. Let us characterize the optimal solution. Define $\calF(P) := A^\top P A - (A^\top P B)(R+B^\top P B)^{-1}(B^\top P A) + Q - P$. Let  $J$ denote the projection onto the subspace spanned by the columns of $B$, and $J_{\perp} = I - J$. We guess a solution of the form
\begin{align*}
P = p_1 J + p_2 J_{\perp}.
\end{align*}
Next, we show that such a $P$ solves $\calF(P) = 0$ when $p_1,p_2$ are appropriately chosen.  With this matrix,  $R+B^\top P B = R + p_1 I_{\dimu} = (1+p) I_{\dimu}$, so $B(R+B^\top P B)^{-1}B^\top = (1+p_1)J$, where $J$ is the projection onto the subspace spanned by the columns of $B$. Let $J^{\perp} = I - J$. Note that in the overactuated case, $J$ is the identity. Hence, 
\begin{align*}
\calF(P) &= \frac{P}{4} - \frac{p_1^2}{4(1+p_1)} J - P + I\\
&= (1 + \frac{p_1}{4} - p_1  - \frac{p_1^2}{4(1+p_1)}) J + (1 + \frac{p_2}{4} - p_2 )J_{\perp} \\
&= (1 - \frac{ 3 p_1}{4} -  \frac{p_1^2}{4(1+p_1)}) J + (1 + \frac{-3p_2}{4}  )J_{\perp}.
\end{align*}
To solve this equation, set $p_2 = \frac{4}{3}$, and set 
\begin{align*}
&1 - \frac{ 3 p_1}{4} -  \frac{p_1^2}{4(1+p_1)} = 0\\
&-4(1+p_1) + 3p_1(1+p_1) + p_1^2 = 0\\
&4p_1^2 - p_1 - 4 = 0.
\end{align*}
Taking the positive solution of the quadratic equation, $p_1 = \frac{1 + \sqrt{65}}{8} \le 4/3$. 

Now, the optimal control policy is $K = -(R + B^\top P B)^{-1}B^\top P A = \frac{1}{2} \cdot \frac{p_1}{1+p_1}B^\top$ (using the form of $B$, $R = I$, and $P$), yielding $BK = - \frac{1}{2} \cdot (\frac{p_1}{1+p_2})J$. Hence, $A + BK = \frac{1}{2}\left(J_{\perp} +  \frac{1}{1+p_2} J\right)$, and thus, $\sigma_{\min}(A+BK) \ge \frac{1}{2(1+p_2)} \ge 1/5$.
\end{proof}

\end{document}